\documentclass[reqno, 12pt]{amsart}
\usepackage{amsfonts,amsmath,amssymb,color}

\numberwithin{equation}{section}

\usepackage[kerning=true]{microtype}

\usepackage[top = 4cm, left = 3.8cm, right = 3.8cm]{geometry}

\usepackage{amsthm,leqno}
\usepackage{hyperref}
\usepackage{bigints}
\usepackage{dsfont}

\usepackage{amsfonts,amsmath,amssymb,enumerate}
\usepackage{graphicx}
\usepackage{tikz}
\usetikzlibrary{decorations.markings}
\usetikzlibrary{plotmarks}
\usetikzlibrary{patterns}

\newtheorem{theo}{Theorem}[section]
\newtheorem{prop}[theo]{Proposition}

\newtheorem{lemme}[theo]{Lemma}

\newtheorem{claim}[theo]{Claim}

\theoremstyle{remark}

\renewcommand{\(}{\left(}
\renewcommand{\)}{\right)}

\newcommand{\be}{\begin{equation*}}
\newcommand{\ee}{\end{equation*}}
\newcommand{\ben}{\begin{equation}}
\newcommand{\een}{\end{equation}}
\newcommand{\begincal}{\begin{eqnarray*}}
\newcommand{\fincal}{\end{eqnarray*}}
\newcommand{\bal}{\begin{aligned}}
\newcommand{\eal}{\end{aligned}}

\newcommand{\pui}{\frac{n-2}{2}}

\newcommand{\hpv}{\hat{\vp}_{\ve} }

\newcommand{\RR}{\mathbb{R}}

\newcommand{\ve}{\varepsilon}
\newcommand{\vp}{\varphi}

\numberwithin{equation}{section}

\title[]{Towers of bubbles for Yamabe-type equations and for the Br\'ezis-Nirenberg problem in dimensions $n \ge 7$} %

\author{Bruno Premoselli}

\address{Bruno Premoselli, Universit\'e Libre de Bruxelles, Service de G\'eom\'etrie Diff\'erentielle CP 218, Boulevard du Triomphe, B-1050 Bruxelles, Belgique.}
\email{bruno.premoselli@ulb.ac.be \\ }

\begin{document}

\begin{abstract}
Let $(M,g)$ be a closed locally conformally flat Riemannian manifold of dimension $n \ge 7$ and of positive Yamabe type. If $\xi_0$ denotes a non-degenerate critical point of the mass function we prove the existence, for any $ k \ge 1$ and $\ve >0$,  of a \emph{positive} blowing-up solution $u_\ve$ of
$$\triangle_g u_\ve +\big( c_n S_g  +\ve h\big) u_\ve = u_\ve^{2^*-1},$$
that blows up like the superposition of $k$ positive bubbles concentrating at different speeds at $\xi_0$. The method of proof combines a finite-dimensional reduction with the sharp pointwise analysis of solutions of a linear problem. As another application of this method of proof we construct \emph{sign-changing} blowing-up solutions $u_\ve$ for the Br\'ezis-Nirenberg problem
$$ \triangle_{\xi} u_\ve - \ve u_\ve = |u_\ve|^{\frac{4}{n-2}} u_\ve \textrm{ in } \Omega, \quad u_\ve = 0 \textrm{ on } \partial \Omega $$
on a smooth bounded open set $\Omega \subset \RR^n$, $n \ge 7$, that look like the superposition of $k$ positive bubbles of alternating sign.
\end{abstract}

\maketitle

\section{Introduction and statement of the results}

\subsection{Positive towers of bubbles in the Riemannian case}

Let $\(M,g\)$ be a closed smooth connected Riemannian manifold, that is compact without boundary, of dimension $n\ge3$. Assume $(M,g)$ is of positive Yamabe type.  The main goal of this work is to construct blowing-up families of positive solutions $(u_\ve)_\ve$ of critical nonlinear elliptic equations of the following form:
\begin{equation} \label{eqZ}
\triangle_g u_{\ve} + h_\ve u_{\ve} = u_{\ve}^{2^*-1},
\end{equation}
where $\triangle_g = - \textrm{div}_g(\nabla \cdot)$, $2^* = \frac{2n}{n-2}$ and where  $h_\ve$ converges in $C^1(M)$ as $\ve \to 0$. We construct sequences that display \emph{towering} bubbling phenomena when $h_\ve \to c_nS_g$, where $c_n = \frac{n-2}{4(n-1)}$ and where $S_g$ is the scalar curvature of $(M,g)$.  

\medskip

When $h_\ve \equiv c_n S_g$ equation \eqref{eqZ} is the celebrated Yamabe equation. The existence of a positive solution in this case was proven through the work of Yamabe \cite{yamabe}, Aubin \cite{aubin}, Trudinger \cite{Trudinger} and Schoen \cite{SchoenYamabe}. The question of the compactness of the set of positive solutions of the Yamabe equation on manifolds of positive Yamabe type has generated a vast amount of work in the last decades. On the standard sphere $(\mathbb S^{n},g_0)$, $n \ge 3$, the set of positive solutions is not compact, but if $(M,g)$ is not conformally diffeomorphic to the standard sphere, the set of positive solutions of the Yamabe equation is compact when the dimension of the manifold satisfies $3\le n\le 24$, as was shown by Khuri, Marques and Schoen \cite{KhuriMarquesSchoen} (previous  results were obtained Li and  Zhu \cite{LiZhu}, Druet \cite{DruetYlowdim}, Marques \cite{Marques} and Li and  Zhang \cite{LiZhang}). In dimensions $n \ge 25$ Brendle \cite{Brendle} and Brendle and Marques \cite{BreMa} constructed examples
of manifolds for which the set of positive solutions is not compact. From the PDE point of view, understanding the compactness of solutions to the Yamabe equation (and, more generally, of larger classes of nonlinear critical elliptic equations) relies on establishing sharp blow-up estimates near blow-up points. A crucial step in the proof of the compactness of the Yamabe equation consists in showing that, in dimensions $3 \le n \le 24$, all the possible blow-up points of a sequence of positive solutions must be {\em isolated} and {\em simple} (following the terminology of Schoen \cite{SchoenPreprint}) i.e. that around each blow-up point $\xi_0$ the solution can be approximated by a so called {\it standard bubble }
 \ben \label{stdbubble}
 W_k(x)\sim { \frac{\mu_k^{\pui}}{\(\mu_k^2+\frac{d_g(x,\xi_k)^2}{n(n-2)} \)^{\pui}}} \ \hbox{
 for some $\xi_k\in M$ and $\mu_k\to0.$}
 \een

\medskip

For more general equations like \eqref{eqZ} one may wonder if compactness properties  -- or, more generally, \emph{stability} properties according to the terminology in Hebey \cite{HebeyZLAM} -- remain true for positive solutions when $h_\ve$ in \eqref{eqZ} is not everywhere equal to $c_nS_g$. There the picture becomes more complicated. If $n=3$, no \emph{positive} blowing-up positive solutions exist as long as the mass of the limiting operator is positive as proved by  Li-Zhu  \cite{LiZhu}. If we denote by $h_\infty$ the $C^1$ limit of $h_\ve$ and if $(M,g)$ is not conformally diffeomorphic to the standard sphere, the mass of $\triangle_g + h_\infty$ remains positive even if $h_\infty > c_3 S_g$ everywhere, provided $h_\infty$ is not too large. When $n \ge 4$, Druet \cite{DruetYlowdim} proved that \eqref{eqZ} does not have any positive blowing-up solution when $h_{\infty} < c_n S_g$ everywhere, or when $n=4,5$ and $h_\ve \le c_n S_g$ for all $\ve$ and the manifold is not conformally equivalent to the round sphere. 

\medskip

If $h_\ve > c_n S_g$ somewhere, the situation is subtler and blow-up is likely to occur when $n \ge 4$. We focus in this paper in the following special case of \eqref{eqZ}: 
\begin{equation} \label{eqY}
\triangle_g u +\big( c_n S_g  +\ve h\big) u = u^{2^*-1},
\end{equation}
where $h \in C^0(M)$ and $\ve >0$. If $n\ge4$, positive single-bubble blowing-up solutions or multi-bubble solutions with isolated concentration points do exist  as shown by Esposito, Pistoia and Vetois in \cite{EspositoPistoiaVetois}. Pistoia and Vaira proved in \cite{PistoiaVaira}, when $h \equiv 1$ and $n \ge7$, the existence of blowing-up solutions of \eqref{eqY} possessing a so-called  {\it clustering} blow-up point at a non-degenerate non-zero minimum point of the Weyl's tensor. In these so-called \emph{clustering} configurations the solutions blow-up as a finite number of peaks, of comparable heights, whose centers converge to the same point. Clustering configurations in dimensions $4$ and $5$ have also been constructed by Thizy and V\'etois \cite{ThizyVetois} and, for families of equations like \eqref{eqZ}, in dimensions $n \ge 6$ by Robert and V\'etois \cite{RobertVetois2}. 

\medskip

We address here the question of existence of so-called \emph{towering bubbling configurations} for positive solutions of \eqref{eqY}. These are finite-energy blowing-up positive solutions that blow-up as a superposition of positive bubbles concentrating at different speed, and whose centers converge to the same point -- these configurations are also sometimes referred to as \emph{towers of bubbles}. We prove  that \emph{towering} blow-up points for \eqref{eqY} can be constructed in dimension $n \ge 7$ on locally conformally flat manifolds. Our main result is as follows:
\begin{theo} \label{maintheo}
Let $(M,g)$ be a closed locally conformally flat Riemannian manifold of dimension $n \ge 7$ of positive Yamabe type. Let $\xi_0$ denote a non-degenerate critical point of the mass function $\xi \mapsto m(\xi)$.  Let $k \ge 1$. Then there exists a family $(u_\ve)_\ve$ of positive solutions of
\[ \triangle_g u_\ve + (c_n S_g + \ve ) u_\ve = u_\ve^{2^*-1} \]
 that blows-up with $k$ bubbles at $\xi_0$, and for which $\xi_0$ is a \emph{towering} blow-up point.
\end{theo}
For the definition of the Riemannian mass we refer to \eqref{defmass} below. Morabito, Pistoia and Vaira in \cite{mpv} first proved the existence of positive towering bubbling configurations on non-locally conformally flat manifolds, at a point $\xi_0$ where the Weyl's tensor does not vanish, under the stringent assumption that the manifold is symmetric with respect to $\xi_0.$ In the present work we treat the locally conformally flat case that was left open, and we perform the construction without any symmetry assumption on $M$: this considerably increases the complexity of the proof and requires us to develop a new analytical approach. Towering bubble configurations for positive solutions similar to the ones constructed in this paper do not exist in dimensions $n \le 6$, as can be seen by an adaptation of the arguments in Druet \cite{DruetJDG} (see also Hebey \cite{HebeyZLAM}, Chapter 8). In low dimensions $n =4,5$, the closest analogue to a positive towers of bubbles is made of isolated peaks concentrating at different speeds and has been built in Premoselli-Thizy  \cite{PremoselliThizy}. Theorem \ref{maintheo} thus completes the constructive picture of the possible blow-up scenarios for positive solutions of \eqref{eqY}.

\medskip

We stated Theorem \ref{maintheo} when $h \equiv 1$ for simplicity. If one allows a more general function $h \in C^1(M)$ (compare with \eqref{eqY}) it is still possible to prove the existence of towering blow-up points. The nondegeneracy assumption on $\xi_0$, in particular, can be lifted if one chooses $h$ suitably (see e.g. the end of the proof of Theorem \ref{maintheo} at the very end of Section \ref{expansionkernel} where this topic is dicussed).

\subsection{Sign-changing towers of bubbles for the Br\'ezis-Nirenberg problem}

The method we use to prove Theorem \ref{maintheo} also applies to the construction of \emph{sign-changing} solutions to the classical Br\'ezis-Nirenberg problem \cite{bn}
\begin{equation}\label{brni}
\left\{\begin{aligned}&\triangle_\xi u-\varepsilon u=|u|^{\frac{4}{n-2}}u\ &&
  \hbox{in}\ \Omega,\\  &u=0 \  && \hbox{on}\ \partial\Omega\end{aligned}\right.\end{equation}
 where $\xi$ is the Euclidean metric, $ \triangle_\xi  = - \sum_{i=1}^n \partial_i^2$,  $\Omega\subset\mathbb R^n,$ $n\ge 3$ is an open and bounded smooth domain. It is well  known that \eqref{brni} possesses positive blowing-up solutions, when $\epsilon>0$ is small enough and $n\ge4$, blowing-up at distinct isolated points (see Rey \cite{r1} and Musso and Pistoia \cite{mupi}). In the sign-changing case and when $n \ge 7$, Iacopetti and Vaira \cite{iv} built \emph{sign-changing} solutions to problem \eqref{brni} whose shape looks like the superposition of bubbles with alternating sign  centered at the origin, provided the domain $\Omega$ is symmetric with respect to the origin. This is in sharp contrast with the \emph{positive} case where positive towers of bubbles, at least on symmetric domains, cannot occur for  $\epsilon>0$ small enough as was proven by Cerqueti  \cite{ce} using ideas of Li \cite{LiSn1, LiSn2}. We also mention Iacopetti and Vaira \cite{IacopettiVaira2} where sign-changing solutions with asymptotically vanishing negative part have been constructed.

 \medskip 
 
 The analysis that we develop in this paper to prove Theorem \ref{maintheo} adapts to the Euclidean case of the Br\'ezis-Nirenberg problem. Without assuming any symmetry in the domain we are able to prove the existence of sign-changing towering bubbling configurations:
 
\begin{theo} \label{theoBN}
Let $\Omega$ be a smooth bounded subset of $\RR^n$, $n \ge 7$, and let $\xi_0$ denote a non-degenerate critical point of the Robin function $\xi \mapsto H(\xi,\xi)$. Let $k \ge 1$. There exists a family $(u_\ve)_\ve$ of sign-changing solutions of
\ben \label{eqh1}
 \triangle_{\xi} u_\ve - \ve u_\ve = |u_\ve|^{\frac{4}{n-2}} u_{\ve} 
 \een
 that blows-up as the superposition of $k$ bubbles with alternating sign at $\xi_0$. 
 \end{theo}

The existence of radial towers of bubbles centered at the origin on the ball had been previously obtained by Esposito-Ghoussoub-Pistoia-Vaira \cite{egpv} in the more general context of Hardy-Schr\"odinger equations. Theorem \ref{theoBN} is again a generalisation of this construction to arbitrary non-symmetric domains. For the precise definition of the Robin function we refer to \eqref{defRobin}. As for Theorem \ref{maintheo} the assumptions on $\xi_0$ can be relaxed if a general $h$ is taken into account in the operator $\triangle_\xi - \ve h$, see the arguments at the end of Section \ref{BrezisNirenberg}. The existence of sign-changing towers of bubbles for problem \eqref{brni} is likely to be true only in dimensions $n \ge 7$, and a proof that sign-changing towers of bubbles made of two bubbles cannot appear in dimensions $3$ to $6$ can be found in Ben Ayed-El Medhi-Pacella \cite{bep, bep2} and Iacopetti and Pacella \cite{IacopettiPacella}. Let us also mention an asymptotic description of low-energy sign-changing solutions in dimensions $n \ge 7$ in Iacopetti \cite{iacopetti}.

    \subsection{Strategy of proof and outline of the paper}

 The proofs of Theorems \ref{maintheo} and \ref{theoBN} follow the strategy that was introduced in Premoselli \cite{Premoselli6, Premoselli7} and applied in Premoselli-Thizy \cite{PremoselliThizy}. The blowing-up family $(u_\ve)_\ve$ that we construct  in Theorem \ref{maintheo} has the following form:
 $$ u_{\ve} = \sum_{i=1}^k W_{i,\ve} + \vp_\ve, $$
 where $W_{i,\ve}$ are, at first-order, Riemannian standard bubbles of respective parameters $\mu_{k, \ve}  << \mu_{k-1, \ve} << \dots << \mu_{1,\ve}$ and centered at points $\xi_{i,\ve}$ and $\vp_\ve$ is a remainder. The points $\xi_\ell$ satisfy $d_g(\xi_{\ell-1}, \xi_\ell) = O(\mu_{\ell-1})$ for $2 \le \ell \le k$, they all converge to the same point $\xi_0 \in M$, and the parameters are related according to \eqref{growthmuell} below. The main idea of the present paper is to construct a remainder $\vp_\ve$ satisfiying
\ben \label{estremainder}
\left \Vert \frac{\vp_\ve}{\sum_{i=1}^k W_{i,\ve}} \right \Vert_{L^\infty(M)} = o(1)
\een
(more precise estimates are actually required, see \eqref{estoptvp}). This contrasts with the approach followed in Morabito-Pistoia-Vaira \cite{mpv}: in the highly symmetric setting of \cite{mpv} all the bubbles were centered in the center of symmetry of the manifold and a suitable $\vp_\ve$ was constructed in $H^1$ by splitting it in $k$ distinct remainders with a triangular dependence on the bubble parameters. Since we work here on a general non-symmetric background, where the centers of the bubbles also have to be chosen, the energy estimates of \cite{mpv} fail to capture precisely enough the interactions between the bubbles $W_{i,\ve}$. We thus have to construct a remainder $\vp_\ve$ which comes with sharp global pointwise estimates such as \eqref{estremainder}.

\medskip

Constructing $\vp_\ve$ that satisfies \eqref{estremainder} is the main challenge of this paper. It requires to prove sharp pointwise estimates for solutions of a linear equation (see Proposition \ref{proplin} below). These estimates are obtained by means of a bubble-tree-\emph{like} analysis, which is carried on at a significant precision. Once this linear theory is available, $\vp_\ve$ is constructed \emph{via} a standard fixed-point procedure, which now takes place in strong weighted spaces. This construction of $\vp_\ve$ applies both to positive and sign-changing solutions. The sharp pointwise estimates that we obtain on $\vp_\ve$, combined with the finite-dimensional reduction, also provide another way to obtain \emph{a priori} pointwise blow-up estimates for finite-energy solutions of \eqref{eqY}, similar to the ones in Druet-Hebey-Robert \cite{DruetHebeyRobert} and to the symmetry estimates of Li-Zhang \cite{LiZhang} and Khuri-Marques-Schoen \cite{KhuriMarquesSchoen} (see Section \ref{nonlin}).

\medskip

The outline of the paper is as follows. In Section \ref{notations} we introduce the notations and the bubbling profiles used in the construction. In Section \ref{lin} we prove sharp pointwise estimates for solutions of the linearization of \eqref{eqY}. This is achieved in Proposition \ref{proplin} and is the core of the analysis of the paper. In section \ref{nonlin} we perform the nonlinear construction to  solve \eqref{eqZ} up to kernel elements. In section \ref{expansionkernel} we perform an asymptotic expansion of the coefficients of the kernel elements and conclude the proof of Theorem \ref{maintheo}. Theorem \ref{theoBN} is then proven in  \ref{BrezisNirenberg}. Finally, the Appendix gathers a few technical results used throughout the paper. Although Theorem \ref{maintheo} is proven in the locally conformally flat case, we prove the results in sections \ref{lin} and \ref{nonlin} in the general case, since the analysis carries on and might be applied in future works.  We also take into account a general $h$ as in \eqref{eqY}  since it does not create additional difficulties.

\medskip
\noindent {\bf Acknowledgments:} The author warmly thanks A. Pistoia for many fruitful and motivating discussions during the preparation of this work, and in particular for pointing out the possible application to the Br\'ezis-Nirenberg problem, which led to Theorem \ref{theoBN}. The author was supported by a FNRS CdR grant J.0135.19 and by the Fonds Th\'elam.

\section{Notations and Approximate solutions}  \label{notations}

\subsection{The ansatz.} 

Let $(M,g)$ be a compact Riemannian manifold without boundary of dimension $n \ge 3$. The celebrated conformal normal coordinates result of Lee-Parker \cite{LeeParker} states the existence of a smooth function $\Lambda: M \times M \to (0, + \infty)$ suth that for any $\xi \in M$, the conformal metric $g_{\xi} = \Lambda_{\xi}^{\frac{4}{n-2}}g$ satisfies:
$$ S_{g_{\xi}}(\xi) = 0, \quad \nabla S_{g_{\xi}}(\xi) = 0 \quad \textrm{ and } \triangle_{g_{\xi}} S_{g_{\xi}}(\xi) = \frac16 |W_g(\xi)|_g^2, $$
where we have let $\Lambda_{\xi} = \Lambda(\xi, \cdot)$ and where $S_h$ and $W_h$ denote respectively the scalar curvature and the Weyl tensor of a metric $h$. If $(M,g)$ is locally conformally flat in a neighbourhood of some $\xi \in M$, $g_{\xi}$ is flat in a geodesic ball centered at $\xi$. If $h$ is any metric of positive Yamabe type on $M$ we denote by $G_h$ the Green function of the conformal laplacian $\triangle_h + c_n S_h$. Let $\xi \in M$ and consider the metric $g_{\xi} = \Lambda_{\xi}^{\frac{4}{n-2}} g$. If $(M,g)$ is locally conformally flat in a neighbourhood of $\xi$ the Green function $G_{g_\xi}$ expands as
\ben \label{defmass}
G_{g_{\xi}} \big( \xi, \exp_{\xi}^{g_{\xi}} (y) \big) = \frac{1}{(n-2) \omega_{n-1}} |y|^{2-n} + m(\xi) + O(|y|),
\een
for $|y|$ small, where we denoted by  $\exp^{g_{\xi}}_{\xi}$ the exponential map at $\xi$ for $g_{\xi}$. The quantity $m(\xi)$, which is a smooth function of $\xi$, is called the \emph{mass} at $\xi$. If $(M,g)$ is not conformally flat in a neighbourhood of $\xi \in M$ we have instead, when $n \ge 7$:
\ben \label{nonlcf}
 G_{g_{\xi}} \big( \xi, \exp_{\xi}^{g_{\xi}} (y) \big) = \frac{1}{(n-2) \omega_{n-1}} |y|^{2-n} + O \big( |y|^{6-n} \big). 
 \een
We refer to Lee-Parker \cite{LeeParker} for a proof of these facts.

\medskip

The bubbling profiles we will use in this work  are inspired from the profiles in Esposito-Pistoia-V\'etois \cite{EspositoPistoiaVetois}, with a suitable modification to smoothen them.   Let $0 < r_0  < \frac12 i_g(M)$ be such that
\ben \label{defr0}
2 r_0 <  \frac12 i_{g_{\xi}}(M) \quad \textrm{ for any } \xi \in M 
 \een
and let $\chi: \RR_+ \to \RR_+ $ be a smooth compactly supported function satisfying $\chi(r) = r$ for $r \le r_0$ and $\chi(r) = 2r_0$ for $r \ge 2 r_0$. If $\xi \in M$ and $\mu >0$, we define the bubble  $W_{\mu, \xi}$ as:
\ben \label{defW}
\begin{aligned}
W_{\mu, \xi}(x) &  = \Phi(\xi,x)   \frac{ \mu^{\pui}}{ \Big( \mu^2 + \frac{\chi \big( d_{g_{\xi}}(\xi,x)^2\big)}{n(n-2)}\Big)^{\pui}},
\end{aligned}
\een
where $d_{g_{\xi}}$ is the geodesic distance for the metric $g_{\xi}$ and where we have let, for all $\xi,x \in M$:
\ben \label{defPhi}
\Phi(\xi, x) = (n-2) \omega_{n-1} G_{g_{\xi}}(\xi, x) \Lambda_{\xi}(x) \chi \big( d_{g_{\xi}}(\xi, x)\big)^{n-2}.
\een
This ansatz satisfies in particular:
\be
\bal
W_{\mu, \xi} (x) = \Phi(\xi,x)
 \mu^{\pui} \Bigg( \mu^2 + \frac{d_{g_{\xi}}(\xi,x)^2}{n(n-2)}\Bigg)^{1 - \frac{n}{2}} & \textrm{ if }  d_{g_{\xi}}(\xi, x) \le r_0. \\
\eal
\ee
As shown in \cite{EspositoPistoiaVetois}, $W_{\mu, \xi}$ satisfies:
\ben \label{lapbulle}
\bal
\triangle_g W_{\mu, \xi}  + c_n S_g &  W_{\mu, \xi} - W_{\mu, \xi}^{2^*-1} =  
&
 \left \{ 
 \bal &  \frac{O\big( \mu^{\frac{n+2}{2}} \big)}{\big( \mu^2 + \frac{d_{g_{\xi}}(\xi,x)^2}{n(n-2)} \big)^2}  & (M,g) \textrm{ is l.c.f}\\
& \frac{O\big( \mu^{\frac{n+2}{2}} \big)}{\big( \mu^2 + \frac{d_{g_{\xi}}(\xi,x)^2}{n(n-2)} \big)^\pui} & \textrm{ otherwise },\\
\eal \right.
\eal
\een
where ``l.c.f'' means here \emph{locally conformally flat in a a neighbourhood of $\xi$}. We will use the same abbreviation throughout the paper, and there should be no confusion since our analysis will take place in a neighbourhood of a point $\xi_1 \in M$. In \eqref{lapbulle} we also used the following notation: 
$$ A = O(B) \textrm{ or, equivalently, } A \lesssim B $$
if there exists a constant $C >0$ independent of $\ve$ such that $|A| \le C |B|$. Similarly, we will write $A = o(B)$ if $|A||B|^{-1}$ goes to $0$ as $\ve \to 0$.

\subsection{The parameters.} We introduce the parameters of our construction. Let $A>1$ be a large positive constant, $k \ge 1$ be an integer.
Let $(t_i)_{1 \le i \le k} \in [A^{-1}, A]^k$ and $\ve  > 0$. For any $1 \le \ell \le k$ define:
$$ \mu_{\ell}  = \mu_{\ell, \ve}(t_{\ell}) =   \ve^{\gamma_{\ell}} t_{\ell}, $$
where
\ben \label{defmu}
  \gamma_{\ell} = \left \{ 
 \bal
&  \frac{n-2}{2(n-4)} \left( \frac{n-2}{n-6} \right)^{\ell-1} - \frac12 & (M,g) \textrm{ is l.c.f} \\
& \left( \frac{n-2}{n-6} \right)^{\ell-1} - \frac12 & \textrm{ otherwise }
 \eal \right. 
 \een
With this definition it is easily seen that we have 
$$ \ve^{1 + 2 \gamma_\ell} = \ve^{\pui(\gamma_\ell-\gamma_{\ell-1})}. $$
In particular, for fixed values of $(t_i)_{1 \le i \le k} \in [A^{-1}, A]^k$, the $\mu_{\ell}$ satisfy:
 \ben \label{growthmuell}
 \mu_1 \sim  \left \{ \bal
& \ve^{\frac{1}{n-4}} & (M,g) \textrm{ l.c.f} \\
& \ve^{\frac12}  & \textrm{ otherwise }
 \eal \right. 
 \quad \textrm{ and } \quad \ve \mu_{\ell+1}^2 \sim \left( \frac{\mu_{\ell+1}}{\mu_{\ell}} \right)^{\frac{n-2}{2}}
 \een
for any $1 \le \ell \le k-1$, where $\sim$ denotes the order of growth in $\ve$ as $\ve \to 0$. 

\medskip

Denote by $B(0,1)$ the open unit ball in $\RR^n$ and define
\ben \label{defA}
S_k =  [A^{-1}, A] \times M \times  \Big(  [A^{-1}, A] \times B(0,1) \Big)^{k-1}.
\een
Let $\big(t_1, \xi_1, (t_i, z_i)_{2 \le i \le k} \big) \in S_k$. For any $\ve >0$, we define the points $(\xi_2, \dots, \xi_k)$ by:
\ben \label{defxi}
 \textrm{for any } 2 \le \ell \le k, \quad \xi_{\ell}  = \exp^{g_{\xi_{\ell-1}}}_{\xi_{\ell-1}} \big( \mu_{\ell-1} z_{\ell} \big) .
\een
Recall that for $\xi \in M$, $\exp^{g_{\xi}}_{\xi}$ is the exponential map at $\xi$ for the metric $g_{\xi} = \Lambda_{\xi}^{\frac{4}{n-2}}$. In \eqref{defxi} we omitted the dependence on the parameters for the sake of clarity but keep in mind that we have:
$$ \xi_2 = \xi_{2,\ve}(t_1, z_2) \quad, \quad \dots \quad, \quad \xi_{\ell} = \xi_{\ell, \ve}(t_{\ell-1}, z_{\ell}). $$
For $1 \le \ell \le k $ we will let 
\ben \label{defWi}
 W_{\ell}(x) = W_{\mu_{\ell}, \xi_{\ell}}(x)
 \een
 where $W_{\mu_{\ell},\xi_{\ell}}$ is as in \eqref{defW}. By definition $W_{\ell}$ only depends on the parameters $(t_{\ell-1}, t_{\ell}, z_{\ell})$ when $\ell \ge 2$, and only on $(t_1, \xi_1)$ when $\ell=1$.  In this setting $W_k$ denotes the highest bubble (the one that concentrates faster) while $W_1$ is the lowest one. In the course of the proof we will need to analyse the mutual pointwise interactions between these bubbles. Following the notion of radius of influence introduced in Druet \cite{DruetJDG} (see also Hebey \cite{HebeyZLAM}, Chapter 8) We define for this the following sets:
 \ben  \label{defBi}
 B_1 = M , B_{k+1} = \emptyset \quad \textrm{ and } \quad  B_{\ell} = B_{g_{\xi_{\ell}}} \big(\xi_\ell, \sqrt{\mu_\ell \mu_{\ell-1}} \big) \textrm{ for } 2 \le \ell \le k,
 \een
 where $B_{g_{\xi_\ell}}$ denotes the geodesic ball for the metric $g_{\xi_\ell}$. For any family $\big(t_{1,\ve}, \xi_{1,\ve}, (t_{i,\ve}, z_{i,\ve})_{2 \le i \le k} \big)$ of elements of $S_k$ it is easily checked by \eqref{defxi} that, for $\ve$ small enough, the balls $B_\ell$ are nested:
 $$ B_k \subset B_{k-1} \subset \dots \subset B_1 .$$
It is also easily checked that, for any $1 \le \ell \le k$,
\ben \label{rapportsbulles}
 W_\ell \lesssim W_{\ell-1} \textrm{ in } B_{\ell-1} \backslash B_\ell \textrm{ and }  W_{\ell-1} \lesssim W_{\ell} \textrm{ in } B_\ell. 
 \een
In particular, for any $1 \le \ell \le k$ there holds
$$\sum_{j=1}^k W_j \lesssim W_\ell \quad \textrm{ on } B_\ell \backslash B_{\ell+1}. $$

\subsection{Kernel elements.}

Let $h \in C^0(M)$ and $\ve >0$. We will work with the following $H^1$ scalar product in $M$:
\ben  \label{psH1}
\langle u, v \rangle = \int_M \Bigg( \langle \nabla u, \nabla v \rangle + (c_n S_g + \ve h) uv \Bigg) dv_g.
\een
Let $1  \le \ell \le k$, $ 0 \le j \le n$ and let $\big(t_1, \xi_1, (t_i, z_i)_{2 \le i \le k} \big) \in S_k$. The approximate kernel elements associated to the $\ell$-th bubble on the manifold are defined by:
\ben \label{defZi}
\begin{aligned}
Z_{\ell,0}(x) & = Z_{\ell,0,\mu_\ell,\xi_\ell}(x)  \\
& =\Lambda_{\xi_\ell}(x) \chi \big( d_{g_{\xi_{\ell}}}(\xi_{\ell}, x)\big) \\
&\times  \mu_\ell^{\frac{n-2}{2}} \left( \frac{d_{g_{\xi_\ell}}(\xi_\ell,x)^2}{n(n-2)} - \mu_\ell^2 \right)  \left(  \mu_\ell^2 +   \frac{d_{g_{\xi_\ell}}(\xi_\ell,x)^2}{n(n-2)} \right)^{- \frac{n}{2}} , \\
 Z_{\ell,j}(x) & = Z_{\ell,j,\mu_\ell,\xi_\ell}(x)    \\
 & =
 \Lambda_{\xi_\ell}(x)  \chi \big( d_{g_{\xi_{\ell}}}(\xi_{\ell}, x)\big)  \\
 & \times \mu_\ell^{\frac{n}{2}}  \left \langle \left( \textrm{exp}_{\xi_\ell}^{g_{\xi_\ell}}\right)^{-1}(x), e_j(\xi_\ell) \right\rangle_{g_{\xi_\ell}(\xi_\ell)}\left(  \mu_\ell^2 + \frac{d_{g_{\xi_\ell}}(\xi_\ell,x)^2}{n(n-2)} \right)^{- \frac{n}{2}}. \\
\end{aligned}
\een
In the previous expression $(e_j)$ is a local orthonormal basis defined in $B_g(\xi_1, 2r_0)$. It is always well-defined and depends smoothly on the choice of the point since we can always choose a covering of parallel-type on $M$, see \cite{RobertVetois}. Let also
\ben \label{defB0}
U_0(x) = \left( 1 + \frac{|x|^2}{n(n-2)} \right)^{1 - \frac{n}{2}} \quad \textrm{ for any } x \in \RR^n
\een
and, for $1 \le j \le n$ and $x \in \RR^n$,
\ben \label{defVi}
 V_0(x) = \frac{\frac{|x|^2}{n(n-2)} -1}{\left( 1 + \frac{|x|^2}{n(n-2)}\right)^{\frac{n}{2}}} \quad \textrm{ and }  \quad V_j(x) = x_j \left( 1 + \frac{|x|^2}{n(n-2)}\right)^{-\frac{n}{2}}\quad .
 \een
For $1 \le \ell \le k$ we will let, for $x \in M$:
\ben \label{defti}
\theta_\ell(x) = \mu_\ell + d_{g_{\xi_\ell}}(\xi_\ell, x).
\een 
For $1 \le \ell \le k$ the functions $Z_{\ell,j}$ in \eqref{defZi} satisfy the following equations in $M$:
\ben \label{lapZ0}
\bal
\triangle_g  Z_{\ell,0}+ c_n S_g Z_{\ell,0} = &(2^*-1) W_\ell^{2^*-2} Z_{\ell,0} +
 \left \{ 
 \bal &  O(\mu_\ell^{\pui})  & (M,g) \textrm{ is l.c.f}\\
& O\big( \mu_\ell^2 W_{\ell} \big)  & \textrm{ otherwise },\\
\eal \right.
\eal 
\een
and, for $1 \le j \le n$,
\ben \label{lapZi}
\triangle_g  Z_{\ell,j}+ c_n S_g Z_{\ell,j} = (2^*-1) W_\ell^{2^*-2} Z_{\ell,j} +  \left \{ 
 \bal &  O( \mu_\ell^{\frac{n}{2}} ) & \textrm{ if } (M,g) \textrm{ is l.c.f}\\
&  O \big( \mu_\ell W_{\ell} \big)  & \textrm{ otherwise },\\
\eal \right.
\een
Explicit computations yield, for $1 \le i < j \le k$ and $0 \le p,q  \le n$:
\ben \label{psZi}
\langle Z_{i,p}, Z_{i,q} \rangle = \Vert \nabla V_p \Vert_{L^2(\RR^n)}^2 \delta_{p q} + O( \mu_i^2 ) 
\een
and 
\ben \label{psZik}
\bal
\langle Z_{i,p}, Z_{j,q} \rangle =  & O \left(  \left(\frac{\mu_j}{\mu_i}\right)^{\pui} \right).
\eal 
\een
The proofs of \eqref{lapZ0} and \eqref{lapZi} can be found in \cite{EspositoPistoiaVetois} (Lemma $6.3$) while we refer to the Appendix of this paper for the proof of \eqref{psZi} and \eqref{psZik}. We finally let, for $1 \le \ell \le k$:
\ben \label{defKi}
K_{\ell} = \textrm{Span} \big\{ Z_{i,j}, 1 \le i \le \ell, 0 \le j \le n \big\},
\een
and we denote by $K_{\ell}^{\perp}$ its orthogonal with respect to the scalar product \eqref{psH1}.

\section{Sharp pointwise asymptotics for the linear problem}  \label{lin}

\subsection{Statement of the result}

Let $k \ge 1$ be an integer and, for any $0 < \ve \le 1$, let $\big(t_{1,\ve}, \xi_{1,\ve}, (t_{i,\ve}, z_{i,\ve})_{2 \le i \le k} \big) \in S_k$ be a family of parameters, where $S_k$ is as in \eqref{defA}. For the sake of clarity we will drop the dependence in $\ve$ and, until the end of the paper,  we will always denote by $W_1, \dots, W_k$ the bubbles obtained as in \eqref{defWi} with $\mu_i$ and $\xi_i$ given by \eqref{defmu} and \eqref{defxi} for this choice of the family $(t_{1,\ve}, \xi_{1,\ve}), (t_{i,\ve}, z_{i,\ve})_{2 \le i \le k}$. In what follows the notation $f(u)$ will either denote $f(u) = |u|^{2^*-2}u$ or $f(u) = (u_+)^{2^*-1}$, where $u_+ = \max(u,0)$. We recall, as a starting point, the well-known linear theory in $H^1(M)$ for \eqref{eqY}:

\begin{prop} \label{proplinbase}
Let $k \ge 1$ be an integer. We fix $(\kappa_i)_{1 \le i \le n} \in \{-1,1 \}^k$. There exists $\ve_0 > 0$ such that the following holds: for any $\ve \in (0, \ve_0)$ and $\big(t_{1,\ve}, \xi_{1,\ve}, (t_{i,\ve}, z_{i,\ve})_{2 \le i \le k} \big)\in S_k$, $h \in C^0(M)$ and for any $k_{\ve} \in L^{\frac{2n}{n+2}}(M)$, there exists a unique $\vp_{\ve} \in K_k^{\perp}$ satisfying:
\ben \label{eqlinbase}
\Pi_{K_k^{\perp}}\Big( \vp_{\ve} - (\triangle_g + c_n S_g + \ve h)^{-1} \Big[    (2^*-1) f'\big(  \sum_{i=1}^k \kappa_i  W_i \big) \vp_{\ve} + k_{\ve}  \Big]\Big) = 0.
\een
Here $\Pi_{K_k^{\perp}}$ denotes the orthogonal projection on $K_k^{\perp}$ for \eqref{psH1}, and $K_k$ is as in \eqref{defKi}. This function satisfies in addition 
\ben \label{estH1lin}
 \Vert \vp_{\ve} \Vert_{H^1(M)} \lesssim \Vert  k_{\ve} \Vert_{L^{\frac{2n}{n+2}}}.
 \een
\end{prop}

\begin{proof}
This is a by now very classical result. We refer for instance to Robert-V\'etois \cite{RobertVetois}, Esposito-Pistoia-V\'etois \cite{EspositoPistoiaVetois} or Musso-Pistoia \cite{MussoPistoia2} for a proof. 
\end{proof}

Equation \eqref{eqlinbase} implies the existence of real numbers $\lambda_{ij} =\lambda_{ij, \ve} $, with $1 \le i \le k$ and $0 \le j \le n$, such that 
\ben \label{eqLvp}
\bal
 \triangle_g \vp_{\ve}  + (c_n S_g + \ve h) \vp_{\ve} - & f'\Big(  \sum_{i=1}^k \kappa_i W_i \Big) \vp_{\ve} = k_{\ve} \\
 & + \sum \limits_{\underset{j =0 .. n }{i=1..k}} \lambda_{ij} ( \triangle_g   + c_n S_g + \ve h)Z_{i,j}, 
 \eal
 \een
where $Z_{i,j}$ is as in \eqref{defZi}. By standard elliptic theory, if $k_{\ve} \in C^0(M)$, $\vp_{ \ve} \in C^1(M)$. Let $h \in C^0(M)$ be fixed and assume that $\ve >0$ is  small enough so that  $\triangle_g   + c_n S_g + \ve h$ is coercive. For any fixed $\ve$ and for any $k_\ve \in C^0(M)$
we denote the unique function $\vp_{\ve}$ given by Proposition \ref{proplinbase} by
\ben \label{defL}
\vp_{\ve} := \mathcal{L}_{\ve}(k_\ve).
\een 
We investigate in this section equation \eqref{eqLvp} in more detail. Assuming that a suitable pointwise control holds on $k_{\ve}$ we prove refined pointwise estimates on $\vp_{\ve}$. The pointwise control we will assume on $k_{\ve}$ is the following:
 \ben \label{hypf}
\bal
|k_{\ve}|  \lesssim \sum_{i=1}^k &  \left \{ 
\bal 
& \mu_i^{\frac{n+2}{2}} \theta_i^{-4}   &   \textrm{ If } (M,g) \textrm{ is } & \textrm{ l.c.f. around } \xi_1 \\
& \ve W_i  & \textrm{ If } (M,g) \textrm{ is } & \textrm{ not l.c.f. }
\eal \right \} \\
& +  \left \{ 
\bal 
& W_k^{2^*-2} W_{k-1} &  \textrm{ in } B_k \\  
& \vdots \\
&W_{i}^{2^*-2} \big( W_{i+1} + W_{i-1} \big) & \textrm{ in } B_i \backslash B_{i+1}, \\    \\ 
& \vdots  \\ 
& W_1^{2^*-2} W_2 & \textrm{ in } B_1  \backslash B_2  = M \backslash B_2 
 \eal \right \} ,
 \eal
 \een 
 where in \eqref{hypf} the index $i$ in the right-hand side runs from $2$ to $k-1$ and where $\theta_i$ is defined in \eqref{defti}. It is easily seen, using \eqref{lapbulle} and \eqref{rapportsbulles}, that \eqref{hypf} is satisfied for the following particular choice of $k_{\ve}$:
 \ben \label{ferr}
  k_{\ve} = \big(\triangle_g + c_n S_g + \ve h\big) \Big( \sum_{i=1}^k \kappa_i W_i \Big) - f \Big( \sum_{i=1}^k \kappa_i W_i \Big), 
  \een 
  where we choose $f(u) = (u_+)^{2^*-1}$ if all the $\kappa_i$ are equal to $1$ and  $f(u) = |u|^{2^*-2}u$ otherwise. 
 
 \medskip
 
 The following proposition is the main result of this section and the core of the analysis of this paper. 

\begin{prop} \label{proplin}
Let $k \ge 1$ be an integer, $(\kappa_i)_{1 \le i \le n} \in \{-1,1 \}^k$ and let $\ve_0 >0$ be as in Proposition \ref{proplinbase}. For any $\ve \in (0, \ve_0)$ let $\big(t_{1,\ve}, \xi_{1,\ve}, (t_{i,\ve}, z_{i,\ve})_{2 \le i \le k} \big)\in S_k$,  $k_{\ve} \in C^0(M)$ and let $\vp_{\ve} \in K_k^{\perp}$ be the unique solution of 
\ben \label{eqlin}
\bal
 \triangle_g \vp_{\ve}  + (c_n S_g + \ve h) \vp_{\ve} -  & f'\big(  \sum_{i=1}^k \kappa_i W_i \big) \vp_{\ve} \\
 & = k_{\ve} + \sum \limits_{\underset{j =0 .. n }{i=1..k}}    \lambda_{ij} ( \triangle_g   + c_n S_g + \ve h)Z_{i,j}, 
\eal
 \een
 where the $\lambda_{ij}$ are uniquely defined by \eqref{eqlinbase} and $f(u) = |u|^{2^*-2}u$ or $f(u) = (u_+)^{2^*-1}$. Assume that $k_{\ve}$ satisfies \eqref{hypf}. Then:
\begin{itemize}
\item For any $1 \le \ell \le k$,
\ben \label{vpl}
  \sum_{j =0 .. n }  |  \lambda_{{\ell} j}| \lesssim \ve \mu_\ell^2. 
  \een

\item For any $2 \le \ell \le k$, and any $x \in B_{\ell} \backslash B_{\ell+1}$:
\ben \label{estvpi}
\bal
 |\vp_{\ve}(x)| & \lesssim \mu_{\ell-1}^{1 - \frac{n}{2}} \left( \frac{\mu_{\ell}}{\theta_{\ell}(x)} \right)^2 + \ve \mu_{\ell-1}^{3 - \frac{n}{2}}  + \frac{\theta_{\ell+1}(x)^2}{\mu_{\ell}^2} W_{\ell+1}(x) \mathds{1}_{\{ \theta_{\ell+1}(x) \le \mu_{\ell}\}}, 
  \eal
\een
where the notation $\mathds{1}_{\{ \theta_{\ell+1}(x) \le \mu_{\ell}\}}$ indicates that this term vanishes when $\theta_{\ell+1}(x) > \mu_{\ell}$, and with the convention that $B_{k+1} = \emptyset$ and $W_{k+1} = 0$.
\item For any $x \in B_1 \backslash B_2 = M \backslash B_2$:
\ben \label{estvp1}
\bal
|\vp_{\ve}(x)|  \lesssim &  \left \{ \bal
&  \frac{\mu_{1}^{\frac{n+2}{2}}}{\theta_1(x)^2} & \textrm{ if } (M,g) \textrm{ is l.c.f. around } \xi_1, \\
& \frac{\mu_1^{\frac{n+2}{2}}}{\theta_1(x)^{n-4}} &  \textrm{ if } (M,g) \textrm{ is not l.c.f.}  \\
\eal \right \}  \\
& + \frac{\theta_{2}(x)^2}{\mu_{1}^2} W_{2}(x) \mathds{1}_{\{ \theta_{2}(x) \le \mu_{1}\}} .
\eal
\een
 \end{itemize}
The constants appearing in \eqref{vpl}, \eqref{estvpi}, \eqref{estvp1} are independent of $\ve$ and of the choices of $\big( t_1, \xi_1, (t_{i,\ve}, z_{i,\ve})_{2 \le i \le k} \big) \in S_k$, $U_1, \dots, U_k$ and $k_{\ve}$.
\end{prop}
In Proposition \ref{proplin}, $B_\ell$ is as in \eqref{defBi} and $\theta_\ell$ as in \eqref{defti}.

\medskip

Let us  comment on this result. Proposition \ref{proplin} yields a pointwise control on $\vp_{\ve}$ \emph{and} on the $\lambda_{ij}$ assuming only \eqref{hypf}: no previous control on $\vp_{\ve}$ or on the $\lambda_{ij}$ was assumed. The estimates \eqref{estvpi} and \eqref{estvp1} have a deeper meaning. Let indeed $G_{\ve}$ be the Green's function of $\triangle_g + c_n S_g + \ve h$ in $M$. Assuming that $k_\ve$ satisfies \eqref{hypf}, straightforward computations using \eqref{tech1} and \eqref{tech2} give:
\[  \int_M G_{\ve}(x,\cdot) |k_{\ve}| dv_g \lesssim \mu_{\ell-1}^{1 - \frac{n}{2}} \left( \frac{\mu_{\ell}}{\theta_{\ell}(x)} \right)^2  + \frac{\theta_{\ell+1}(x)^2}{\mu_{\ell}^2} W_{\ell+1}(x) \mathds{1}_{\{ \theta_{\ell+1}(x) \le \mu_{\ell}\}}\]
for   $x \in B_{\ell} \backslash B_{\ell+1}$, $2 \le \ell \le k$, and
\[ \bal \int_M G_{\ve}(x,\cdot) |k_{\ve}| dv_g \lesssim  &  \left \{ \bal
&  \frac{\mu_{1}^{\frac{n+2}{2}}}{\theta_1(x)^2} & \textrm{ if } (M,g) \textrm{ is l.c.f. around } \xi_1, \\
& \frac{\mu_1^{\frac{n+2}{2}}}{\theta_1(x)^{n-4}} &  \textrm{ if } (M,g) \textrm{ is not l.c.f.}  \\
\eal \right \}  \\
& + \frac{\theta_{2}(x)^2}{\mu_{1}^2} W_{2}(x) \mathds{1}_{\{ \theta_{2}(x) \le \mu_{1}\}} .
\eal\]
if $x \in B_1 \backslash B_2 = M \backslash B_2$. Compare the last two estimates with \eqref{estvpi} and \eqref{estvp1}: Proposition \ref{proplin} therefore shows that a solution $\vp_{\ve} \in K_k^{\perp}$ of \eqref{eqlin} behaves \emph{as if it satisfied the much simpler equation}
\ben \label{interpretation} \bal
 \triangle_g \vp_{\ve}  + (c_n S_g + \ve h) \vp_{\ve}  = k_{\ve} + \sum \limits_{\underset{j =0 .. n }{i=1..k}}    \lambda_{ij} ( \triangle_g   + c_n S_g + \ve h)Z_{i,j}, 
\eal \een
with the $\lambda_{ij}$ satisfying \eqref{vpl}. Proposition \ref{proplin} thus shows that the term $f'\big(  \sum_{i=1}^k \kappa_i U_i \big) \vp_{\ve}$ has no incidence on the behavior of $\vp_{\ve}$, provided $\vp_{\ve} \in K_k^{\perp}$: it therefore has to be seen as a considerably more involved analogue of Proposition \ref{proplinbase}, one which yields sharp pointwise estimates on $\vp_\ve$. The usefulness of Proposition \ref{proplin} in the proof of Theorems \ref{maintheo} and \ref{theoBN} becomes clear when anticipating on Section \ref{nonlin} below. In Section \ref{nonlin} we will perform a nonlinear fixed-point argument to solve \eqref{eqY} up to kernel elements. Thanks to Proposition \ref{proplin} we will be able to perform the point-fixing argument in suitably weighted spaces for the $C^0$ norm. An accurate choice of the weight will reduce the point-fixing to an application of Proposition \ref{proplin} with an error term given by \eqref{ferr} and by a negligible (and controlled) nonlinear correction.  

\medskip

The proof of Proposition \ref{proplin} goes through an a priori bubble-tree analysis. We iteratively obtain refined pointwise estimates on $|\vp_{\ve}|$ on each annular region $B_i \backslash B_{i+1}$ starting from $B_k$ (which corresponds to the highest bubble). These estimate depend on the pointwise behaviour of $\vp_\ve$ on all of $M$ and on the $\lambda_{ij}$ and are heavily intertwined. We then propagate these estimates until $M \backslash B_1$, we perform a second-order one-bubble analysis on the lowest bubble in $M \backslash B_1$ to prove \eqref{estvp1}, and we propagate it back to the highest bubble to obtain \eqref{estvpi}. The $\lambda_{ij}$ themselves are iteratively estimated in the course of the proof.

\subsection{A representation formula}

We start by proving a short lemma that will be crucial in the proof of Proposition \ref{proplin}. It is a representation formula for the linearised operator at the standard bubble in $\RR^n$:

\begin{lemme} \label{lemmerep}
Let $\vp \in C^{\infty}_c(\RR^n)$ and let, for any $x \in \RR^n$,
\[ \Pi(\vp)(x) = \sum_{j=0}^n \left( \frac{ \int_{\RR^n} \vp V_j U_0^{2^*-2} dy}{\int_{\RR^n} V_j^2 U_0^{2^*-2} dy}\right) V_j(x) \]
be the orthogonal projection of $\vp$ on $\textrm{Span}\{V_j, 0 \le j \le n\}$ for the $D^{1,2}$ scalar product in $\RR^n$. Then, for any $x \in \RR^n$:
\ben \label{formulerep}
 \big| \vp(x) -  \Pi(\vp)(x) \big| \lesssim \int_{\RR^n} |x-y|^{2-n} \left|\triangle_\xi  \vp- (2^*-1) U_0^{2^*-2} \vp \right|(y)dy
 \een
where $V_j$ and  $U_0$ are as in \eqref{defB0} and \eqref{defVi}. 
\end{lemme}

\begin{proof}
Let $g_0$ be the round metric on the sphere. Let  $N \in \mathbb{S}^n$ and let $\pi_N: \mathbb{S}^n \backslash \{-N\} \to \RR^n$ the stereographic projection sending $N$ to $0$. There holds
\[ (\pi_N^{-1})^* g_0 = U^{\frac{4}{n-2}} \xi, \]
where $\xi$ is the Euclidean metric and $U(x)^{\frac{4}{n-2}}= \frac{4}{(1+|x|^2)^2}$. As a consequence, $dv_{g_0} = U^{\frac{2n}{n-2}} dx$. For $\vp \in C^{\infty}_c(\RR^n)$ and $x \in \mathbb{S}^n$ we let:
\[ \tilde \vp(x) = \frac{\vp( \pi_N(x))}{U(\pi_N(x))}. \]
This defines a smooth function $\tilde{\vp}$ on $\mathbb{S}^n$ with compact support in $\mathbb{S}^n \backslash \{-N\}$, and there holds for any $\vp, \psi \in C^{\infty}_c(\RR^n)$,
\ben \label{L2Sn}
\int_{\mathbb{S}^n}\tilde{ \vp} \tilde{ \psi} dv_{g_0} = \int_{\RR^n}\vp \psi U^{2^*-2} dx . 
\een
By the conformal invariance of the conformal laplacian we have furthermore, for $x \in \mathbb{S}^{n}$:
\[ \Big( \triangle_{g_0} + \frac{n(n-2)}{4}\Big) \tilde{\vp} (x) = U(\pi_N(x))^{1-2^*} \triangle_{\xi} \vp (\pi(x)),  \]
so that for any $\vp, \psi \in C^{\infty}_c(\RR^n)$;
\[ \int_{\mathbb{S}^n} \Big( \langle \nabla \tilde{\vp}, \nabla \tilde{\psi} \rangle  + \frac{n(n-2)}{4} \tilde{\vp} \tilde{\psi}\Big) dv_{g_0} = \int_{\mathbb{R}^n} \langle \nabla \vp, \nabla \psi \rangle dx \]
and
\ben \label{confsphere}
 \big( \triangle_{g_0} -n \big) \tilde{\vp} (x) = U(\pi_N(x))^{1- 2^*} \Big( \triangle_{\xi} \vp (\pi_N(x)) - \frac{n(n+2)}{(1+|x|^2)^2}\vp (\pi_N(x)) \Big) 
 \een
for any $x \in \mathbb{S}^n$. Let 
\[ K = \big \{ f \in H^1(\mathbb{S}^n) \textrm{ such that } \triangle_{g_0} f - n f = 0\big \}. \] 
An orthonormal basis of $K$, for the $L^2$ scalar product \eqref{L2Sn} on $(\mathbb{S}^n, g_0)$, is given by the family $(K_0, \dots, K_{n})$ where 
\[K_i = c_i^{-1} x_i, \quad c_i^2 = \int_{\mathbb{S}^n} x_i^2 dv_{g_0}, \quad 0 \le i \le n, \]
and where the $x_i$ are the coordinate functions restricted to $\mathbb{S}^n$. Let $G(x,y)$ denote the Green's function (see e.g. \cite{RobDirichlet}) for the operator $ \triangle_{g_0} -n$ on $\mathbb{S}^n$. Standard properties of the Green's function ensure that 
\[ |G(x,y)| \le C d_{g_0}(x,y)^{2-n} \textrm{ for } x\neq y \]
and that, for any smooth function $ \psi$ in $\mathbb{S}^n$ and any $z \in \mathbb{S}^n$,
\[ \psi(z) - \Pi_{K}(\psi)(z) = \int_{\mathbb{S}^n} G(z,y) \big( \triangle_{g_0} \psi(y) - n \psi(y) \big) dv_{g_0}(y),\]
where $\Pi_K(\psi) = \sum_{j=0}^n( \int_{\mathbb{S}^n} \psi K_j dv_{g_0}) K_j$ is the $L^2$ projection of $\psi$ over $K$. Let now $x\in \RR^n$, and apply the latter to $\psi = \tilde{\vp}$ and $z = \pi_N^{-1}(x)$. Letting $\tilde{K}_i(x) = U(x) K_i(\pi_N^{-1}(x))$ we thus get, by \eqref{L2Sn} and \eqref{confsphere}:
\ben \label{reppartielle}
 \vp(x) -  \Pi(\vp)(x) = \int_{\RR^n}  \tilde{G}(x,y) \Big(  \triangle_{\xi} \vp (y) - \frac{n(n+2)}{(1+|y|^2)^2}\vp (y) \Big) dy,    
 \een
where we have let
\[ \bal
  \Pi(\vp)(x) & =  \sum_{j=0}^n \Big( \int_{\RR^n} \vp \tilde{K}_j U^{2^*-2} dx \Big) \tilde{K}_j(x)
  \eal \]
  and, for any $x\neq y$ in $\RR^n$,
  \[  \tilde{G}(x,y) = U(x) G\big(\pi_N^{-1}(x), \pi_N^{-1}(y) \big) U(y) . \]
  Standard properties of the stereographic projection yield
  \[ d_{g_0}\big( \pi_N^{-1}(x),\pi_N^{-1}(y) \big) \gtrsim \frac{|x-y|}{\sqrt{(1+|x|^2)(1+|y|^2)}} \]
  so that 
  \[ |  \tilde{G}(x,y) | \lesssim |x-y|^{2-n} \quad \textrm{ for any } x \neq y \textrm{ in } \RR^n.\]
Finally, if $\vp \in C^{\infty}_c(\RR^n)$, letting $ z = \frac{x}{\sqrt{n(n-2)}}$ and  $\hat{\vp}(x) = \vp (z)$ gives
\[ \triangle_{\xi} \hat{\vp}(x) - (2^*-1) U_0^{2^*-2}(x) \hat{\vp}(x) = \frac{1}{n(n-2)} \Big( \triangle_{\xi} \vp(z) -  \frac{n(n+2)}{(1+|z|^2)^2}\vp (z) \Big). \]
It remains to apply \eqref{reppartielle} at $\frac{x}{\sqrt{n(n-2)}}$ and to perform the scaling $y \mapsto \frac{y}{\sqrt{n(n-2)}}$  in the integral. The proof is concluded by remarking that in the chart $\pi_N$ the coordinate functions $(x_0, \dots, x_n)$ on $\mathbb{S}^n$ write respectively as 
\[ \frac{1-|x|^2}{1+|x|^2}, \frac{2x_1}{1+|x|^2}, \dots, \frac{2x_n}{1+|x|^2},  \]
which gives in turn
  \[ \tilde{K}_0\Big( \frac{x}{\sqrt{n(n-2)}}\Big)  = c_0^{-1}2^{\frac{n-2}{2}} V_0(x)   \textrm{ and } \tilde{K}_i \Big( \frac{x}{\sqrt{n(n-2)}}\Big)= c_i^{-1}\frac{2^{\frac{n}{2}}}{n(n-2)} V_i(x)  \]
  for all  $1 \le i \le n$, where the $V_i$ are given by \eqref{defVi}.
\end{proof}

\subsection{Proof of Proposition \ref{proplin}, Part 1: a recursive estimate}

Let $k \ge 1$ be an integer, $\{\kappa_i\}_{1 \le i \le k} \in \{\pm1\}^k$ and $\ve_0 >0$ be as in Proposition \ref{proplinbase}. For any $\ve \in (0, \ve_0)$ let $(t_1, \xi_1, (t_{i,\ve}, z_{i,\ve})_{2 \le i \le k}) \in S_k$, $h \in C^0(M)$, $k_{\ve} \in C^0(M)$ satisfiying \eqref{hypf}, and let $\vp_{\ve} \in K_k^{\perp}$ be the unique solution of \eqref{eqlin}, where the $\lambda_{ij}$ are uniquely defined by \eqref{eqlinbase}. By \eqref{estH1lin} and \eqref{hypf} we easily get 
$$ \Vert \vp_{\ve} \Vert_{H^1} = o(1) $$
as $\ve \to 0$, and hence integrating \eqref{eqlin} against $Z_{i,j}$ for all $1 \le i \le k$ and $0 \le j \le n$ yields
\ben \label{estlambdaij}
|\lambda_{ij}| = o(1) \quad \textrm{ for } 1 \le i \le k,  \quad 0 \le j \le n.
\een
For any $2 \le i \le k$, we define
\ben \label{defCi} 
C_i = B_{g_{\xi_i}}(\xi_i, 4 \sqrt{\mu_i\mu_{i-1}}) \backslash B_{g_{\xi_{i+1}}}(\xi_{i+1}, \frac14 \sqrt{\mu_{i+1}\mu_{i}}),
\een and
\ben \label{defC1}
C_1 = B_{g_{\xi_1}}(\xi_1, 2 r_0) \backslash B_{g_{\xi_{2}}}(\xi_{2}, \frac14 \sqrt{\mu_{2}\mu_{1}}),
\een
where $r_0$ is as in \eqref{defr0}. We also define
\ben \label{defai}
\bal
 a_i & = a_{i,\ve} := \left \Vert \frac{\phi_{\ve}}{W_i} \right \Vert_{L^{\infty}(C_i)} \textrm{ for } 2 \le i \le k, \\
  a_1 & = a_{1,\ve} := \left \Vert \frac{\phi_{\ve}}{W_1} \right \Vert_{L^{\infty}(M\backslash B_{g_{\xi_{2}}}(\xi_{2}, \frac14 \sqrt{\mu_{2}\mu_{1}}))} ,
\eal
\een
and, for a given $1 \le i \le k$,
\ben \label{Lambdai}
\Lambda_i : = \sum_{j=0..n} |\lambda_{ij}|.
\een
By definition of $a_i$, $\vp_{\ve}$ satisfies $|\vp_{\ve} |\le a_i W_i$ in $C_i$. The proof of  Proposition \ref{proplin} consists in an explicit estimation of the $a_i$.

\medskip

For a fixed $1 \le  \ell \le k$, we will say that the property $(H_{\ell})$ is satisfied if the following two properties $(H1_{\ell})$ and $(H2_{\ell})$ are both satisfied:

\medskip

\textbf{Property $(H1_{\ell})$:}   for any $ p \in \{\ell, \dots,  k\} $ and for any $ x \in C_p$,
\ben \label{H1} \tag{$H1_{\ell}$}
\bal
 |\vp_{\ve}(x)| & \lesssim \ve \mu_{p}^{3 - \frac{n}{2}} \sum_{j\ge p} a_j  +   \sum_{j\le p-1} \mu_j^{1 - \frac{n}{2}}\Lambda_j +  \mu_{p-1}^{1 - \frac{n}{2}} \left( \frac{\mu_{p}}{\theta_{p}(x)} \right)^2 \\
 &+ \sum_{j\ge p+1}\mu_{j-1}^2 a_j W_j (x) + \ve \big(1 + \sum_{j\le p} a_j \big) \theta_p^2(x) W_p(x)  \\
 & +   (1+a_p)\frac{\theta_{p+1}(x)^2}{\mu_{p}^2} W_{p+1}(x) \mathds{1}_{\{ \theta_{p+1}(x) \le \mu_{p}\}} \\
  \eal
\een 

\textbf{Property $(H2_{\ell})$:} for any $ p \in \{\ell, \dots,  k\} $,
\ben \label{H2} \tag{$H2_{\ell}$}
 \Lambda_p \lesssim  \ve \mu_p^2 \left( 1 + \sum_{i=1}^k a_i \right) . 
 \een

A word on the notations is in order here. In \eqref{H1} and \eqref{H2} we take the following conventions: $W_{k+1} = 0$ and  $\mu_0 = \frac{1}{\mu_1}$. In the rest of the proof of Proposition \ref{proplin} both the expressions $\ve \mu_\ell^2$ and $\left( \frac{\mu_\ell}{\mu_{\ell-1}} \right)^{\pui}$ shall appear throughout the computations. When $\ell \ge2$, of course, they are of the same order in $\ve$ by \eqref{growthmuell}. They might differ when $\ell=1$, however, according to whether $(M,g)$ is locally conformally flat around $\xi_1$ or not. It will be understood, by the choice of our convention $\mu_0 = \frac{1}{\mu_1}$, that  the term $\left( \frac{\mu_\ell}{\mu_{\ell-1}} \right)^{\pui}$ becomes $\mu_1^{n-2}$ when $\ell=1$.

\medskip

The core of the proof of Proposition \ref{proplin} consists in proving recursively that  $(H_{\ell})$ is true for any $1 \le \ell \le k$. 

\begin{lemme} \label{lemmeinduction}
For any $1 \le \ell \le k$, $(H_{\ell})$ is true. 
\end{lemme}

\begin{proof}
We prove Lemma \ref{lemmeinduction} by induction. Since the proof of the initialisation and the induction part are very close we prove them together. Let $1\le \ell \le k$ be an integer. \textbf{If $\ell = k$, we do not assume anything. If $ \ell \le k-1$, we assume that $(H_{\ell+1})$ holds true.} The proof of $(H_{\ell})$ goes through a series of steps. 

\medskip

\textbf{Step $1$: a local rescaling.} We let $\eta: \RR_+ \to \RR_+$ be a smooth nonnegative function with $\eta \equiv 1$ in $[0,\frac32]$ and $\eta \equiv 0$ in $\RR \backslash [0,2]$. We let in the following:
\[ \hat{B}_{\ell} = B_0 \Big( 2 \sqrt{\frac{\mu_{\ell-1}}{{\mu_\ell}}} \Big) \textrm{ if } \ell \ge 2, \quad \hat{B}_1 = B_0 \Big(\frac{2r_0}{\mu_1} \Big) \]
where $r_0$ is as in \eqref{defr0}. We also let 
\[ \hat{B}_{\ell+1} = B\Big(z_{\ell+1}, \sqrt{\frac{\mu_{\ell+1}}{\mu_\ell}} \Big). \]
For any $x \in \hat{B}_{\ell}$, we let 
\ben \label{defhpv}
 \hat{\vp}_{\ve}(x)  = \mu_{\ell}^\pui \big( \Lambda_{\xi_\ell}^{-1} \vp_{\ve} \big) \big( \exp^{g_{\xi_{\ell}}}_{\xi_{\ell}}(\mu_\ell x) \big)\cdot \eta \Big(   \sqrt{\frac{\mu_{\ell}}{{\mu_{\ell-1}}}} x\Big). 
 \een
In the following we will let $x_{\ell} =  \exp^{g_{\xi_{\ell}}}_{\xi_{\ell}}(\mu_\ell x)$ if $x \in \hat{B}_{\ell}$ and we will define
\[ \hat{g}_{\ell}(x) = (\exp^{g_{\xi_{\ell}}}_{\xi_{\ell}})^*g (\mu_\ell x). \]
Let also, for $x \in \hat{B}_{\ell}$ and $1 \le i \le k $, $0 \le j \le n$:
\[ \hat{W}_i(x) =  \mu_{\ell}^\pui \big( \Lambda_{\xi_\ell}^{-1}W_{i} \big)(x_{\ell}) \quad \textrm{ and } \quad   \hat{Z}_{i,j}(x)  =  \mu_{\ell}^\pui \big( \Lambda_{\xi_\ell}^{-1}Z_{i,j} \big)(x_{\ell}). \]  By the conformal invariance property of the conformal laplacian and \eqref{eqlin}, the function $\hpv$ satisfies the following equation: if $x \in \hat{B}_\ell$,
\ben \label{eqhpv}
\bal
\Big(  \triangle_{\hat{g}_{\ell}} & + \mu_{\ell}^2 c_n S_{g_{\xi_\ell}}(x_\ell) + \mu_{\ell}^2 \ve h(x_{\ell})  \Lambda_{\xi_{\ell}}(x_{\ell})^{2 - 2^*}\Big) \Big(  \hpv -   \sum \limits_{\underset{j =0 .. n }{i=1..k}}    \lambda_{ij} \hat{Z}_{i,j} \Big) \\
& = \mu_{\ell}^{\frac{n+2}{2}} \Lambda_{\xi_{\ell}}(x_{\ell})^{1 - 2^*} k_{\ve}(x_\ell) + \mu_\ell^2  \Lambda_{\xi_{\ell}}(x_{\ell})^{2 - 2^*}f'\big(  \sum_{i=1}^k \kappa_i W_i \big)(x_{\ell}) \hpv  \\
& + O\Big(\frac{\mu_{\ell}}{{\mu_{\ell-1}}} |\hpv|  + \sqrt{\frac{\mu_{\ell}}{{\mu_{\ell-1}}}} |\nabla \hpv| \Big)\cdot \mathds{1}_{\frac32 \sqrt{\frac{\mu_{\ell-1}}{{\mu_\ell}}} \le |x| \le 2  \sqrt{\frac{\mu_{\ell-1}}{{\mu_\ell}}}}
\eal
\een
If $(M,g)$ is locally conformally flat around $\xi_1$ the metric $\Lambda_{\xi_{\ell}}^{\frac{4}{n-2}} g$ is flat, and then 
\[ \big( \triangle_{\hat{g}_{\ell}}  + \mu_{\ell}^2 c_n S_{g_{\xi_\ell}}(x_\ell)\big) \hpv(x)  = \triangle_{\xi} \hpv(x) \]
in $\tilde{B}_{\ell}$, where $\xi$ is the Euclidean metric. Otherwise, fine properties of the conformal normal coordinates (see e.g. \cite{EspositoPistoiaVetois}, Lemma $6.3$) give
\[ \bal
  \big( \triangle_{\hat{g}_{\ell}}  + \mu_{\ell}^2 c_n S_{g_{\xi_\ell}}(x_\ell)\big) \hpv(x)  & = \triangle_{\xi} \hpv(x)  + O \Big( \mu_\ell^2 |x|^2 |\nabla^2 \hpv(x)| \Big) + O \Big( \mu_\ell^3 |x|^2 |\nabla \hpv(x)| \Big) \\
  & + O\Big(\mu_\ell^4 |x|^2 |\hpv|\Big). 
  \eal \] 
By \eqref{rapportsbulles} we have 
\ben \label{diffbulle}
 \Big|  f'\big(  \sum_{i=1}^k \kappa_i W_i \big)(x_{\ell}) - (2^*-1)W_\ell(x_{\ell})^{2^*-2}\Big| \lesssim W_\ell^{2^*-3} (W_{\ell+1} + W_{\ell-1}) 
 \een
 in $B_{\ell} \backslash B_{\ell+1}$, where the term $W_{\ell-1}$ is assumed to vanish when $\ell=1$, which yields after scaling and by \eqref{growthmuell}:
\[  \bal
\mu_\ell^2&   \Lambda_{\xi_{\ell}}(x_{\ell})^{2 - 2^*} \left| f'\big(  \sum_{i=1}^k \kappa_i W_i \big)(x_{\ell}) - (2^*-1)W_\ell(x_{\ell})^{2^*-2}\right|  \\
& \lesssim 
\left \{
\bal
& \ve \mu_\ell^2  (1 + |x|)^{n-6} +  \left( \frac{\mu_{\ell+1}}{\mu_{\ell}} \right)^\pui \frac{(1+|x|)^{n-6}}{\big( \frac{\mu_{\ell+1}}{\mu_\ell} + |x - z_{\ell+1}|\big)^{n-2}}
& \textrm{ in } \hat{B}_{\ell} \backslash \hat{B}_{\ell+1} \\
& \sum_{i=\ell+1}^k \hat{W}_i^{2^*-2}(x) 
& \textrm{ in } \hat{B}_{\ell+1} \eal .
\right.  
\eal \]
By \eqref{defW} and the choice of the conformal correction $\Phi$ we also have, with \eqref{defmass} and \eqref{nonlcf}:
\[\bal
 \left| \mu_\ell^2    \Lambda_{\xi_{\ell}}(x_{\ell})^{2 - 2^*} W_{\ell}^{2^*-2} (x_\ell) - U_0^{2^*-2}(x)\right| \lesssim 
 \left \{
 \bal
 & \mu_\ell^{n-2} (1+|x|)^{n-6} & \textrm{ if } (M,g) \textrm{ is l.c.f.} \\
 & \mu_{\ell}^4 & \textrm{ otherwise. } \\
 \eal 
 \right.
 \eal \]
Finally, the definition of the functions $Z_{i,j}$ in \eqref{defZi} and straightforward computations using \eqref{lapZ0} and \eqref{lapZi} yield that, for any $1 \le i \le k$,
\[ \bal 
\left| \Big(  \triangle_{\hat{g}_{\ell}}  + \mu_{\ell}^2 c_n S_{g_{\xi_\ell}}(x_\ell) + \mu_{\ell}^2 \ve h(x_{\ell})  \Lambda_{\xi_{\ell}}(x_{\ell})^{2 - 2^*}\Big)  \sum_{j=0..n}  \lambda_{ij} \hat{Z}_{i,j} \Big)\right| \\
\lesssim \left \{ 
\bal
& \left( \frac{\mu_\ell}{\mu_{i}} \right)^{\frac{n+2}{2}} \Lambda_i & \textrm{ if } i \le \ell-1 \\
& \Lambda_{\ell} (1+|x|)^{-n-2} &  \textrm{ if } i = \ell \\
&  \Lambda_{i} \hat{W}_i(x)^{2^*-1}  &  \textrm{ if } i \ge \ell + 1 \\ 
\eal 
\right. 
\eal \]
Note also that we have $\hat{Z}_{\ell,j}(x) = V_j(x)$ for any $x \in \hat{B}_\ell$, where $V_j$ is as in \eqref{defVi}. Gathering the previous estimates in \eqref{eqhpv} shows that $\hpv$ satisfies:
\ben \label{eqhpv2}
\bal
&\Big| \Big(\triangle_{\xi} -  (2^*-1) U_0^{2^*-2} \Big)   \hpv  \Big| \lesssim  \ve \mu_\ell^2 | \hpv|  + \mu_{\ell}^{\frac{n+2}{2}}  |k_{\ve}(x_\ell) |   \\
& + \frac{\Lambda_{\ell}}{(1+|x|)^{n+2}} + \sum_{j \le \ell-1}  \left( \frac{\mu_\ell}{\mu_{j}} \right)^{\frac{n+2}{2}}\Lambda_j  + \sum_{i\ge \ell+1}  \Lambda_{i} \hat{W}_i(x)^{2^*-1} \\
& + \left \{
\bal
&\ve \mu_\ell^2  (1 + |x|)^{n-6} & \\
& +  \left( \frac{\mu_{\ell+1}}{\mu_{\ell}} \right)^\pui \frac{(1+|x|)^{n-6}}{\big( \frac{\mu_{\ell+1}}{\mu_\ell} + |x - z_{\ell+1}|\big)^{n-2}}
& \textrm{ in } \hat{B}_{\ell} \backslash \hat{B}_{\ell+1} \\
& \sum_{i=\ell+1}^k \hat{W}_i(x)^{2^*-2} 
& \textrm{ in } \hat{B}_{\ell+1} \eal 
\right \} \cdot |\hpv|   \\
 & + 
 \left \{ 
 \bal
 & \mu_\ell^{n-2}(1+|x|)^{n-6} |\hpv|  & (M,g) \textrm{ l.c.f.} \\
 &  \mu_{\ell}^4 (1+ |x|^2) |\hpv| +  \mu_\ell^3 |x|^2 |\nabla \hpv(x)| + \mu_\ell^2 |x|^2 |\nabla^2 \hpv(x)| & \textrm{ otherwise } \\
 \eal 
 \right. \\
& +  \Big(\frac{\mu_{\ell}}{{\mu_{\ell-1}}} |\hpv|  + \sqrt{\frac{\mu_{\ell}}{{\mu_{\ell-1}}}} |\nabla \hpv| \Big)\cdot \mathds{1}_{\frac32 \sqrt{\frac{\mu_{\ell-1}}{{\mu_\ell}}} \le |x| \le 2  \sqrt{\frac{\mu_{\ell-1}}{{\mu_\ell}}}}.  \\
\eal
\een
When $\ell \ge 2$ or when $\ell=1$ and  $(M,g)$ is locally conformally flat the term $ \mu_\ell^{n-2} (1+|x|)^{n-6} |\hpv|$  is absorbed into the term $ \ve \mu_\ell^2 (1 + |x|)^{n-6} |\hpv|$. Recall also that when $\ell=1$ the term $\ve \mu_\ell^2 (1 + |x|)^{n-6} |\hpv|$, which is due to the contribution of $W_\ell^{2^*-3} W_{\ell-1}$, vanishes.

\medskip

\textbf{Step $2$: a first pointwise estimate.} In this step we prove the following rough estimate:

\begin{lemme}
Let $x \in \hat{B}_{\ell} \backslash \hat{B}_{\ell+1}$. Then 
\ben \label{claim1}
\bal
 \big| \hpv(x) &-  \Pi(\hpv)(x) \big| \\
 & \lesssim \ve \mu_\ell^2   \sum_{j \ge \ell} a_j + \sum_{j \le \ell-1} \left(\frac{\mu_\ell}{\mu_j}\right)^\pui  \Lambda_j   +  \frac{\ve \mu_\ell^2\big( 1 +  \sum_{j \le \ell} a_j \big)}{(1+|x|)^{n-4}} \\
 &+\frac{ \Lambda_\ell}{(1+|x|)^{n-2}} + \left(  \frac{\mu_\ell}{\mu_{\ell-1}} \right)^\pui  \frac{1}{(1+|x|)^2} \\
 & \sum_{j \ge \ell+1} \mu_{j-1}^2 a_j \hat{W}_j + \frac{\mu_\ell^2 a_\ell}{(1+|x|)^{n-4}} \mathds{1}_{\textrm{n.l.c.f}} \\
 & +  (1 +a_\ell)  \Big( \frac{\mu_{\ell+1}}{\mu_\ell} + |x - z_{\ell+1}|\Big)^{2} \hat{W}_{\ell+1}(x) \mathds{1}_{\{  |x - z_{\ell+1}| \le 2 \}},  \eal 
\een
where $ \Pi(\hpv)$ is defined in Lemma \ref{lemmerep}. The notation $\mathds{1}_{\textrm{n.l.c.f}}$ indicates that this term vanishes if $(M,g)$ is locally conformally flat.
\end{lemme} 

\begin{proof}
We apply \eqref{formulerep} to $\hpv$. As already mentioned, if $\ell=k$ we do not assume anything, but if $\ell \le k-1$ we assume that $(H_{\ell+1})$ holds. Let $x \in \hat{B}_{\ell} \backslash \hat{B}_{\ell+1}$ (which only implies $x \in \hat{B}_k$ if $\ell=k$). 

\medskip

\textbullet\ First, straightforward computations using \eqref{defai} and Giraud's lemma  show that 
\be 
\int_{\hat{B}_{\ell} \backslash \hat{B}_{\ell+1}}  \ve \mu_\ell^2 | \hpv(y)| |x-y|^{2-n} dy \lesssim \frac{\ve \mu_\ell^2 a_{\ell}}{(1+|x|)^{n-4}}. 
\ee
Since $x \in \hat{B}_{\ell} \backslash \hat{B}_{\ell+1}$ we also obtain, if $\ell \le k-1$,
\be 
\bal
\int_{\hat{B}_{\ell+1}}  \ve \mu_\ell^2 & | \hpv(y)| |x-y|^{2-n} dy \\
& \lesssim  \Big( \frac{\mu_{\ell+1}}{\mu_\ell} + |x - z_{\ell+1}|\Big)^{2-n} \ve \mu_{\ell}^2 \int_{\hat{B}_{\ell+1}} |\hpv(y)| dy  \\
& \lesssim  \Big( \frac{\mu_{\ell+1}}{\mu_\ell} + |x - z_{\ell+1}|\Big)^{2-n} \ve \mu_{\ell}^{1 - \frac{n}{2}} \int_{B_{\ell+1}} |\vp_{\ve}(y)| dv_g(y) \\
& \lesssim \big( \sum_{j \ge \ell+1} \mu_{j-1}^2 a_j\big) \hat{W}_{\ell+1}(x) + \left( \frac{\mu_{\ell}}{\mu_{\ell-1}} \right)^{\pui} \frac{1}{(1+|x|)^2}, \\
\eal
\ee
where we used the induction property $(H1_{\ell+1})$ to obtain the last line. Gathering these two estimates yields
\ben  \label{proplin1}
\bal
\int_{\hat{B}_{\ell}} & \ve \mu_\ell^2 | \hpv(y)| |x-y|^{2-n} dy\lesssim  \frac{\ve \mu_\ell^2 a_{\ell}}{(1+|x|)^{n-4}}  \\
& + \big( \sum_{j \ge \ell+1} \mu_{j-1}^2 a_j\big) \hat{W}_{\ell+1}(x) + \left( \frac{\mu_{\ell}}{\mu_{\ell-1}} \right)^{\pui} \frac{1}{(1+|x|)^2}.
\eal
\een 

\medskip

\textbullet\ Using now \eqref{hypf}, \eqref{tech1} and \eqref{tech2} we obtain 

\ben \label{proplin2}
\bal
\int_{\hat{B}_{\ell}} &\mu_\ell^{\frac{n+2}{2}} |k_{\ve}(x_\ell)| |x-y|^{2-n} dy \lesssim \frac{\ve \mu_\ell^2}{(1+|x|)^{n-4}} + \frac{\mu_\ell^n}{(1+|x|)^2} \mathds{1}_{\textrm{l.c.f.}}   \\
& +  \left( \frac{\mu_{\ell}}{\mu_{\ell-1}} \right)^{\pui} \frac{1}{(1+|x|)^2} +  \Big( \frac{\mu_{\ell+1}}{\mu_\ell} + |x - z_{\ell+1}|\Big)^{2} \hat{W}_{\ell+1} \mathds{1}_{\{  |x - z_{\ell+1}| \le 2 \}}.
\eal
\een

\medskip

\textbullet\ We now compute the terms involving the $\lambda_{ij}$. On one side, Giraud's lemma together with straightforward computations yields
\ben \label{proplin3}
\bal
\int_{\hat{B}_\ell} \Bigg( \frac{\Lambda_{\ell}}{(1+|x|)^{n+2}} &+   \sum_{j \le \ell-1}  \left( \frac{\mu_\ell}{\mu_{j}} \right)^{\frac{n+2}{2}}\Lambda_j\Bigg) |x-y|^{2-n} dy \\
& \lesssim \frac{\Lambda_\ell}{(1+|x|)^{n-2}} + \frac{\mu_{\ell-1}}{\mu_{\ell}} \sum_{j \le \ell-1}  \left( \frac{\mu_\ell}{\mu_{j}} \right)^{\frac{n+2}{2}}\Lambda_j \\
& \lesssim  \frac{\Lambda_\ell}{(1+|x|)^{n-2}} +   \ve \mu_\ell^2 \sum_{j \le \ell-1}\Lambda_j,
\eal
\een 
where the sum is empty if $\ell=1$. On the other side, straightforward computations yield 
\be 
\bal
\int_{\hat{B}_\ell}  \sum_{i\ge \ell+1}  \Lambda_{i} \hat{W}_i(y)^{2^*-1}|x-y|^{2-n} dy & \lesssim \sum_{i\ge \ell+1} \Lambda_i \hat{W}_i(x).
\eal
\ee
This sum is empty if $\ell=k$. If $ \ell \le k-1$, we use in addition the induction property $(H2_{\ell+1})$: by \eqref{growthmuell} and by \eqref{rapportsbulles} we get 
\ben \label{proplin4}
\bal
\int_{\hat{B}_\ell} & \sum_{i\ge \ell+1}  \Lambda_{i} \hat{W}_i(y)^{2^*-1}|x-y|^{2-n} dy \lesssim   \left( \frac{\mu_{\ell+1}}{\mu_{\ell}} \right)^{\pui} \frac{\sum_{j=1}^k a_j}{(1+|x|)^{n-2}}. \\
\eal
\een

\medskip

\textbullet\ We now turn to the third line in the right-hand side of \eqref{eqhpv2}. If $ \ell \le k-1$, using the induction property $(H1_{\ell+1})$ we get that 
\ben \label{proplin5}
\bal
\int_{\hat{B}_{\ell+1}}  \sum_{i\ge \ell+1}  \hat{W}_i(y)^{2^*-2} |\hpv(y)||x-y|^{2-n} dy & \lesssim  \big( \sum_{j \ge \ell+1} \mu_{j-1}^2 a_j\big) \hat{W}_{\ell+1}(x) \\
& + \ve \mu_\ell^2 \frac{\big(\sum_{j\le \ell} a_j \big)}{(1+|x|)^{n-2}}.
\eal
\een
With the definition of $a_\ell$ in \eqref{defai} we obtain with Giraud's lemma
\ben \label{proplin6}
\int_{\hat{B}_{\ell} \backslash \hat{B}_{\ell+1}} \left( \frac{\mu_\ell}{\mu_{\ell-1}} \right)^\pui (1 + |y|)^{n-6}|\hpv(y)| |x-y|^{2-n} dy   \lesssim   \left( \frac{\mu_\ell}{\mu_{\ell-1}} \right)^\pui  \frac{a_\ell}{(1+|x|)^2}, \\
\een 
while using \eqref{tech2} below yields
\ben \label{proplin7}
\bal
\int_{\hat{B}_{\ell} \backslash \hat{B}_{\ell+1}}   & \left( \frac{\mu_{\ell+1}}{\mu_{\ell}} \right)^\pui \frac{(1+|y|)^{n-6}}{\big( \frac{\mu_{\ell+1}}{\mu_\ell} + |y - z_{\ell+1}|\big)^{n-2}} |\hpv(y)| |x-y|^{2-n} dy \\ 
& \lesssim a_\ell  \Big( \frac{\mu_{\ell+1}}{\mu_\ell} + |x - z_{\ell+1}|\Big)^{2} \hat{W}_{\ell+1}(x) \mathds{1}_{\{  |x - z_{\ell+1}| \le 2 \}}.
\eal
\een

\medskip

\textbullet\ Mimicking the previous computations we also obtain, using $(H1_{\ell+1})$ when $\ell \le k-1$, that when $(M,g)$ is locally conformally flat
\ben \label{proplin8}
\bal
\int_{\hat{B}_{\ell}} \mu_\ell^{n-2} (1 + |y|)^{n-6}|\hpv(y)| |x-y|^{2-n} dy &  \lesssim  \big( \sum_{j \ge \ell+1} \mu_{j-1}^2 a_j\big) \hat{W}_{\ell+1}(x) \\
&  +  \frac{\mu_{\ell}^{n-2}(1+a_\ell) }{(1+|x|)^2} 
\eal
\een
holds, while when $(M,g)$ is not locally conformally flat around $\xi_1$ we obtain:
\ben \label{proplin9}
\bal
\int_{\hat{B}_{\ell}}  \mu_{\ell}^4 (1+ |y|^2) |\hpv(y)| |x-y|^{2-n} dy  & \lesssim   \frac{\mu_\ell^4(1+a_\ell)}{(1+|x|)^{n-6}} + \frac{\mu_\ell^4 \sum_{j \ge \ell}a_j }{(1+|x|)^{n-2}}\\
 & + \sum_{j \ge \ell+1} \mu_{j-1}^2 a_j \hat{W}_j(x) .
\eal
\een

\medskip

\textbullet \ The only terms in \eqref{eqhpv2}  that remain to be estimated involve derivatives of $\hpv$. We control them with the following Claim:

\begin{claim} \label{claimder2}
Let $1 \le i \le k$. For any $x \in C_i$, where $C_i$ is as in \eqref{defCi}, there holds
\[\bal 
|\vp_{\ve}(x)| &+ \tilde{\theta}_i(x) |\nabla \vp_{\ve}(x)| + \tilde{\theta}_i(x)^2 |\nabla^2 \vp_{\ve}(x)|  \\
& \lesssim (a_i  + \sum_{p\ge i} \Lambda_p) W_i(x) + \mu_{i-1}^{1- \frac{n}{2}} \left( \frac{\mu_i}{\theta_i(x)}\right)^2  + \frac{\theta_{i+1}(x)^2}{\mu_{i}^2} W_{i+1}(x) \mathds{1}_{\{ \theta_{i+1}(x) \le \mu_{i}\}} \\
&+   \sum_{p\le i-1} \Lambda_p \mu_p^{1 - \frac{n}{2}} , 
\eal \]
where we have let 
\[ \tilde{\theta}_i(x) = \min \big( \theta_i(x), \theta_{i+1}(x) \big). \]
 \end{claim}
 
 \begin{proof}
 The proof is a simple application of elliptic regularity results. We keep the notations of \eqref{defhpv}. By \eqref{eqhpv2} it is easily seen that for any $y \in \hat{B}_i \backslash \hat{B}_{i+1}$
 \[\bal
  \Big| \Big(& \triangle_{\hat{g}_i} + O(\mu_i)^2 \Big) \Big( \hpv -  \sum \limits_{\underset{q =0 .. n }{p=1..k}}    \lambda_{p,q} \hat{Z}_{p,q} \Big)(y) \Big| \lesssim \frac{a_i}{(1+|y|)^{n+2}}  \\
  &  +\frac{\ve \mu_i^2 }{(1+|y|)^{n-2}}  + \left( \frac{\mu_{i+1}}{\mu_{i}}\right)^{\pui} \frac{1}{(1+|y|)^4} \frac{1}{\Big( \frac{\mu_{i+1}}{\mu_i} + |z_{i+1} - y| \Big)^{n-2}}.
 \eal
 \] 
 Let $(y_\ve)_\ve$ be a sequence of points in $\hat{B}_i \backslash \hat{B}_{i+1}$. 
 
 We first assume that $|y_\ve - z_{i+1}| \le 2$. For any $z \in B_0(3/2) \backslash B_0(2/3)$ we let 
 \[ \tilde{\vp}_\ve(z) = \Big(\hpv  -  \sum \limits_{\underset{q =0 .. n }{p=1..k}}    \lambda_{p,q} \hat{Z}_{p,q}\Big) \big(z_{i+1} + \alpha_i z \big), \]
 where we have let $\alpha_i = |y_\ve - z_{i+1}|$. Letting $\tilde{g}_i(z) = \hat{g}_i\big( z_{i+1} + \alpha_i z \big)$ and since $ \alpha_i \lesssim  \sqrt{\frac{\mu_{i-1}}{\mu_i}}$ by definition  it is then easily seen that $\tilde{\vp}_\ve$ satisfies
 \[ \bal
  \Big|  \triangle_{\tilde{g}_i}  \tilde{\vp}_\ve(z)+ O\big(\mu_i \mu_{i-1} |\tilde{\vp}_\ve(z)|\big) \Big| &\lesssim \alpha_i^2 \Big( a_i +\ve \mu_i^2  \Big) + \left( \frac{\mu_{i+1}}{\mu_{i}}\right)^{\pui} \alpha_i^2 \Big( \frac{\mu_{i+1}}{\mu_i} + \alpha_i\Big)^{2-n}.
 \eal
 \]
 Since $ \alpha_i \lesssim  \sqrt{\frac{\mu_{i-1}}{\mu_i}}$, $\tilde{g}_i$ is $C^k$-close to the Euclidean metric for any $k \ge1$. Standard elliptic theory then yields 
 \ben \label{claimder21} \bal 
 \Vert \tilde{\vp}_\ve \Vert_{C^2\big(B_0(5/4) \backslash B_0(4/5)\big)} & \lesssim \alpha_i^2 \Big( a_i+\ve \mu_i^2  \Big)   + \left( \frac{\mu_{i+1}}{\mu_{i}}\right)^{\pui} \alpha_i^2 \Big( \frac{\mu_{i+1}}{\mu_i} + \alpha_i\Big)^{2-n}.
 \eal \een
 Since $y_{\ve} \in \hat{B}_i \backslash \hat{B}_{i+1}$, by letting $\tilde{y}_\ve = \exp_{\xi_i}^{g_{\xi_i}}(\mu_i y_\ve)$  there holds $\alpha_i \lesssim \frac{\theta_{i+1}(\tilde{y}_\ve)}{\mu_i}$ and $\frac{\theta_{i+1}(\tilde{y}_\ve)}{\mu_i} \lesssim \alpha_i$. Scaling back \eqref{claimder21} to the original variables yields, for all $x \in B_i$ with $\theta_{i+1}(x) \le \mu_i$:
 \[\bal 
|\vp_{\ve}(x)| &+ \theta_{i+1}(x) |\nabla \vp_{\ve}(x)| + \theta_{i+1}(x)^2 |\nabla^2 \vp_{\ve}(x)|  \\
& \lesssim  (\sum_{p\ge i} \Lambda_p) W_i(x) + (a_i + \ve \mu_i^2) \frac{\theta_{i+1}(x)^2}{\mu_i^2} \mu_i^{1 - \frac{n}{2}}+  \frac{\theta_{i+1}(x)^2}{\mu_{i}^2} W_{i+1}(x) \mathds{1}_{\{ \theta_{i+1}(x) \le \mu_{i}\}} \\
&+  \sum_{p\le i-1} \Lambda_p \mu_p^{1 - \frac{n}{2}} \\
& \lesssim (a_i + \sum_{p\ge i} \Lambda_p) W_i(x) +  \mu_{i-1}^{1- \frac{n}{2}} \left( \frac{\mu_i}{\theta_i(x)}\right)^2 + \frac{\theta_{i+1}(x)^2}{\mu_{i}^2} W_{i+1}(x) \mathds{1}_{\{ \theta_{i+1}(x) \le \mu_{i}\}} \\
&+  \sum_{p\le i-1} \Lambda_p \mu_p^{1 - \frac{n}{2}}  ,\\
\eal \]
which proves Claim \ref{claimder2} in the case $\theta_{i+1}(x) \le \mu_i$. We used here that $\mu_i^{1 - \frac{n}{2}} \lesssim W_i(x)$ when $x \in B_i$ and $\theta_{i+1}(x) \le \mu_i$.

\medskip

Assume then that  $|y_\ve - z_{i+1}| \ge 2$. Let, for any $z \in B_0(3/2) \backslash B_0(2/3)$,
 \[ \tilde{\vp}_\ve(z) = \Big(\hpv  -  \sum \limits_{\underset{q =0 .. n }{p=1..k}}    \lambda_{p,q} \hat{Z}_{p,q}\Big) \big(\alpha_i z \big), \]
 where again $\alpha_i = |y_\ve - z_{i+1}|$. Remember that $|z_{i+1}| \le 1$. Letting $\tilde{g}_i(z) = \hat{g}_i\big(  \alpha_i z \big)$, and since $1+ |z_{i+1} - \alpha_i z| \gtrsim  1+\alpha_i \gtrsim 1$ for any $z \in B_0(3/2) \backslash B_0(2/3)$, it is then easily seen that $\tilde{\vp}_\ve$ satisfies
 \[ \bal
  \Big|  \triangle_{\tilde{g}_i}  \tilde{\vp}_\ve(z)+ O\big(\mu_i \mu_{i-1} |\tilde{\vp}_\ve(z)|\big) \Big| & \lesssim \frac{a_i}{(1+\alpha_i)^{n-2}} + \frac{\mu_i^{n-2}\mathds{1}_{\textrm{l.c.f}}  }{(1+\alpha_i)^{2}} + \frac{\ve \mu_i^2 }{(1+\alpha_i)^{n-4}}  \\
  & + \left( \frac{\mu_{i+1}}{\mu_{i}}\right)^{\pui} \frac{1}{(1+\alpha_i)^n} .
   \eal
 \]
 Scaling back to the original variables as before proves the Claim when $x \in B_i $ and $\theta_{i+1}(x) \ge \mu_i$, and concludes the proof of Claim \ref{claimder2}. 
 \end{proof}
 
 We now conclude the proof of \eqref{claim1}.

\medskip

\textbullet\ A straightforward application of Claim \ref{claimder2} yields, together with $(H2_{\ell+1})$ in the case where $\ell \le k-1$:

\ben \label{proplin10}
\bal
 &\int_{\frac32 \sqrt{\frac{\mu_{\ell-1}}{{\mu_\ell}}} \le |y| \le 2  \sqrt{\frac{\mu_{\ell-1}}{{\mu_\ell}}}}  \Big(\frac{\mu_{\ell}}{{\mu_{\ell-1}}} |\hpv(y)|  + \sqrt{\frac{\mu_{\ell}}{{\mu_{\ell-1}}}} |\nabla \hpv(y)| \Big) |x-y|^{2-n} dy \\
 & \lesssim \ve \mu_\ell^2 \sum_{j\ge \ell} a_j +  \sum_{j\le \ell-1} \left( \frac{\mu_\ell}{\mu_j} \right)^\pui\Lambda_j   +  \left( \frac{\mu_\ell}{\mu_{\ell-1}} \right)^\pui\frac{1}{(1+|x|)^2} \\
 & + \frac{\ve \mu_\ell^2 \big(1 +\sum_{j \le \ell-1} a_j\big)}{(1+|x|)^{n-4}}. \\
\eal
\een

If $(M,g)$ is locally conformally flat, estimates \eqref{proplin1} -- \eqref{proplin8} and \eqref{proplin10} together with \eqref{formulerep} conclude the proof of \eqref{claim1}.

\medskip

\textbullet \ It therefore remains to conclude the proof in the non-locally conformally flat case. In that case, using \eqref{defai}, $(H1_{\ell+1})$ and $(H2_{\ell+1})$ if $\ell\le k-1$, and Claim \ref{claimder2}, we have
\ben \label{proplin11}
\bal
& \int_{\hat{B}_\ell} \Big(\mu_\ell^3 |y|^2 |\nabla \hpv(y)| + \mu_\ell^2 |y|^2 |\nabla^2 \hpv(y)|\Big) |x-y|^{2-n} dy \lesssim \frac{\mu_\ell^2 (a_\ell + \Lambda_\ell)} {(1+|x|)^{n-4}} \\
 & +\ve \mu_\ell^2  \frac{\sum_{j \le \ell-1} a_j}{(1+|x|)^{n-2}} +   \big( \sum_{j \ge \ell+1} \mu_{j-1}^2 a_j\big) \hat{W}_{\ell+1}(x)  \\
 & + \Big( \frac{\mu_{\ell+1}}{\mu_\ell} + |x - z_{\ell+1}|\Big)^{2} \hat{W}_{\ell+1}(x) \mathds{1}_{\{  |x - z_{\ell+1}| \le 1 \}} + \frac{\ve \mu_\ell^2}{(1+|x|)^{n-4}}  .\eal
\een
Combining \eqref{proplin1}--\eqref{proplin9} with \eqref{proplin11} and \eqref{formulerep} then concludes the proof of \eqref{claim1} in the non-locally conformally flat case.
\end{proof}

\medskip

\textbf{Step $3$: Estimation of $| \Pi(\hpv)(x)|$.} We prove in this part the following

\begin{claim}
For all $0 \le j \le n$,
\[  \int_{\RR^n} \hpv V_j U_0^{2^*-2} dy = O( \ve \mu_\ell^2 \sum_{j\le \ell-1} a_j ) +  \left \{ \bal & O(\ve \mu_\ell^2 a_\ell) & (M,g) \textrm{ l.c.f. }  \\ & O(\mu_\ell^2 a_\ell) & \textrm{ otherwise } \eal \right. .\]
As a consequence, for any  $x \in \hat{B}_{\ell} \backslash \hat{B}_{\ell+1}$:
\ben  \label{claimproj}
 | \Pi(\hpv)(x)| \lesssim  \frac{\ve \mu_\ell^2 \big(1 + \sum_{j \le \ell} a_j\big)}{(1+|x|)^{n-4}} +  \frac{\mu_\ell^2 a_\ell}{(1+|x|)^{n-2}} \mathds{1}_{\textrm{n.l.c.f}}.
 \een
\end{claim}

\begin{proof}
Let $x \in \hat{B}_{\ell} \backslash \hat{B}_{\ell+1}$. By definition we have 
 \[ \Pi(\hpv)(x) = \sum_{j=0}^n \frac{1}{\Vert \nabla V_j \Vert_{L^2(\RR^n)}^2}\Big(  \int_{\RR^n} \hpv V_j U_0^{2^*-2} dy\Big) V_j(x), \]
where $U_0$ is as in \eqref{defB0}. Since $|V_j(y)| \lesssim U_0(y)$ we only need to prove the first part of the Claim to obtain \eqref{claimproj}. Let $0 \le j \le n$.  By definition of $\vp_{\ve}$ in Proposition \ref{proplinbase} we have $\vp_\ve \in  K_k^{\perp}$, which implies that 
\[ \int_M \vp_\ve \Big( \triangle_g + c_n S_g + \ve h \Big) Z_{\ell,j} dv_g = 0.\]
By \eqref{defai}, \eqref{defhpv}, by \eqref{lapZ0} and \eqref{lapZi} and since $\hat{Z}_{\ell,j} = V_j$ we have:
\ben \label{claimproj1} \bal
0 & =  \int_M \vp_\ve \Big( \triangle_g + c_n S_g + \ve h \Big) Z_{\ell,j} dv_g \\
 & = \int_{B_\ell} \vp_\ve \Big( \triangle_g + c_n S_g + \ve h \Big) Z_{\ell,j} dv_g + O\Big( \sum_{j\le \ell-1} \left( \frac{\mu_\ell}{\mu_j}\right)^\pui a_j\Big) \\
 & = \int_{\hat{B}_\ell} \hpv \Big( \triangle_{\hat{g}_\ell}(y) + \mu_\ell^2 c_n S_{\hat{g}_\ell}(y_\ell) \Big) V_j(y) dy + O(\ve \mu_\ell^2 \sum_{j \le \ell} a_j), \\
  \eal \een
  where $y_\ell = \exp_{\xi_\ell}^{g_{\xi_\ell}}(\mu_\ell y)$. If $(M,g)$ is locally conformally flat we have 
  \[ \Big( \triangle_{\hat{g}_\ell}(y) + \mu_\ell^2 c_n S_{\hat{g}_\ell}(y_\ell) \Big) V_j(y)  = (2^*-1) U_0^{2^*-2} V_j, \]
  so that, with \eqref{growthmuell} and since $\hpv$ is supported in $\hat{B}_\ell$,
    \[ \int_{\RR^n}  \hpv U_0^{2^*-2} V_j dy  = O \Big( \ve \mu_\ell^2 \big( 1 + \sum_{j \le \ell} a_j \big) \Big). \]
If $(M,g)$ is not locally conformally flat we have 
  \ben \label{step31}
   \Big( \triangle_{\hat{g}_\ell}(y) + \mu_\ell^2 c_n S_{\hat{g}_\ell}(y_\ell) \Big) V_j(y)  =  (2^*-1) U_0^{2^*-2} V_j + O(\mu_\ell^2 (1+|y|)^{2-n}), 
   \een
  so that 
  \[  \int_{\hat{B}_\ell} \hpv \Big( \triangle_{\hat{g}_\ell}(y) + \mu_\ell^2 c_n S_{\hat{g}_\ell}(y_\ell) \Big) V_j(y) dy = \int_{\hat{B}_\ell} \hpv U_0^{2^*-2} V_j dy + O(\mu_\ell^2 a_\ell). \]
  Combining with \eqref{claimproj1} concludes the proof of \eqref{claimproj}.
\end{proof}

\medskip

\textbf{Step $4$: improvement of \eqref{claim1} into an optimal estimate.} We prove in this step the following improvement of \eqref{claim1}:

\begin{claim} 
For any $x \in \hat{B}_\ell \backslash \hat{B}_{\ell+1}$ there holds
\ben \label{claim1bis}
\bal
 | \hpv(x)| & \lesssim \ve \mu_\ell^2   \sum_{j \ge \ell} a_j +  \sum_{j\le \ell-1} \left( \frac{\mu_\ell}{\mu_j} \right)^\pui\Lambda_j  +  \frac{\ve \mu_\ell^2\big(1+  \sum_{j \le \ell} a_j \big)}{(1+|x|)^{n-4}}\\
 & + \frac{ \Lambda_\ell}{(1+|x|)^{n-2}} + \left(  \frac{\mu_\ell}{\mu_{\ell-1}} \right)^\pui  \frac{1}{(1+|x|)^2} \\
 &+   \sum_{j \ge \ell+1} \mu_{j-1}^2 a_j \hat{W}_j  +   \frac{\mu_\ell^4 a_\ell}{(1+|x|)^{n-6}} \mathds{1}_{\textrm{n.l.c.f}}\\
 & +  (1 +a_\ell)  \Big( \frac{\mu_{\ell+1}}{\mu_\ell} + |x - z_{\ell+1}|\Big)^{2} \hat{W}_{\ell+1}(x) \mathds{1}_{\{  |x - z_{\ell+1}| \le 2 \}} . \\
 \eal 
\een
\end{claim} 

\begin{proof}
First, in the locally conformally flat case, \eqref{claim1bis} follows from combining \eqref{claim1} with \eqref{claimproj}. We therefore assume that $(M,g)$ is not locally conformally flat. Combining \eqref{claim1} and \eqref{claimproj} we have, for any $x \in \hat{B}_\ell \backslash \hat{B}_{\ell+1}$:
\ben \label{step41}
\bal
 | \hpv(x)|  & \lesssim \ve \mu_\ell^2  \sum_{j \ge \ell} a_j +  \sum_{j\le \ell-1} \left( \frac{\mu_\ell}{\mu_j} \right)^\pui\Lambda_j  +  \frac{\ve \mu_\ell^2\big(1+  \sum_{j \le \ell} a_j \big)}{(1+|x|)^{n-4}}\\
 & + \frac{ \Lambda_\ell}{(1+|x|)^{n-2}} + \sum_{j \ge \ell+1} \mu_{j-1}^2 a_j \hat{W}_j + \frac{\mu_\ell^2 a_\ell}{(1+|x|)^{n-4}}  \\
 & +  \left(  \frac{\mu_\ell}{\mu_{\ell-1}} \right)^\pui  \frac{1}{(1+|x|)^2}  \\
& +  (1 +a_\ell)  \Big( \frac{\mu_{\ell+1}}{\mu_\ell} + |x - z_{\ell+1}|\Big)^{2} \hat{W}_{\ell+1}(x) \mathds{1}_{\{  |x - z_{\ell+1}| \le 2 \}}.
  \eal
\een
With \eqref{step41} we can now improve \eqref{claimproj}. Using \eqref{step31} together with \eqref{step41} now yields
\[ \bal
  \Bigg|  \int_{\hat{B}_\ell} \hpv & \Big( \triangle_{\hat{g}_\ell}(y) + \mu_\ell^2 c_n S_{\hat{g}_\ell}(y_\ell) \Big) V_j(y) dy -   \int_{\hat{B}_\ell} \hpv U_0^{2^*-2} V_j dy  \Bigg| \\
  & \lesssim  \ve \mu_\ell^2  \sum_{j =1}^k a_j   + \Lambda_\ell +  \sum_{j\le \ell-1} \left( \frac{\mu_\ell}{\mu_j} \right)^\pui\Lambda_j ,
  \eal \] 
  where we used \eqref{growthmuell} to claim that in the non-locally conformally flat case we have $\mu_\ell^2 \lesssim \ve$ and thus $\mu_\ell^4 \lesssim \ve \mu_\ell^2$ for all $\ell \ge1$. The other steps in the proof of \eqref{claimproj} are left unchanged and we obtain now
\ben  \label{claimprojbis}
 | \Pi(\hpv)(x)| \lesssim \ve \mu_\ell^2 \sum_{j \ge \ell} a_j +  \sum_{j\le \ell-1} \left( \frac{\mu_\ell}{\mu_j} \right)^\pui\Lambda_j +  \frac{\Lambda_\ell + \ve\mu_\ell^2\sum_{j \le \ell-1} a_j}{(1+|x|)^{n-2}}.
 \een
Going back to the proof of \eqref{claim1} it is easily seen that, when $(M,g)$ is not locally conformally flat, the term $ \mu_\ell^2 a_\ell(1+|x|)^{4-n}$ comes from the contribution of  \eqref{proplin11}. We are now going to use \eqref{step41} to compute \eqref{proplin11} again. First, a straightforward adaptation of the proof of Claim \ref{claimder2}, using now \eqref{step41} to control $\vp_\ve$, yields the following improved estimate on the derivatives of $\vp_\ve$: for all $x \in C_\ell$, 
\[\bal 
|\vp_{\ve}(x)| &+ \tilde{\theta}_\ell(x) |\nabla \vp_{\ve}(x)| + \tilde{\theta}_\ell(x)^2 |\nabla^2 \vp_{\ve}(x)|  \\
& \lesssim\Lambda_\ell W_\ell(x) + \mu_{\ell-1}^{1- \frac{n}{2}} \big( 1 + \sum_{j \ge \ell+1} a_j \big) \left( \frac{\mu_\ell}{\theta_\ell(x)}\right)^2 + a_\ell \theta_\ell^2 W_\ell(x)  \\
& +(1+ a_\ell) \frac{\theta_{\ell+1}(x)^2}{\mu_{\ell}^2} W_{\ell+1}(x) \mathds{1}_{\{ \theta_{\ell+1}(x) \le \mu_{\ell}\}}+   \sum_{p\le \ell-1} \Lambda_p \mu_p^{1 - \frac{n}{2}}, 
\eal \]
where $C_\ell$ is as in \eqref{defCi} and where we have again let $\tilde{\theta_\ell}(x) = \min(\theta_{\ell+1}(x), \theta_\ell(x))$. The latter estimate now yields, together with $(H1_{\ell+1})$ when $\ell\le k-1$:
\ben  \label{proplin11bis}
\bal
& \int_{\hat{B}_\ell} \Big(\mu_\ell^3 |y|^2 |\nabla \hpv(y)| + \mu_\ell^2 |y|^2 |\nabla^2 \hpv(y)|\Big) |x-y|^{2-n} dy \lesssim  \ve \mu_\ell^2  \sum_{j \ge \ell} a_j \\
 & +  \frac{\Lambda_\ell}{(1+|x|)^{n-2}} + \frac{\ve \mu_\ell^2(1 + \sum_{j \le \ell} a_j)}{(1+|x|)^{n-4}} +   \big( \sum_{j \ge \ell+1} \mu_{j-1}^2 a_j\big) \hat{W}_{\ell+1}(x) 
 \\
 & + \Big( \frac{\mu_{\ell+1}}{\mu_\ell} + |x - z_{\ell+1}|\Big)^{2} \hat{W}_{\ell+1}(x) \mathds{1}_{\{  |x - z_{\ell+1}| \le 1 \}}  +  \frac{\mu_\ell^4 a_\ell}{(1+|x|)^{n-6}} \mathds{1}_{\textrm{n.l.c.f}}.
  \eal
\een
Combining \eqref{proplin1} -- \eqref{proplin9} with \eqref{claimprojbis} and  \eqref{proplin11bis} finally proves \eqref{claim1bis}.
\end{proof}

\medskip

\textbf{Step $5$: estimation of $\Lambda_\ell$ and conclusion.} Let $j \in \{0, \dots, n\}$ be fixed. We integrate \eqref{eqlin} against $Z_{\ell,j}$: using \eqref{growthmuell}, \eqref{psZi}, \eqref{psZik}, \eqref{hypf} and \eqref{estlambdaij} we obtain
\[\bal
 |\lambda_{\ell,j}| + o(\Lambda_\ell)  & = O (\ve \mu_\ell^2)  + \int_M \vp_\ve \Big(\triangle_g + c_n S_g + \ve h - f'\big( \sum_{j=1}^k \kappa_j W_j \big) \Big)Z_{\ell,j} dv_g \\
 & = O (\ve \mu_\ell^2)  + O \Big( \sum_{j \le \ell-1} \left( \frac{\mu_\ell}{\mu_j} \right)^\pui a_j +  \sum_{j \ge \ell+1} \left( \frac{\mu_j}{\mu_\ell} \right)^\pui a_j  \Big)  \\
 & + \int_{B_\ell\backslash B_{\ell+1}} \vp_\ve \Big(\triangle_g + c_n S_g + \ve h - f'\big( \sum_{j=1}^k \kappa_j W_j \big) \Big)Z_{\ell,j} dv_g \\
 & = O\big(\ve\mu_\ell^2 \big(1 + \sum_{j=1}^k a_j \big)\big) \\ 
 & +  \int_{B_\ell\backslash B_{\ell+1}} \vp_\ve \Big(\triangle_g + c_n S_g + \ve h - (2^*-1)W_\ell^{2^*-2} \Big)Z_{\ell,j} dv_g 
 \eal \]
 where we used \eqref{diffbulle} and where we also used \eqref{defmu} to claim that $\mu_{\ell+1}/\mu_\ell \lesssim \mu_\ell/\mu_{\ell-1}$ for $1 \le \ell \le k$. Using \eqref{lapZ0} and \eqref{lapZi} and controlling $\hpv$ by \eqref{claim1bis} finally yields, after straightforward computations:
 \[  |\lambda_{\ell,j}| + o(\Lambda_\ell)   = O\big(\ve\mu_\ell^2 \big(1 + \sum_{j=1}^k  a_j \big)\big) . \] 
 By \eqref{Lambdai} this yields in turn
 \ben \label{H2rec}
 \Lambda_\ell \lesssim \ve \mu_\ell^2 \Big( 1 +  \sum_{j=1}^k a_j \Big), 
 \een
 which proves that $(H2_{\ell})$ defined by \eqref{H2} holds true. By plugging \eqref{H2rec} into \eqref{claim1bis} we obtain that, for any $x \in \hat{B}_\ell \backslash \hat{B}_{\ell+1}$ and for any $1 \le \ell \le k$:
 \ben \label{quasiH1}\bal
 | \hpv(x)| & \lesssim \ve \mu_\ell^2   \sum_{j \ge \ell} a_j +  \sum_{j\le \ell-1} \left( \frac{\mu_\ell}{\mu_j} \right)^\pui\Lambda_j +  \frac{\ve \mu_\ell^2\big( 1+\sum_{j \le \ell} a_j \big)}{(1+|x|)^{n-4}}
\\
& +   \left( \frac{\mu_\ell}{\mu_{\ell-1}} \right)^\pui \frac{1}{(1+|x|)^{n-2}}  +  \frac{ \mu_\ell^4 a_\ell}{(1+|x|)^{n-6}} \mathds{1}_{n.l.c.f} \\
& + \sum_{j \ge \ell+1} \mu_{j-1}^2 a_j \hat{W}_j \\
& +  (1 +a_\ell)  \Big( \frac{\mu_{\ell+1}}{\mu_\ell} + |x - z_{\ell+1}|\Big)^{2} \hat{W}_{\ell+1}(x) \mathds{1}_{\{  |x - z_{\ell+1}| \le 2 \}} .\\
 \eal \een
 
 \medskip
 
 We can now conclude the proof of Lemma \ref{lemmeinduction}. If $(M,g)$ is locally conformally flat, $(H1_\ell)$ follows from \eqref{quasiH1} by scaling back in the original variables.  In the non-locally conformally flat case we use \eqref{growthmuell} to claim that $\mu_\ell^2 \lesssim \ve$ and thus 
\[  \frac{\mu_\ell^4 a_\ell}{(1+|x|)^{n-6}} \lesssim \ve \mu_\ell^2 a_\ell. \]
Scaling back \eqref{quasiH1} to the original variables and using \eqref{growthmuell} we obtain:
\[ \bal
 |\vp_{\ve}(x)| & \lesssim \ve \mu_{\ell}^{3 - \frac{n}{2}} \sum_{j\ge \ell} a_j  +   \sum_{j\le \ell-1} \mu_j^{1 - \frac{n}{2}}\Lambda_j +  \mu_{\ell-1}^{1 - \frac{n}{2}} \left( \frac{\mu_{\ell}}{\theta_{\ell}(x)} \right)^2 \\
 &+ \sum_{j\ge \ell+1}\mu_{j-1}^2 a_j W_j (x) + \ve \big(1 + \sum_{j\le \ell} a_j \big) \theta_\ell^2(x) W_\ell(x)  \\
 & +   (1+a_\ell)\frac{\theta_{\ell+1}(x)^2}{\mu_{\ell}^2} W_{\ell+1}(x) \mathds{1}_{\{ \theta_{\ell+1}(x) \le \mu_{\ell}\}}, \\
  \eal
\]
which proves that \eqref{H1} holds true. Since in the case $\ell = k$ we did not assume anything, this shows that $(H_k)$ is true. The general case for $1 \le \ell \le k$ follows by induction.
\end{proof}

\subsection{Proof of Proposition \ref{proplin}, Part 2: end of the proof}

The last step of the proof of Proposition \ref{proplin} is an inductive estimation of the $a_i$.

\begin{claim}
For any $2 \le \ell \le k$ we have 
\ben \label{estal}
 a_\ell \lesssim \Big(\frac{\mu_\ell}{\mu_{\ell-1}} +\ve \mu_{\ell-1}^{2} \Big) \Big( 1 + \sum_{j=1}^{\ell-1} a_j \Big).\een
\end{claim}

\begin{proof}
We prove \eqref{estal} inductively using \eqref{H1} and \eqref{H2}. Let $\ell \ge 2$. If $\ell =k$ we do not assume anything, while if $\ell\le k-1$ we assume that \eqref{estal} is true for any $p \ge \ell+1$. For $x \in M$ let $G_\ve(x,\cdot)$ denote the Green's function for $\triangle_g + c_n S_g + \ve h$ in $M$. Standard properties of $G_\ve$ (Robert \cite{RobDirichlet}) ensure that 
\[ G_{\ve}(x,y) \lesssim d_g(x,y)^{2-n} \quad \textrm{ for all } x \neq y. \]
We write a representation formula for $\vp_\ve$ in $M$: with \eqref{eqlin} it writes as 
\ben \label{formulerepfin} \bal
\Big| \vp_\ve(x) &-  \sum \limits_{\underset{j =0 .. n }{i=1..k}}    \lambda_{i,j} Z_{i,j}(x)  \Big| \\
& \lesssim \int_M d_g(x,\cdot)^{2-n} |k_\ve| dv_g  + \sum_{i=2}^k \int_{B_i \backslash B_{i+1}} d_g(x, \cdot)^{2-n} W_i^{2^*-2} |\vp_{\ve}| dv_g \\
 & + \int_{M \backslash B_{g_{\xi_1}}(\xi_1, r_0)} d_g(x,\cdot)^{2-n} W_1^{2^*-2} |\vp_\ve| dv_g \\
 & + \int_{B_{g_{\xi_1}}(\xi_1, r_0) \backslash B_2} d_g(x,\cdot)^{2-n} W_1^{2^*-2} |\vp_\ve| dv_g ,\\
\eal \een
where $r_0$ is as in \eqref{defr0}. Assume now that $x \in C_\ell$, where $C_{\ell}$ is defined in \eqref{defCi}. First, straightforward computations using \eqref{growthmuell}, \eqref{rapportsbulles}, \eqref{hypf}, \eqref{tech1} and \eqref{tech2} below show that
\[ \bal \int_M d_g(x,\cdot)^{2-n} |k_\ve| dv_g& \lesssim \mu_{\ell-1}^{ 1- \frac{n}{2}} \left( \frac{\mu_\ell}{\theta_\ell(x)}\right)^2  + \ve \mu_{\ell-1}^{3 - \frac{n}{2}} \\
&+ \frac{\mu_\ell}{\mu_{\ell-1}} W_\ell(x) + \frac{\theta_{\ell+1}(x)^2}{\mu_\ell^2} W_{\ell+1}(x)\mathds{1}_{\theta_{\ell+1}(x) \le \mu_\ell}. \\
\eal \]
Independently, using the definition of $a_1$ in \eqref{defai}, \eqref{tech1} and since $x \in \mathcal{C}_\ell$, $\ell \ge 2$, we have
\[ \int_{M \backslash B_{g_{\xi_1}}(\xi_1, r_0)} d_g(x,\cdot)^{2-n} W_1^{2^*-2} |\vp_\ve| dv_g \lesssim a_1 \mu_1^{\frac{n+2}{2}} \lesssim a_1 \ve \mu_1^{3 - \frac{n}{2}}.\]
Using \eqref{rapportsbulles}, \eqref{H1},  \eqref{tech1} and \eqref{tech2} we similarly obtain, since $x \in C_\ell$ and $\ell \ge2$, that:
\[ \bal \sum_{i=2}^k & \int_{B_i \backslash B_{i+1}}  d_g(x, \cdot)^{2-n} W_i^{2^*-2} |\vp_{\ve}| dv_g  + \int_{B_{g_{\xi_1}}(\xi_1, r_0) \backslash B_2} d_g(x,\cdot)^{2-n} W_1^{2^*-2} |\vp_\ve| dv_g \\
&  \lesssim \Big(1 + \sum_{j \ge \ell} a_j \Big) \Bigg(  \mu_{\ell-1}^{ 1- \frac{n}{2}} \left( \frac{\mu_\ell}{\theta_\ell(x)}\right)^2 + \ve \mu_{\ell-1}^{3 - \frac{n}{2}}  + \frac{\mu_\ell}{\mu_{\ell-1}}W_\ell(x) \\
& + \frac{\theta_{\ell+1}(x)^2}{\mu_\ell^2} W_{\ell+1}(x)\mathds{1}_{\theta_{\ell+1}(x) \le \mu_\ell} \Bigg) .
\eal \]
Using \eqref{rapportsbulles} and \eqref{H2} we also obtain that
\[ \bal \Big|  \sum \limits_{\underset{j =0 .. n }{i=1..k}}    \lambda_{ij} Z_{i,j}(x)  \Big| \lesssim   \Big( 1 +  \sum_{j=1}^k a_j \Big)  \Big( \frac{\mu_\ell}{\mu_{\ell-1}}W_{\ell}(x) + \ve \mu_{\ell-1}^{3 - \frac{n}{2}}
 \Big).
\eal \]
Remark now that there holds, for any $x \in C_{\ell}$ and since $\ell\ge2$ : 
\ben \label{estutile}
  \mu_{\ell-1}^{ 1- \frac{n}{2}} \left( \frac{\mu_\ell}{\theta_\ell(x)}\right)^2 + \frac{\theta_{\ell+1}(x)^2}{\mu_\ell^2} W_{\ell+1}(x)\mathds{1}_{\theta_{\ell+1}(x) \le \mu_\ell}  \lesssim \frac{\mu_\ell}{\mu_{\ell-1}} W_{\ell}(x). \een
Plugging the latter estimates in \eqref{formulerepfin} finally gives 
\[ \bal
 |\vp_\ve(x)| &\lesssim  \Big( 1 + \sum_{j=1}^k a_j \Big) \Bigg(   \frac{\mu_\ell}{\mu_{\ell-1}} W_\ell(x)  +  \ve \mu_{\ell-1}^{3 - \frac{n}{2}}
 \Bigg).
 \eal \]
If $\ell=k$ we do not assume anything, while if $\ell \le k-1$  we use in addition \eqref{estal} inductively. Together with the latter estimate we obtain in each case
\[   |\vp_\ve(x)| \lesssim \Big( 1 + \sum_{j=1}^{\ell} a_j \Big) \Bigg(   \frac{\mu_\ell}{\mu_{\ell-1}} W_\ell(x)  + \ve \mu_{\ell-1}^{3 - \frac{n}{2}}
\Bigg). \]
 Let now $x_\ell \in C_{\ell}$ be a point where $a_\ell$ is attained (recall that  $a_\ell$ is defined in \eqref{defai}). Then $|\vp_\ve(x_\ell)| = a_\ell W_\ell(x_\ell)$. Since $W_\ell(x) \gtrsim \mu_{\ell-1}^{1- \frac{n}{2}}$ on $C_{\ell}$  the latter estimate gives, together with \eqref{growthmuell},
 \[ a_\ell  \lesssim \Big(  \frac{\mu_\ell}{\mu_{\ell-1}} + \ve \mu_{\ell-1}^2 \Big)  \Big( 1 + \sum_{j=1}^{\ell-1} a_j \Big) + o(a_\ell), \]
which proves \eqref{estal} for $\ell \ge 2$. 
 \end{proof}

  It remains to estimate $a_1$ to conclude the proof. This is the content of the following Claim:
  
  \begin{claim} \label{claima1}
  We have 
  \[ a_1 \lesssim \min \Big( \frac{\mu_2}{\mu_1}, \ve, \mu_1^2 \Big). \]
  In particular, $a_1 = o(1)$ as $\ve \to 0$.
    \end{claim} 
    
\begin{proof}
Estimate $(H1_1)$ is for the moment still too rough on $M \backslash B_2$, because of the constant term $\ve \mu_1^{3 - \frac{n}{2}}$. As a first thing, we improve the precision of the estimate. Let $x \in M \backslash B_2$. We use again the representation formula \eqref{formulerepfin} and we estimate again each term. Since $x \in M \backslash B_2$ we have, with \eqref{tech1} and \eqref{tech2}
 \[ \bal \int_M d_g(x,\cdot)^{2-n} |k_\ve| dv_g& \lesssim \left \{ \bal & \mu_1^{\pui} \left( \frac{\mu_1}{\theta_1(x) } \right)^2 & \textrm{ if } (M,g) \textrm{ is l.c.f } \\ 
& \ve \theta_1(x)^2 W_1(x) & \textrm{ otherwise } \\ \eal \right \}  \\
&  + \ve\mu_2^2   W_1  + \frac{\theta_{2}(x)^2}{\mu_1^2} W_{2}(x)\mathds{1}_{\theta_{2}(x) \le \mu_1}.
\eal \]
By definition of $a_1$ we also have
\[ \int_{M \backslash B_{g_{\xi_1}}(\xi_1, r_0)} d_g(x,\cdot)^{2-n} W_1^{2^*-2} |\vp_\ve| dv_g \lesssim a_1 \mu_1^{\frac{n+2}{2}} . \]
Using \eqref{rapportsbulles}, \eqref{H1}, \eqref{estal}, \eqref{tech1} and \eqref{tech2} we similarly obtain, since $x \in M \backslash B_2$, that:
\ben \label{problematicB1}
 \bal &\sum_{i=2}^k  \int_{B_i \backslash B_{i+1}}  d_g(x, \cdot)^{2-n} W_i^{2^*-2} |\vp_{\ve}| dv_g \\
 &+ \int_{B_{g_{\xi_1}}(\xi_1, r_0) \backslash B_2} d_g(x,\cdot)^{2-n} W_1^{2^*-2} |\vp_\ve| dv_g \\
& \lesssim \Big(1 + a_1\Big) \Bigg[ \ve \mu_1^{3 - \frac{n}{2}} \left( \frac{\mu_1}{\theta_1(x) }\right)^2  + \ve\mu_2^2   W_1  + \frac{\theta_{2}(x)^2}{\mu_1^2} W_{2}(x)\mathds{1}_{\theta_{2}(x) \le \mu_1} \Bigg] .
\eal \een
By \eqref{H2}, \eqref{estal} and \eqref{growthmuell} we now have 
\[ \bal \Big|  \sum \limits_{\underset{j =0 .. n }{i=1..k}}    \lambda_{ij} Z_{i,j}(x)  \Big| \lesssim \ve \mu_1^2 ( 1 + a_1 ) W_{1}(x).
\eal \]
Combining the latter estimates into  \eqref{formulerepfin} then yields
\ben \label{betterB1}
 \bal |\vp_\ve(x) |& \lesssim  \left \{ \bal & \mu_1^{\pui} \left( \frac{\mu_1}{\theta_1(x) } \right)^2 & \textrm{ if } (M,g) \textrm{ is l.c.f } \\ 
& \ve \theta_1(x)^2 W_1(x) & \textrm{ otherwise } \\ \eal \right \}  \\
& + \Big(1 + a_1\Big) \Bigg[ \ve \mu_1^{3 - \frac{n}{2}} \left( \frac{\mu_1}{\theta_1(x) }\right)^2  + \frac{\theta_{2}(x)^2}{\mu_1^2} W_{2}(x)\mathds{1}_{\theta_{2}(x) \le \mu_1} \Bigg] .
\eal \een
It is easily seen that the limiting term $ \ve \mu_1^{3 - \frac{n}{2}} \left( \frac{\mu_1}{\theta_1(x) }\right)^2$ in \eqref{problematicB1} is due to the second integral in the left-hand side of \eqref{problematicB1} (the contribution over $B_{g_{\xi_1}}(\xi_1, r_0) \backslash B_2$). We now use \eqref{betterB1} to improve the precision of this term. After a finite number of iterations using \eqref{tech1}, \eqref{problematicB1} becomes
\ben \label{estconclu}
 \bal |\vp_\ve(x) |& \lesssim  (1+a_1) \left \{ \bal & \mu_1^{\pui} \left( \frac{\mu_1}{\theta_1(x) } \right)^2 & \textrm{ if } (M,g) \textrm{ is l.c.f } \\ 
& \ve \theta_1(x)^2 W_1(x) & \textrm{ otherwise } \\ \eal \right \}  \\
& + \Big(1 + a_1\Big)\frac{\theta_{2}(x)^2}{\mu_1^2} W_{2}(x)\mathds{1}_{\theta_{2}(x) \le \mu_1} .
\eal \een
As before, since $x \in M \backslash B_2$, we have 
\[ \frac{\theta_{2}(x)^2}{\mu_1^2} W_{2}(x)\mathds{1}_{\theta_{2}(x) \le \mu_1}  \lesssim \frac{\mu_2}{\mu_1} W_1(x). \]
It remains to apply \eqref{estconclu} to a point $x_1 \in M \backslash B_2$ where $a_1$ is attained. We obtain:
\[ a_1 \lesssim o(a_1) + \frac{\mu_2}{\mu_1} + \left \{ \bal & \mu_1^2 & \textrm{ if } (M,g) \textrm{ is l.c.f. } \\ & \ve & \textrm{ otherwise } \eal \right. ,  \]
which concludes the proof of the Claim. 
\end{proof}

\medskip

\begin{proof}[End of the proof of Proposition \ref{proplin}]
We are now ready to conclude the proof of  Proposition \ref{proplin}. First, Claim \ref{claima1} shows that $a_1 = o(1)$. Using this in \eqref{estconclu} proves  \eqref{estvp1}. Then, using Claim \ref{claima1} together with \eqref{estal} shows inductively that 
\ben \label{estal2}
  a_\ell \lesssim \frac{\mu_\ell}{\mu_{\ell-1}} + \ve \mu_{\ell-1}^2
  \een
for all $\ell \ge 2$, so that in particular $a_\ell = o(1)$ for all $1\le \ell \le k$.  With \eqref{H2}, this proves \eqref{vpl}. It remains to prove that \eqref{estvpi} is true. Let $2 \le \ell \le k$ be fixed. By \eqref{estal2} and \eqref{growthmuell} we have
\[ \bal \mu_{\ell-1}^{1 - \frac{n}{2}} \sum_{j \ge \ell} a_j & \lesssim \ve \mu_{\ell-1}^{ 3 - \frac{n}{2}} + \frac{\mu_\ell}{\mu_{\ell-1}} \mu_{\ell-1}^{1 - \frac{n}{2}} \\
& \lesssim \ve \mu_{\ell-1}^{ 3 - \frac{n}{2}} +  \mu_{\ell-1}^{1 - \frac{n}{2}}\left( \frac{\mu_\ell}{\theta_\ell(x)} \right)^2
\eal \]
for any $x \in B_\ell \backslash B_{\ell+1}$. Similarly, \eqref{vpl} now shows that 
\[ \sum_{j \le \ell-1} \mu_j^{1 - \frac{n}{2}} \Lambda_j \lesssim \ve \mu_{\ell-1}^{ 3 - \frac{n}{2}}. \]
With \eqref{estal2} we also have, for any $p \ge \ell+1$ and for any $x \in B_\ell \backslash B_{\ell+1}$:
\[ \bal
\mu_{p-1}^2 a_p W_p(x) & \lesssim \left( \frac{\mu_{\ell+1}}{\mu_\ell} \right)^{\pui} W_{\ell}(x) + \frac{\theta_{\ell+1}(x)^2}{\mu_\ell^2} W_{\ell+1}(x) \mathds{1}_{\theta_{\ell+1} \le \mu_\ell} \\
& \lesssim \mu_{\ell-1}^{1 - \frac{n}{2}} \left( \frac{\mu_\ell}{\theta_\ell(x)} \right)^2 + \frac{\theta_{\ell+1}(x)^2}{\mu_\ell^2} W_{\ell+1}(x) \mathds{1}_{\theta_{\ell+1} \le \mu_\ell}. 
\eal  \]
Finally, since $\ell \ge2$, we have 
\[  \ve  \theta_\ell^2(x) W_\ell(x)  \lesssim  \mu_{\ell-1}^{ 1 - \frac{n}{2}}  \left( \frac{\mu_\ell}{\theta_\ell(x)} \right)^{2}. \]
Combining these estimates in \eqref{H1} finally prove \eqref{estvpi} and concludes the proof of Proposition \ref{proplin}.
\end{proof}

\section{The nonlinear problem} \label{nonlin}

\subsection{Solving \eqref{eqY} up to kernel elements}

We perform in this section a nonlinear fixed-point argument in strong spaces and solve \eqref{eqY} up to kernel elements.  Let $k \ge 1$ be an integer, $(\kappa_i)_{1 \le i \le n} \in \{-1,1 \}^k$ and let $\ve_0 >0$ be as in Proposition \ref{proplinbase}. Let $\big(t_{1}, \xi_{1}, (t_{i}, z_{i})_{2 \le i \le k} \big)\in S_k$, where $S_k$ is defined in \eqref{defA}. For $0< \ve\le \ve_0$ we define a function $\Psi_\ve$ in $M$ as follows:
\ben \label{CapPsi}
\bal
\bullet\  \Psi_\ve(x) & =  \mu_{\ell-1}^{1 - \frac{n}{2}} \left( \frac{\mu_{\ell}}{\theta_{\ell}(x)} \right)^2 + \ve \mu_{\ell-1}^{3 - \frac{n}{2}}  + \frac{\theta_{\ell+1}(x)^2}{\mu_{\ell}^2} W_{\ell+1}(x)  \mathds{1}_{\{ \theta_{\ell+1}(x) \le \mu_{\ell}\}}  \\
&  \textrm{ if } x\in B_\ell \backslash B_{\ell+1}, \quad \textrm{ for } \ell \ge 2, \\
\bullet\  \Psi_\ve(x) & =  \left \{
\bal
&   \frac{\mu_{1}^{\frac{n+2}{2}}}{\theta_1(x)^2}   & \textrm{ if } (M,g) \textrm{ is l.c.f. around } \xi_1 \\
&  \frac{\mu_1^{\frac{n+2}{2}}}{\theta_1(x)^{n-4}} &  \textrm{ if } (M,g) \textrm{ is not l.c.f.}  \\ \eal
 \right \} 
  +  \frac{\theta_{2}(x)^2}{\mu_{1}^2} W_{2}(x) \mathds{1}_{\{ \theta_{2}(x) \le \mu_{1}\}} \\
 &\textrm{ if } x\in M \backslash B_{2},
\eal
\een
where $W_\ell$ is given by \eqref{defWi}. Let $C >0$ to be chosen later. We define 
\ben \label{set1}
E_{\ve} ^C = \Big \{ \vp \in C^0(M) \textrm{ such that } |\vp(x)| \le C \Psi_{\ve}(x) \textrm{ for all } x \in M\Big \}
\een
endowed with the following norm:
\ben \label{norm1}
\Vert \vp \Vert_{*}:=  \left \Vert \frac{\vp}{ \Psi_\ve} \right \Vert_{L^\infty(M)}  \quad \textrm{ for } \vp \in E^C.
\een
Note that $E_{\ve}^C$ depends on $\ve$ and on $\big(t_{1}, \xi_{1}, (t_{i}, z_{i})_{2 \le i \le k} \big)$. Since $\Psi_\ve >0$ in $M$, it is easily seen that $(E_\ve^C, \Vert \cdot \Vert_*)$ is a Banach space. Building on the linear theory of Proposition \ref{proplin} we perform a nonlinear fixed-point procedure in $E_\ve^C$:

\begin{prop} \label{propnonlin}
Let $k \ge 1$ be an integer, $(\kappa_i)_{1 \le i \le n} \in \{-1,1 \}^k$ and denote $f(u) = (u+)^{2^*-1}$ or $f(u) = |u|^{2^*-2}$. There exist $\ve_1 >0$ and $C >0$ such that the following holds. For any $\ve \in (0, \ve_1)$, any $h \in C^0(M)$ and for any  $\big(t_{1}, \xi_{1}, (t_{i}, z_{i})_{2 \le i \le k} \big)\in S_k$, there exists 
\[ \vp_\ve:= \vp_\ve \Big( (t_{1}, \xi_{1}, (t_{i}, z_{i})_{2 \le i \le k} \Big)  \in K_k^{\perp}\]
which is the unique solution in $K_k^{\perp} \cap E_{\ve}^C$ of 
\ben \label{eqnonlin}
 \Pi_{K_{k}^{\perp}}\Bigg(  \sum_{i=1}^k \kappa_i W_i + \vp_{\ve} - (\triangle_g + c_n S_g + \ve h)^{-1}\Big( f\Big( \sum_{i=1}^k \kappa_i W_i + \vp_\ve \Big) \Big) \Bigg) =0,
 \een
where $E_\ve^C$ is defined in \eqref{set1}. In addition, 
\[ \big( (t_{1}, \xi_{1}, (t_{i}, z_{i})_{2 \le i \le k} \big) \in S_k \mapsto \vp_\ve \in C^0(M) \]
is continuous for any fixed value of $\ve$.
\end{prop}
The main point of Proposition \ref{propnonlin} is that $\vp_\ve$ is constructed in the space $E_\ve^C$ -- and not only in $K_k^{\perp}$ -- and satisfies therefore
\ben \label{estoptvp}
 |\vp_\ve|\lesssim \Psi_\ve
\een  in $M$, where $\Psi_\ve$ is given by \eqref{CapPsi}. In particular, $\vp_\ve$ still satisfies the estimates \eqref{estvpi} and \eqref{estvp1}, which is another way of thinking of the nonlinear problem \eqref{eqnonlin} as a perturbation of the linear equation \eqref{eqlin}. Proposition \ref{propnonlin} therefore provides us with a canonical remainder $\vp_\ve$ with sharp pointwise estimate that precisely capture the pointwise interactions between the towering bubbling profiles $W_i$ on the whole of $M$. The information that we obtain in Proposition \ref{propnonlin} is thus much more precise than a mere energy estimate and will turn out to be crucial to prove Theorem \ref{maintheo} in Section \ref{expansionkernel}. Such a precision, of course, comes at the expense of proving sharp pointwise estimates for the linear equation as in Proposition \ref{proplin}. Remark that since $\vp_\ve \in E_\ve^C$ for some $C$, integrating \eqref{eqnonlin} against $\vp_\ve$ yields
\ben \label{vpveH1}
\Vert \vp_\ve \Vert_{H^1} = o(1)
\een as $\ve \to 0$, where the $H^1$ norm is given by \eqref{psH1}.

\begin{proof}
Let $k \ge 1$, $(\kappa_i)_{1 \le i \le n} \in \{-1,1 \}^k$, $0< \ve \le \ve_0 $ and let $\big(t_{1}, \xi_{1}, (t_{i}, z_{i})_{2 \le i \le k} \big) \in S_{k}$. In the following proof we will omit, for the sake of clarity, the dependency in $\ve$ and $\big(t_{1}, \xi_{1}, (t_{i}, z_{i})_{2 \le i \le k} \big) \in S_{k}$ in the quantities $W_i, Z_{i,j}, \dots$ appearing. The proof of Proposition \ref{propnonlin} follows the standard steps of a nonlinear fixed-point scheme, with the notable exception that the point-fixing is now achieved in the strong space $E_\ve^C$ defined by \eqref{set1}. 

\medskip

Let $C > 0$ be fixed (to be chosen later). For $\vp \in E_{\ve}^C$, as defined in \eqref{set1}, we define
$$T(\vp) =  \mathcal{L}_\ve\Big(- R_{\ve} + N(\vp) \Big), $$
where $ \mathcal{L}_{ \ve}$ is given by \eqref{defL} and where we have let
\be
 R_{\ve} = \big(\triangle_g + c_n S_g + \ve h\big)\Big( \sum_{i=1}^k \kappa_i W_i \Big) - f\Big(\sum_{i=1}^k \kappa_i W_i \Big) 
  \ee
and 
$$ N(\vp) = f\Big( \sum_{i=1}^k \kappa_i W_i  + \vp \Big) - f \Big( \sum_{i=1}^k \kappa_i W_i  \Big)- f' \Big(\sum_{i=1}^k \kappa_i W_i  \Big) \vp. $$
The mapping $T$ is well-defined since $\vp \in E_\ve^C$ and thus $- R_{\ve} + N(\vp)  \in C^0(M)$. By \eqref{defL}, $T(\vp)$ solves
\ben \label{eqT}
\bal
\big( \triangle_g + c_n S_g + \ve h\big)T(\vp) -& f'\Big(  \sum_{i=1}^k \kappa_i W_i  \Big) T(\vp)  = - R_{\ve} + N(\vp) \\
&+ \sum \limits_{\underset{j =0 .. n }{i=1..k}}    \lambda_{ij} ( \triangle_g   + c_n S_g + \ve h)Z_{i,j}
 \eal
\een
for some uniquely defined real numbers $(\lambda_{ij})$.  We show that $T$ is a contraction on $(E_\ve^C, \Vert \cdot \Vert_*)$ for $\ve$ small enough, for a suitable fixed value of $C$.  

\medskip

We first prove that $T$ stabilises $E_{\ve}^C$. By \eqref{CapPsi} one has 
\[\Psi_\ve \lesssim \sum_{i=1}^k W_i, \]
so that for any  $\vp \in E_{\ve}^C$ we have
\[ |N(\vp)| \lesssim \Big(\sum_{i=1}^k W_i\Big)^{2^*-3}  |\vp|^2.\]
Let $2 \le i \le k$ and $\vp \in E_\ve^C$. By \eqref{CapPsi} we have, for any $x  \in B_i \backslash B_{i+1}$,
\ben \label{nonlin11} \bal W_i(x)&^{2^*-3} |\vp(x)|^2 \\
& \lesssim C^2 W_i(x)^{2^*-3} \Bigg(\mu_{i-1}^{2-n} \left(  \frac{\mu_i}{\theta_i(x)} \right)^4 +\ve^2 \mu_{i-1}^{6-n} + \left( \frac{\theta_{i+1}(x)}{\mu_i}\right)^4 W_{i+1}^2 \mathds{1}_{\theta_{i+1}(x) \le \mu_i} \Bigg) \\
& \lesssim   C^2\left( \frac{\mu_{i}}{\mu_{i-1}} \right)^{2} \mu_{i-1}^{1 - \frac{n}{2}} W_i(x)^{2^*-2}  +C^2 \ve^2 \mu_{i-1}^{4} W_i(x)^{2^*-2} W_{i-1}(x)  \\
&+ C^2\left( \frac{\mu_{i+1}}{\mu_i} \right)^2 W_i(x)^{2^*-2} W_{i+1}(x),
\eal, \een
where we used \eqref{estutile} and \eqref{growthmuell}. Similarly for $x \in M \backslash B_2$, we get by \eqref{CapPsi}
\ben \label{nonlin22} \bal W_1(x)^{2^*-3}|\vp(x)|^2 & \lesssim C^2 \Bigg(\mu_1^4 \frac{\mu_{1}^{\frac{n+2}{2}}}{\theta_1(x)^4}   + \left( \frac{\mu_{2}}{\mu_1} \right)^2 W_1(x)^{2^*-2} W_{2}(x) \Bigg)\eal \een
when $(M,g)$ is locally conformally flat around $\xi_1$, and 
\ben \label{nonlin33} \bal W_1(x)^{2^*-3} |\vp(x)|^2 & \lesssim C^2 \Bigg( \mu_1^4 \ve W_1(x) + \left( \frac{\mu_{2}}{\mu_1} \right)^2 W_1(x)^{2^*-2} W_{2}(x) \Bigg)\eal \een
when $(M,g)$ is not conformally flat. The previous estimates show that there exists a positive constant $D_0 >0$ such that 
\ben \label{propnonlinX}
\frac{N(\vp)}{D_0 C^2\Big(\mu_1^4 +  \frac{\mu_2^2}{\mu_1^2}\Big)} 
\een
satisfies \eqref{hypf}, where $C$ is the constant in \eqref{set1}. Independently, by \eqref{ferr} we know that there exists a constant $D_1 >0$ such that $R_\ve/D_1$ also satisfies \eqref{hypf}. By \eqref{eqT} we can then apply Proposition \ref{proplin} which yields a positive constant $C_1 >0 $ independent of $\ve$ such that
$$ |T(\vp)(x)|  \le C_1 \Big(1 + C^2 \Big(\mu_1^4 +  \frac{\mu_2^2}{\mu_1^2} \Big)\Big)\Psi_\ve(x) $$
for any $\vp \in E_\ve^C$ and $x \in M$, where $\Psi_\ve$ is as in \eqref{CapPsi}, or in other words
\[ \Vert T(\vp) \Vert_* \le C_1 \Big(1 + C^2 \Big(\mu_1^4 + \frac{\mu_2^2}{\mu_1^2} \Big)\Big) \]
where the norm is defined in \eqref{norm1}. Choose $C = 2 C_1$. Since $\mu_1^4 + \frac{\mu_2^2}{\mu_1^2} \to 0$ as $\ve \to 0$, for $\ve$ small enough we then have $T(\vp) \in E_{\ve}^C$. 

\medskip

We now prove that $T$ is a contraction on $(E_{\ve}^C, \Vert \cdot \Vert)$ for this fixed value of $C$. If $\vp_1, \vp_2 \in E_{\ve}^C$, $T(\vp_1) - T(\vp_2)$ satisfies by assumption:
\ben \label{eqTcont}
\bal
\big( \triangle_g + c_n S_g + \ve h\big)\big(T(\vp_1) &- T(\vp_2) \big) - f' \Big( \sum_{i=1}^k \kappa_i W_i \Big)\big( T(\vp_1) - T(\vp_2) \big) \\
&   = N(\vp_1) - N(\vp_2) +  \sum \limits_{\underset{j =0 .. n }{i=1..k}}    \nu_{ij} ( \triangle_g   + c_n S_g + \ve h)Z_{i,j}
\eal
\een
for some real numbers $(\nu_{ij})$. By \eqref{set1}, \eqref{norm1}, the mean-value theorem and since $\vp \in E_\ve^C$ we have 
$$ \bal 
|N(\vp_1) - N(\vp_2)| & \lesssim \Big(\sum_{i=1}^k W_i \Big)^{2^*-3}\big( |\vp_1| + |\vp_2| \big)| \vp_1 - \vp_2 | \\
& \lesssim  \Big(\sum_{i=1}^k W_i \Big)^{2^*-3}  \Psi_\ve^2 \Vert \vp_1 - \vp_2 \Vert_{*}. 
\eal $$
In \eqref{nonlin11}, \eqref{nonlin22} and \eqref{nonlin33} we already showed that
\[ \frac{1}{ \Big(\mu_1^4 +  \frac{\mu_2^2}{\mu_1^2}\Big)} \Big(\sum_{i=1}^k W_i\Big)^{2^*-3} \Psi_\ve^2 \]
satisfies \eqref{hypf} up to some positive constant, and therefore 
\[ \frac{1}{D_2 \Big(\mu_1^4 +  \frac{\mu_2^2}{\mu_1^2}\Big)} \frac{N(\vp_1) - N(\vp_2)}{\Vert \vp_1 - \vp_2 \Vert_{*}} \]
satisfies \eqref{hypf} for some positive constant $D_2$ independent of $\ve$. With \eqref{eqTcont} we can again apply Proposition  \ref{proplin} and we obtain that
\[ \Vert T(\vp_1) - T(\vp_2) \Vert_* \le C_2 \Big(\mu_1^4 +  \frac{\mu_2^2}{\mu_1^2}\Big)\Vert \vp_1 - \vp_2 \Vert_{*}  \]
for some positive constant $C_2$ independent of $\ve$.
This shows that $T$ is a contraction on $E_\ve^C$ for $\ve$ small enough. Picard's fixed-point theorem then yields the existence of a unique fixed-point $\vp_{\ve}$ of $T$ in $E_{\ve}^C$ which is by \eqref{eqT}, equivalently, a solution of \eqref{eqnonlin}. 

\medskip

It remains to see that for a fixed $\ve$ the mapping
\[ \big( (t_{1}, \xi_{1}, (t_{i}, z_{i})_{2 \le i \le k} \big) \in S_k \mapsto \vp_\ve \in C^0(M) \]
is continuous. By construction $\vp_\ve$ satisfies \eqref{eqT} with $T(\vp_\ve) = \vp_\ve$. By \eqref{propnonlinX}, $N(\vp_\ve)$ satisfies \eqref{hypf} up to a constant, and therefore Proposition \ref{proplin} applies again: \eqref{vpl} in particular shows that the $\lambda_{ij}$ are uniformly bounded in $\ve$ and in the choice of $ \big( (t_{1}, \xi_{1}, (t_{i}, z_{i})_{2 \le i \le k} \big) \in S_k$.  The continuity of $(t_i, z_i)_{1 \le i \le \ell} \in S_{\ell} \mapsto  \vp_{\ve, \ell} \in C^0(M) $  now easily follows from the continuity of the $W_i$, $1 \le i \le k$, from \eqref{eqT} and from standard elliptic theory, and by the uniqueness property of $\vp_{\ve,\ell}$ in $K_k^{\perp} \cap E_\ve^C$.
\end{proof}

\subsection{Application to \emph{a priori} pointwise estimates}

The procedure developed in Propositions \ref{proplin} and \ref{propnonlin} can be useful when looking for \emph{a priori} estimates. We illustrate this on a toy-model example. Consider for instance $h \in C^0(M)$, a sequence $(\ve_k)$ of positive numbers converging to $0$ as $k \to + \infty$ and a sequence $(u_k)_k$ of \emph{positive} functions satisfying 
\ben \label{eqapriori}
 \Big( \triangle_g + c_n S_g + \ve_k h \Big) u_k = u_k^{2^*-1} 
 \een
in $M$. We assume for simplicity that the sequence $(u_k)_k$ blows-up with finite-energy and a zero weak-limit. By Struwe's celebrated compactness result \cite{Struwe} there exists $N \ge 1$ such that 
\[ \Vert u_k \Vert_{H^1} = N K_n^{-n} + o(1) \]
as $k \to + \infty$, where $K_n^{-n}$ is the energy of the standard bubble in $\RR^n$ and where the $H^1$ norm is given by \eqref{psH1}. Consider the manifold of the sums of $N$ bubbles in $M$:
\[ \mathcal{M}_N: = \Big \{ \sum_{i=1}^N W_{\mu_i, \xi_i}, (\mu_i, \xi_i)_{1 \le i \le N} \in (\RR_+^* \times M)^N \Big \}, \]
where $W_{\mu_i, \xi_i}$ is given by \eqref{defW}. It is a $N(n+1)$-dimensional closed manifold of $H^1(M)$. For any $k \ge 1$, we can thus let $(\mu_{i,k},\xi_{i,k})_{1 \le i \le N}$ be such that 
\[ W_k: = \sum_{i=1}^N W_{ \mu_{i,k},\xi_{i,k}} \]
attains the distance of $u_k$ to $\mathcal{M}_N$. Letting $\vp_k = u_k - W_k $ we therefore have
\[ \Vert \vp_k \Vert_{H^1} = \inf_{W \in \mathcal{N}} \Vert u_k - W \Vert_{H^1}. \]
As a consequence, $\vp_k$ is orthogonal to the tangent space to $\mathcal{M}_N $ at $(\mu_{i,k},\xi_{i,k})_{1 \le i \le N}$ which, for $k$ large enough, is given by
$$ \mathcal{K}_{N,k}: = \textrm{Span} \Big \{ \mathcal{Z}_{i,0}, \mathcal{Z}_{i,j}, \quad 1 \le i \le N, 0 \le j \le n \Big\}, $$
where we have let
$$  \mathcal{Z}_{i,0} = \frac{2}{n-2} \mu_i \frac{\partial}{\partial \mu_i} W_{\mu_i, \xi_i} \quad \textrm{ and } \quad  \mathcal{Z}_{i,j} = - n  \mu_i \frac{\partial}{\partial (\xi_i)_j} W_{\mu_i, \xi_i} .$$
Struwe's result asserts in addition that $\Vert \vp_k \Vert_{H^1} \to 0$ as $k \to + \infty$. We thus have
\ben \label{defvpk}
 u_k = \sum_{i=1}^N W_{ \mu_{i,k},\xi_{i,k}} + \vp_k,  \quad  \vp_k \in \mathcal{K}_{N,k}^{\perp} , \quad \Vert \vp_k \Vert_{H^1(M)} \to 0
 \een
 for $k$ large enough.

 \medskip

Now, the elements $ \mathcal{Z}_{i,0}$ and $ \mathcal{Z}_{i,j}$ are respectively equal to $Z_{i,0}$ and $Z_{i,j}$ in \eqref{defZi} at first order. A simple adaptation of the results in Robert-V\'etois \cite{RobertVetois} therefore shows that, for $\delta >0$ small enough fixed there exists, for $k$ large enough, a unique $\psi_k \in \mathcal{K}_{N,k}^{\perp}$ satisfying $\Vert \psi_k \Vert_{H^1} < \delta$ and
\be
  \Pi_{\mathcal{K}_{N,k}^{\perp}}\Bigg(  \sum_{i=1}^N  W_{ \mu_{i,k},\xi_{i,k}}  + \psi_k - (\triangle_g + c_n S_g + \ve_k h)^{-1}\Big(  \Big( \sum_{i=1}^N W_{ \mu_{i,k},\xi_{i,k}}  + \psi_k \Big)_+^{2^*-1} \Big) \Bigg) =0. \ee
By   \eqref{defvpk} we thus have $\psi_k = \vp_k$. Assume now that $(\mu_{i,k},\xi_{i,k})_{1 \le i \le N}$ satisfies \eqref{growthmuell} and \eqref{defxi}. It is easily seen that the results in Section \ref{lin} still apply if we replace $Z_{i,j}$ by $\mathcal{Z}_{i,j}$ and $K_k$ by $\mathcal{K}_{N,k}$, since only the leading-order behavior of $\mathcal{Z}_{i,j}$ comes into play in the proof of Proposition \ref{proplin}. We can therefore apply Proposition \ref{propnonlin}  with $\mathcal{K}_{N,k}$ (and, say, $t_i = 1$ for all $1 \le i \le N$) which provides us, for $k$ large enough, with a uniquely defined $\tilde{\vp}_k \in \mathcal{K}_{N,k}^{\perp} $ which solves the previous equation and satisfies in addition $\Vert \tilde{\vp}_k \Vert_{H^1} = o(1)$ as $k \to + \infty$ by \eqref{vpveH1}. We therefore have $\vp_k = \psi_k = \tilde{\vp}_k$ in the end and by \eqref{set1} $\vp_k$ satisfies
\ben \label{estapriori}
 |{\vp}_k | \lesssim \Psi_{\ve_k} .
 \een
Proposition \ref{propnonlin} thus turns the weak $H^1$ control given by \eqref{defvpk} into a highly accurate pointwise control. This example is intended as a toy model to match the analysis developed in this paper. But the method described here is likely to generalize to more general bubbling configurations (with or without a weak limit), to more general potential than $c_n S_g + \ve_k h$ and to sign-changing solutions, provided an analogue of Proposition \ref{proplin} is available. 

This approach can be seen as an alternative to the approach in Druet-Hebey-Robert \cite{DruetHebeyRobert} to obtain pointwise blow-up estimates for solutions of critical elliptic equations. Here we reduce the problem to proving sharp pointwise estimates for solutions \emph{of a linear problem}. And the precision of these linear estimates can be expected to only depend on the precision of the \emph{ansatz} chosen for the bubbles (see the discussion surrounding \eqref{interpretation}). This approach yields an improved precision for symmetry estimates in particular configurations (such as the one described in this paper) and could prove particularly useful in the compactness analysis of highly nonlinear problems, such as critical $p$-Laplace equations or critical anisotropic equations.

\section{Expansion of the kernel coefficients and proof of Theorem \ref{maintheo}} \label{expansionkernel}

In this section we prove Theorem \ref{maintheo}. Let $(M,g)$ be a locally conformally flat manifold. Let $k \ge 1$ be an integer, $\ve_1$ be given by Proposition \ref{propnonlin}, $h \in C^1(M)$ and $\big( (t_{1}, \xi_{1}, (t_{i}, z_{i})_{2 \le i \le k} \big) \in S_k$, where $S_k$ is as in \eqref{defA}. Throughout this section we will assume that 
\[ f(u) = (u_+)^{2^*-1} \quad \textrm{ and } \quad \kappa_i = 1 \quad \forall  i\in \{1, \dots,  k\}.\] 
Let $\vp_{\ve}$ be the unique function given by Proposition \ref{propnonlin}. Equation \eqref{eqnonlin} together with \eqref{estoptvp} shows the existence of real numbers $(\nu_{i,j})$, $1 \le i \le k$ and $0 \le j \le n$ such that 
\ben
 \label{eqvpC0}
\bal
 \big(\triangle_g + c_n S_g + \ve h\big)\big(  \sum_{i=1}^{k} W_i + \vp_{\ve} \big)&  - \big( \sum_{i=1}^{k} W_i + \vp_{\ve} \big) ^{2^*-1}  \\
 & = \sum \limits_{\underset{j =0 .. n }{i=1..k}}   \nu_{i,j} ( \triangle_g   + c_n S_g + \ve h)Z_{i,j}.
 \eal 
 \een
The $\nu_{i,j}$ are, by Proposition \ref{propnonlin}, continuous functions of  $\big( (t_{1}, \xi_{1}, (t_{i}, z_{i})_{2 \le i \le k} \big) \in S_k$. They \emph{a priori} all depend on each variable in $S_k$. This is a striking difference with the symmetric situation of Morabito-Pistoia-Vaira \cite{mpv}, where an inductive splitting of the remainder term led to a triangular dependency of the coefficients $\nu_{i,j}$. In the present work, we overcome the difficulties due to the non-symmetric background  by using the sharp estimates on $\vp_\ve$ given by \eqref{estoptvp}.

\medskip

We perform an asymptotic expansion of the $\nu_{i,j}$ in \eqref{eqvpC0}. We state the $\ell=1$ and $\ell \ge 2$ case separately for clarity but the following two statements are proven together.

\begin{prop}[Case $\ell = 1$] \label{propnu1}
We have:
$$\Vert \nabla V_0 \Vert^2_{L^2(\RR^n)} \nu_{1,0} (t_1, \dots, z_k) = \ve^{\frac{n-2}{n-4}} \Big( \frac{4D_1}{n-2} h(\xi_1)  t_{1}^2 - 2 D_2 t_1^{n-2} m(\xi_1) + \Lambda_{0,1}\Big) $$
and,  for $1 \le j \le n$,
$$  \Vert \nabla V_j \Vert^2_{L^2(\RR^n)}  \nu_{1,j}  (t_1, \dots, z_k) = \ve^{\frac{n-1}{n-4}} \Big( n D_1 t_1^3 \nabla_j h(\xi_1) - n D_2 t_1^{n-1} \nabla_j m(\xi_1) + \Lambda_{j,1} \Big) $$
as $\ve \to 0$. 
In these expressions $m(\xi_1)$ denotes the mass of the Green function at $\xi_1$ as defined in \eqref{defmass} and $D_1, D_2$ are positive constants that only depend on $n$. Also, $V_j$ is as in \eqref{defVi}, and $\Lambda_{j,1}$, for $0 \le j \le n$, denote continuous functions defined in $S_k$ which converge uniformly to $0$ as $\ve \to 0$.
\end{prop}

\begin{prop}[Case $2 \le \ell \le k$] \label{propnuge2}
For any $2 \le \ell \le k$ we have
$$\bal
\Vert \nabla V_0 \Vert^2_{L^2(\RR^n)} &\nu_{\ell,0} (t_1, \dots, z_k) = \ve^{1 + 2 \gamma_{\ell}} \Big( \frac{4D_1}{n-2} h(\xi_1) t_{\ell}^2 - D_3 \left( \frac{t_{\ell}}{t_{\ell-1}} \right)^{\pui} U_0(z_\ell) + \Lambda_{0,\ell} \Big)
\eal 
$$
and, for $1 \le j \le n$,
$$ \Vert \nabla V_j \Vert^2_{L^2(\RR^n)}  \nu_{\ell,j}  (t_1, \dots, z_k) = \ve^{ \frac{n}{2}(\gamma_{\ell} - \gamma_{\ell-1})} \Bigg(  -  n D_3 \left( \frac{t_{\ell}}{t_{\ell-1}} \right)^{\frac{n}{2}} \partial_j U_0(z_\ell) 
+ \Lambda_{j,\ell}\Bigg)  $$
as $\ve \to 0$, where $U_0$ is as in \eqref{defB0} and $\gamma_\ell$ is as in \eqref{defmu}. Here $D_3$ is a positive constant, $D_1$ is the same as in Proposition \ref{propnu1} and $\Lambda_{j,\ell}$, for $0 \le j \le n$ denote functions defined in $S_k$ which converge uniformly to $0$ as $\ve \to 0$. 
\end{prop}

In view of Propositions \ref{propnu1} and \ref{propnuge2} and of \eqref{growthmuell} we have to carry on the computations at the order $\ve \mu_1^2$ and $\ve \mu_1^3$ respectively for $\nu_{1,0}$ and $\nu_{1,j}$, $1 \le j \le n$, and at the order $\ve \mu_{\ell}^2$ and $\frac{\mu_{\ell}}{\mu_{{\ell}-1}} \ve \mu_{\ell}^2 = \left( \frac{\mu_{\ell}}{\mu_{{\ell}-1}} \right)^{\frac{n}{2}} $ respectively for  $\nu_{\ell,0}$ and $\nu_{\ell,j}$, for $\ell \ge 2 $ and $1 \le j \le n$.  In the following we will let: 
\ben \label{defTHETA10}
\Theta_{1,0}  = \ve \mu_1^2 \sim \mu_1^{n-2} 
\een
since $(M,g)$ is locally conformally flat and, for any $1 \le j \le n$,
\ben \label{defTHETA1j}
\Theta_{1,j}  = \ve \mu_1^3 \sim \mu_1^{n-1}. 
\een
For $2 \le \ell \le k$ we define similarly:
\ben \label{defTHETAell}
\Theta_{\ell,0}  = \ve \mu_\ell^2   
\quad \textrm{ and, for any } 1 \le j \le n, \quad 
\Theta_{\ell,j}  =  \frac{\mu_{\ell}}{\mu_{\ell-1}}\ve \mu_{\ell}^2 \sim  \left( \frac{\mu_{\ell}}{\mu_{{\ell}-1}} \right)^{\frac{n}{2}} . 
\een
For $a, b > 0$, the symbol $a\sim b$ in the previous expressions means that there exists a constant $C >0$, independent of $\ve$ and of the choice of $\big( t_{1}, \xi_{1}, (t_{i}, z_{i})_{2 \le i \le k} \big) \in S_k$, such that 
\[ \frac{1}{C}a \le b \le C a. \]  
Throughout the proof of Propositions \ref{propnu1} and \ref{propnuge2} we will also make use of the following fact: 
\ben \label{coolfact} \ve \mu_2^2 \sim \left( \frac{\mu_2}{\mu_1} \right)^\pui =  o (\ve \mu_1^3) \een
as $\ve \to 0$ which, since $(M,g)$ is locally conformally flat, follows easily from \eqref{growthmuell}.

\begin{proof}[Proof of Propositions  \ref{propnu1}  and \ref{propnuge2}]
First, $\vp_{\ve}$ satisfies \eqref{estoptvp}, where $\Psi_\ve$ is as in \eqref{CapPsi}. As shown in the proof of Proposition \ref{propnonlin} (see e.g. \eqref{eqT} and \eqref{propnonlinX}), Proposition \ref{proplin} applies to \eqref{eqvpC0} and yields, for any $1 \le i \le k$:
\ben \label{DL1}
  \sum_{j=0}^n  | \nu_{i, j}| \lesssim \ve \mu_i^2.
\een
As a consequence of \eqref{DL1} and of \eqref{psZik} we have, for any $\ell \in \{1, \dots, k\}$ and $0 \le q \le n$,
\ben \label{DL2}
\int_M \Big( \sum \limits_{\underset{j =0 .. n }{i=1..k}}   \nu_{i,j} ( \triangle_g   + c_n S_g + \ve h)Z_{i,j} \Big) Z_{\ell,q} = \Vert \nabla V_q \Vert_{L^2(\RR^n)}^2 \nu_{\ell,q} + o (\Theta_{\ell,q}).
\een
Let $1 \le \ell \le k$ be fixed. We write \eqref{eqvpC0} in the following way:
\ben \label{DL3}
\bal
 &\sum \limits_{\underset{j =0 .. n }{i=1..k}}  \nu_{i,j} \big( \triangle_g   + c_n S_g + \ve h\big)Z_{i,j} \\
  & =    \Big(\triangle_g + c_n S_g + \ve h\Big) \Big( \sum_{i=1}^k W_{i} \Big) - \Big(  \sum_{i=1}^k W_{i} \Big) ^{2^*-1}  \\
  & +  \big(\triangle_g + c_n S_g + \ve h\big) \vp_{\ve} - (2^*-1)W_{\ell}^{2^*-2} \vp_{\ve} \\
& - (2^*-1) \Bigg[ \Big(  \sum_{i=1}^k W_{i} \Big) ^{2^*-2} - W_\ell^{2^*-2}  \Bigg] \vp_\ve \\
  & - \Bigg[ \Big(  \sum_{i=1}^k W_{i} + \vp_\ve \Big) ^{2^*-1}  - \Big(  \sum_{i=1}^k W_{i} \Big) ^{2^*-1}  - (2^*-1) \Big(  \sum_{i=1}^k W_{i} \Big) ^{2^*-2} \vp_\ve \Bigg] \\
& = I_1 + I_2 + I_3 + I_4,\eal 
  \een
where each term $I_i$, $1 \le i \le 4$, is given by the $i$-th line in the right-hand side of \eqref{DL3}. We first claim that the sole contribution to $\nu_{\ell,j}$ is given by $I_1$:

\begin{lemme} \label{lemmaDL}
For any $1 \le \ell \le k$ and any $0 \le j \le n$ there holds:
$$ \int_M(I_2 + I_3 + I_4)Z_{\ell,j} dv_g = o ( \Theta_{\ell,j}) ,$$ 
where  $\Theta_{\ell,j}$ is as in \eqref{defTHETA10}, \eqref{defTHETA1j} and \eqref{defTHETAell}.
\end{lemme}

\begin{proof}[Proof of Lemma \ref{lemmaDL}]
We estimate each term separately. 

\medskip

\textbf{Estimation of $I_2$.} Let $0 \le j \le n$. First, integrating by parts shows that
$$ \int_M I_2 Z_{\ell,j} dv_g = \int_M \vp_{\ell,\ve} \Big(  \big(\triangle_g + c_n S_g + \ve h\big) Z_{\ell,j} - (2^*-1)W_{\ell}^{2^*-2} Z_{\ell,j} \Big) dv_g. $$
Assume first that $j = 0$. By \eqref{lapZ0} and \eqref{estoptvp}, and since $|Z_{\ell,0}|\lesssim W_{\ell}$, we get
\[ \bal
\Bigg|\int_{B_{\ell+1}}\vp_{\ell,\ve} \Big(  \big(\triangle_g + c_n S_g &+ \ve h\big) Z_{\ell,j} - (2^*-1)W_{\ell}^{2^*-2} Z_{\ell,j} \Big) dv_g \Bigg| \\
& \lesssim \mu_{\ell}^{\frac{n-6}{2}} \int_{B_{\ell+1}} |\vp_\ve| dv_g   \\
& \lesssim  \mu_{\ell}^{\frac{n-6}{2}} \sum_{i=\ell+1}^k  \int_{B_i \backslash B_{i+1}} |\vp_\ve| dv_g  = o(\ve \mu_{\ell}^2). 
\eal  \]
Similarly, using again \eqref{lapZ0} and  \eqref{estoptvp},
\[ \bal 
\Bigg| \int_{M \backslash B_\ell} \vp_{\ell,\ve} \Big(  \big(\triangle_g + c_n S_g &+ \ve h\big) Z_{\ell,j} - (2^*-1)W_{\ell}^{2^*-2} Z_{\ell,j} \Big) dv_g  \Bigg| \\ 
& \lesssim  \sum_{i=1}^{\ell-1}  \int_{B_i \backslash B_{i+1}} \Big( \mu_{\ell}^{\pui} + \ve W_{\ell}  \Big) |\vp_\ve| dv_g \\
& = o(\ve \mu_{\ell}^2),
\eal \]
where this last integral is zero when $\ell=1$, and 
\[ \bal 
\Bigg| \int_{B_\ell \backslash B_{\ell+1}}  \vp_{\ell,\ve} \Big(  \big(\triangle_g + c_n S_g &+ \ve h\big) Z_{\ell,j} - (2^*-1)W_{\ell}^{2^*-2} Z_{\ell,j} \Big)  dv_g  \Bigg|  \\
& \lesssim   \int_{B_\ell \backslash B_{\ell+1}} \Big( \mu_{\ell}^{\pui} + \ve W_{\ell}  \Big) |\vp_\ve| dv_g \\
& = o(\ve \mu_{\ell}^2).
\eal \]
In conclusion we have 
\ben \label{DL42}
\int_M I_2 Z_{\ell,0} dv_g  =  o( \Theta_{\ell,0}),
\een
for any $1 \le \ell \le k$, where $\Theta_{1,0}$ and $\Theta_{\ell,0}$ are given respectively by \eqref{defTHETA10} and \eqref{defTHETAell}. 

\smallskip

Assume now that $1 \le j \le n$. Integrating by parts and using \eqref{lapZi} 
gives
$$ 
\bal
\left|\int_M I_2 Z_{\ell,j} dv_g \right| &  \lesssim  \ve \int_M |\vp_{\ell,\ve}| |Z_{\ell,j}|dv_g + \int_M \mu_{\ell}^{\frac{n}{2}} \theta_{\ell}^{2-n} |\vp_{\ell,\ve} |dv_g .
\eal $$
We first have, with \eqref{estoptvp}:
\[\bal 
 \int_{B_{\ell+1}} \big( \ve  |Z_{\ell,j}| + \mu_{\ell}^{\frac{n}{2}} \theta_{\ell}^{2-n}\big) |\vp_{\ell,\ve} |dv_g  & \lesssim \mu_{\ell}^{2 - \frac{n}{2}} \int_{B_{\ell+1}} |\vp_\ve| dv_g \\
 & =    o(\Theta_{\ell,j})
 \eal \]
 for all $1 \le \ell \le k$, where $\Theta_{\ell,j}$ is given by \eqref{defTHETA1j} and \eqref{defTHETAell}. Straightforward computations using \eqref{estoptvp} to estimate the integral over $B_{\ell} \backslash B_{\ell+1}$ and $M \backslash B_\ell$ similarly give 
\ben \label{DL51}
  \int_M I_2 Z_{\ell,j} dv_g  = o(\Theta_{\ell,j}) 
  \een
for any $1 \le \ell \le k$ and $1 \le j \le n$, where $\Theta_{1,j}$ and $\Theta_{\ell,j}$ are given respectively by \eqref{defTHETA1j} and \eqref{defTHETAell}. 

\medskip

\textbf{Estimation of $I_3$.} By  \eqref{estoptvp} and \eqref{DL3} we have
\[ |I_3(x)| \lesssim \left \{
\bal
& W_{i}^{2^*-2} |\phi_{\ve}| & \textrm{ if } x \in B_{i} \backslash B_{i+1}, \quad i \neq \ell \\
& W_{\ell}^{2^*-3} \big( W_{\ell-1} + W_{\ell+1} \big) |\vp_\ve|  & \textrm{ if } x \in B_\ell \backslash B_{\ell+1}
\eal \right. ,\]
where it is intended that the term $W_{\ell-1}$ vanishes when $\ell=1$ and the term $W_{\ell+1}$ vanishes when $\ell=k$. So that 
\[ 
\bal
\left| \int_M I_3 Z_{\ell,j} dv_g \right| &\lesssim \sum_{i \neq \ell} \int_{B_{i} \backslash B_{i+1}} W_{i}^{2^*-2} |\vp_{\ve}| |Z_{\ell,j}| dv_g \\
& + \int_{B_\ell \backslash B_{\ell+1}} W_{\ell}^{2^*-3} \big( W_{\ell-1} + W_{\ell+1} \big) |\vp_\ve| |Z_{\ell,j}| dv_g .
\eal\]
Assume first that $\ell =1$ and recall that $B_1 = M$. Using \eqref{estoptvp} we obtain
\[ \bal \int_{M \backslash B_{2}} W_{1}^{2^*-3}  W_{2} |\vp_\ve| |Z_{1,j}| dv_g & \lesssim 
 o(\mu_1^{n-1}) + \left( \frac{\mu_2}{\mu_1}\right)^{\frac{n+2}{2}} = o(\Theta_{\ell,j})
\eal\]
for all $0 \le j \le n$, where $\Theta_{1,j}$ is as in \eqref{defTHETA10} and \eqref{defTHETA1j}. Similarly we have, using \eqref{estoptvp} and since $|Z_{1,j}| \lesssim W_1$:
\[ \bal
 \sum_{i \ge 2} \int_{B_{i} \backslash B_{i+1}} W_{i}^{2^*-2} |\vp_{\ve}| |Z_{1,j}| dv_g  & \lesssim  \sum_{i \ge 2} \int_{B_{i} \backslash B_{i+1}}  W_{i}^{2^*-2} W_1 |\vp_{\ve}| dv_g \\
& \lesssim \mu_2 \mu_1^{n-2} + \mu_2^{\frac{n+2}{2}}\mu_1^{-\frac{n+2}{2}} +\mu_2^{\frac{n}{2}} \mu_1^{-\frac{n-2}{2}} \\
& = o (\Theta_{1,j}).
 \eal
 \]
In the end we obtain 
\ben  \label{DL41}
 \int_M I_3 Z_{1,j} dv_g  = o(\Theta_{1,j})
\een
for all $0 \le j \le n$. 

\smallskip

Assume now that $2 \le \ell \le k$. First, straightforward computations using \eqref{estoptvp} show that, for any $0 \le j \le n$, 
\[ \bal
 \int_{B_\ell \backslash B_{\ell+1}} W_{\ell}^{2^*-3} \big( W_{\ell-1} + W_{\ell+1} \big) |\vp_\ve| |Z_{\ell,j}| dv_g  & \lesssim  \int_{B_\ell \backslash B_{\ell+1}} W_{\ell}^{2^*-2} \big( W_{\ell-1} + W_{\ell+1} \big) |\vp_\ve| dv_g \\
 & \lesssim  \left( \frac{\mu_\ell}{\mu_{\ell-1}}\right)^{\frac{n+2}{2}} + \ve \mu_\ell^3 +  \left( \frac{\mu_{\ell+1}}{\mu_{\ell-1}}\right)^{\frac{n+2}{2}} \\
& = o(\Theta_{\ell,j}). 
\eal \]
Similarly, using again \eqref{estoptvp} we have
\[ \bal
 \sum_{i \ge \ell+1}   \int_{B_{i} \backslash B_{i+1}} W_{i}^{2^*-2} |\vp_{\ve}| |Z_{\ell,j}| dv_g & \lesssim \frac{\mu_{\ell+1}^{\frac{n}{2}}}{\mu_\ell \mu_{\ell-1}^\pui} + \mu_{\ell+1} \left( \frac{\mu_{\ell+1}}{\mu_\ell} \right)^{\pui} \\
 & = o(\Theta_{\ell,j}) 
\eal\]
for any $0 \le j \le n$. Finally, on the one hand, using that $|Z_{\ell,0}| \lesssim W_\ell$ we find
\[ \bal  \sum_{i \le \ell-1}   \int_{B_{i} \backslash B_{i+1}} W_{i}^{2^*-2} |\vp_{\ve}| |Z_{\ell,0}| dv_g &\lesssim \mu_\ell^{\pui} \mu_{\ell-2}^{1 - \frac{n}{2}} + \mu_\ell^{\frac{n+2}{2}} \mu_{\ell-1}^{- \frac{n+2}{2}} + o \left(\left( \frac{\mu_\ell}{\mu_{\ell-1}} \right)^{\pui}\right) \\
& = o(\Theta_{\ell,0}). \eal \]
On the other hand, using that $|Z_{\ell,j}| \lesssim \mu_\ell^{\frac{n}{2}} \theta_{\ell}^{1-n}$ for $1 \le j \le n$, where $\theta_\ell$ is as in \eqref{defti}, we have:
\[ \bal  \sum_{i \le \ell-1}  \int_{B_{i} \backslash B_{i+1}} W_{i}^{2^*-2} |\vp_{\ve}| |Z_{\ell,j}| dv_g & \lesssim \mu_\ell^{\frac{n}{2}} \mu_{\ell-1}^{-1} \mu_{\ell-2}^{1 - \frac{n}{2}} +  o \left(\left( \frac{\mu_\ell}{\mu_{\ell-1}} \right)^{\frac{n}{2}}\right) \\
& = o(\Theta_{\ell,j}). \eal \]
In the end we thus have 
\ben  \label{DL42bis}
 \int_M I_3 Z_{\ell,j} dv_g  = o(\Theta_{\ell,j})
\een
for all $0 \le j \le n$.

\medskip

\textbf{Estimation of $I_4$.}  We finally estimate the contribution of $I_4$. By definition and by \eqref{estoptvp} we have 
\[ |I_4| \lesssim \big( \sum_{i=1}^k W_i \big)^{2^*-3} |\vp_\ve|^2 \quad \textrm{ in } M\]
and therefore 
\[ \left| \int_M I_4 Z_{\ell,j} dv_g \right| \lesssim  \sum_{i=1}^k \int_{B_i \backslash B_{i+1}} W_i^{2^*-3}|Z_{\ell,j}| |\vp_\ve|^2 dv_g .\]
Using again \eqref{estoptvp} and the following estimate on $Z_{\ell,j}$:
\ben \label{estbetterZ}
 |Z_{\ell,j} | \lesssim \left \{ \bal & W_{\ell} & \textrm{ if } j = 0 \\ & \mu_{\ell}^{\frac{n}{2}} \theta_{\ell}^{1-n} & \textrm{ if } 1 \le  j \le n \\ \eal \right. ,\een
we get by straightforwards computations that mimic those that led to  \eqref{DL42}, \eqref{DL51}, \eqref{DL41} and \eqref{DL42bis} that
\ben \label{DL62}
\int_M I_4 Z_{\ell,j} dv_g =  o (\Theta_{\ell,j}), 
\een
for any $1 \le \ell \le k$ and $0 \le j \le n$, where $\Theta_{\ell,j}$ is as in  \eqref{defTHETA10},  \eqref{defTHETA1j} and \eqref{defTHETAell}. Combining \eqref{DL42}, \eqref{DL51}, \eqref{DL41}, \eqref{DL42bis} and \eqref{DL62} concludes the proof of Lemma \ref{lemmaDL}.
\end{proof}

Let us now estimate the contribution due to $I_1$ in \eqref{DL3}. We first claim that the following holds:
for any $1 \le \ell \le k$ and $0 \le j \le n$,
\ben \label{stepX}
\bal
 \int_M I_1 Z_{\ell,j} dv _g & = \int_M \Bigg[   \Big(\triangle_g + c_n S_g + \ve h\Big)W_\ell - W_\ell^{2^*-1} \\
 & - \Big[ \big( \sum_{i=1}^k W_i\big)^{2^*-1} - \sum_{i=1}^k W_i^{2^*-1} \Big] \Bigg] Z_{\ell,j} + o(\Theta_{\ell,j}),
\eal  
\een
where $\Theta_{\ell,j}$ is as in   \eqref{defTHETA10},  \eqref{defTHETA1j} and \eqref{defTHETAell}. Indeed, by definition of $I_1$ in \eqref{DL3},
\[ I_1 =   \sum_{i=1}^k  \Bigg[ \Big(\triangle_g + c_n S_g + \ve h\Big)W_{i} - W_i^{2^*-1}  \Bigg]- \Bigg[ \big( \sum_{i=1}^k W_i\big)^{2^*-1} - \sum_{i=1}^k W_i^{2^*-1} \Bigg] . \]
By \eqref{lapbulle} and \eqref{rapportsbulles} and since $(M,g)$ is locally conformally flat we have for any $1 \le i \le k$,
\[ \bal
\Big|  \Big(\triangle_g + c_n S_g + \ve h\Big)W_{i} - W_i^{2^*-1} \Big| & \lesssim  
 \frac{\mu_i^{\frac{n+2}{2}}}{\theta_i^4}  
+ \ve W_i  .\\
\eal \]
Let $i > \ell$. Straightforward computations using \eqref{growthmuell} show that
\ben \label{stepX1}
\bal
\int_M 
\frac{\mu_i^{\frac{n+2}{2}}}{\theta_i^4} 
|Z_{\ell,j}| dv_g   &+ \int_M \ve W_i |Z_{\ell,j}| dv_g  \lesssim (\ve + \mu_{i}^2) \ve \mu_{i}^2 = o(\Theta_{\ell,j})  .
\eal
\een
Similarly, if $i < \ell$,
\ben \bal \label{stepX2} 
\int_M &
\frac{\mu_i^{\frac{n+2}{2}}}{\theta_i^4} 
|Z_{\ell,j}| dv_g   + \int_M \ve W_i |Z_{\ell,j}| dv_g \\
& \lesssim 
\left \{
\bal
& \mu_\ell^\pui \mu_{i}^{\pui} & j=0 \\ 
&\mu_\ell^{\frac{n}{2}} \mu_{i}^{\frac{n-4}{2}} & j\in \{1, \dots, n\} 
\eal \right.  \\
& = o(\Theta_{\ell,j}) \eal \een
again by \eqref{growthmuell}. Estimates \eqref{stepX1} and \eqref{stepX2} together prove \eqref{stepX}.

\medskip

We now prove the following claim:

\begin{lemme} \label{finalclaim}
For $\ell \ge 2$,
\[  -\int_{M} \Big[ \big( \sum_{i=1}^k W_i\big)^{2^*-1} - \sum_{i=1}^k W_i^{2^*-1} \Big] Z_{\ell,j} dv_g = -  n D_3\left( \frac{\mu_\ell}{\mu_{\ell-1}}\right)^{\frac{n}{2}} \partial_j U_0(z_{\ell}) + o(\Theta_{\ell,j})\]
for some positive constant $D_3$ (given by \eqref{defD4} below), while for $\ell=1$
\[  -\int_{M} \Big[ \big( \sum_{i=1}^k W_i\big)^{2^*-1} - \sum_{i=1}^k W_i^{2^*-1} \Big] Z_{1,j} dv_g=  o(\Theta_{1,j}).\]
\end{lemme}

\begin{proof}
We first assume that $\ell \ge2$. On $B_{\ell+1}$ we have 
\[   \left| \big( \sum_{i=1}^k W_i\big)^{2^*-1} - \sum_{i=1}^k W_i^{2^*-1} \right| \lesssim W_{\ell+1}^{2^*-2} (\sum_{p \ge \ell+2}W_{p} + W_\ell), \]
with the convention that $W_{i} = 0$ if $i \ge k+1$, so that 
\[ \bal \left| \int_{B_{\ell+1}} \Big[ \big( \sum_{i=1}^k W_i\big)^{2^*-1} - \sum_{i=1}^k W_i^{2^*-1} \Big] Z_{\ell,j}  dv_g \right| & \lesssim \left( \frac{\mu_{\ell+2}}{\mu_\ell} \right)^\pui +  \left( \frac{\mu_{\ell+1}}{\mu_\ell} \right)^{\frac{n}{2}} \\
& = o (\Theta_{\ell,j}),\eal \]
since by \eqref{growthmuell} we have (for $\ell \ge 2$):
\[ \frac{\mu_{\ell+1}}{\mu_\ell}  = o \left( \frac{\mu_{\ell}}{\mu_{\ell-1}}  \right). \]
Independently, if $i \le \ell-1$, we have 
\[   \left| \big( \sum_{i=1}^k W_i\big)^{2^*-1} - \sum_{i=1}^k W_i^{2^*-1} \right| \lesssim W_i^{2^*-2} (W_{i+1} + W_{i-1}) \quad \textrm{ in } B_i \backslash B_{i+1}, \]
with the convention that $W_0 = 0$. Using that $|Z_{\ell,0}| \lesssim W_{\ell}$ and that $|Z_{\ell,j}| \lesssim \mu_\ell^{\frac{n}{2}} \theta_{\ell}^{1-n}$ then gives that 
\[ \bal \left| \int_{M\backslash B_\ell} \Big[ \big( \sum_{i=1}^k W_i\big)^{2^*-1} - \sum_{i=1}^k W_i^{2^*-1} \Big] Z_{\ell,0} dv_g \right| & \lesssim \sum_{i=1}^{\ell-1} \Bigg( \left( \frac{\mu_\ell}{\mu_{i}}\right)^{\frac{n}{2}}  +  \left( \frac{\mu_\ell}{\mu_{i-1}}\right)^{\frac{n-2}{2}} \Bigg)\\
& = o (\Theta_{\ell,0}),
\eal \]
and that 
\[ \bal \left| \int_{M\backslash B_\ell} \Big[ \big( \sum_{i=1}^k W_i\big)^{2^*-1} - \sum_{i=1}^k W_i^{2^*-1} \Big] Z_{\ell,j} dv_g \right| & \lesssim \sum_{i=1}^{\ell-1} \Bigg( \left( \frac{\mu_\ell}{\mu_{i}}\right)^{\frac{n+1}{2}}  + \frac{\mu_\ell^{\frac{n}{2}}}{\mu_{i-1}^\pui \mu_i} \Bigg)\\
& = o( \Theta_{\ell,j}),
\eal \]
with the convention that the term $ \left( \frac{\mu_\ell}{\mu_{i-1}}\right)^{\frac{n-2}{2}}$ vanishes if $i=1$. It therefore remains to estimate the integral over $B_\ell \backslash B_{\ell+1}$. Using \eqref{rapportsbulles} we can write that
\[ \bal \Bigg| \big( \sum_{i=1}^k W_i\big)^{2^*-1} &- \sum_{i=1}^k W_i^{2^*-1} - (2^*-1) W_\ell^{2^*-2} \sum_{i=1}^{\ell-1} W_i\Bigg|\\
&  \lesssim W_{\ell}^{2^*-2} \sum_{p \ge \ell+1} W_p + W_\ell^{2^*-3} (W_{\ell+1}^2 + W_{\ell-1}^2), \eal \]
so that 
\[ \bal
& \left| \int_{B_\ell \backslash B_{\ell+1}} \Bigg[ \big( \sum_{i=1}^k W_i\big)^{2^*-1} - \sum_{i=1}^k W_i^{2^*-1} - (2^*-1) W_\ell^{2^*-2} \sum_{i=1}^{\ell-1} W_i \Bigg] Z_{\ell,j} dv_g \right| \\
& \lesssim  \left( \frac{\mu_{\ell+1}}{\mu_\ell}\right)^{\frac{n-2}{2}} +  \left\{ \bal &  \left( \frac{\mu_{\ell}}{\mu_{\ell-1}}\right)^{\frac{n}{2}} & j = 0 \\ &  \left( \frac{\mu_{\ell}}{\mu_{\ell-1}}\right)^{\frac{n+1}{2}} & j \in \{1, \dots, n \}  \eal  \right. \\
& = o(\Theta_{\ell,j}). \eal \]
For the last line we wrote that 
\[ \left( \frac{\mu_{\ell+1}}{\mu_\ell}\right)^{\frac{n-2}{2}} = \left( \frac{\mu_{\ell}}{\mu_{\ell-1}} \right)^{\frac{n-2}{2}} \left(\frac{\mu_{\ell+1} \mu_{\ell-1}}{\mu_\ell^2} \right)^\pui = o(\Theta_{\ell,j}) , \]
and used \eqref{growthmuell}, which shows that
\[ \left(\frac{\mu_{\ell+1} \mu_{\ell-1}}{\mu_\ell^2} \right)^\pui = o \left( \frac{\mu_\ell}{\mu_{\ell-1}} \right).\] 
We have thus proven that 
\[\bal    -\int_{M} \Big[ \big( &\sum_{i=1}^k W_i\big)^{2^*-1} -\sum_{i=1}^k W_i^{2^*-1} \Big] Z_{\ell,j} dv_g \\
& = -(2^*-1) \int_{B_\ell \backslash B_{\ell+1}} W_\ell^{2^*-2}\Big( \sum_{i \le \ell-1}W_{i}\Big) Z_{\ell,j} dv_g + o(\Theta_{\ell,j}).\eal  \]
Let, for any $y \in B_0(\frac{\mu_\ell}{\mu_{\ell-1}})$ and any $i \le \ell-1$:
\[ \hat{W}_i(y) = \mu_\ell^\pui W_i \big( \exp_{\xi_\ell}^{g_{\xi_\ell}}(\mu_\ell y) \big). \]
As is easily checked, $V_j$ defined in \eqref{defVi} satisfies 
\[ V_j(x) = - n\partial_j U_0(x) \quad \textrm{ for all } x \in \RR^n. \]
Let $i \le \ell-1$. Using \eqref{defWi}, changing variables and integrating by parts we have
\[  \bal (2^*-1) \int_{B_\ell \backslash B_{\ell+1}} W_\ell^{2^*-2}W_{i} Z_{\ell,j} & =(2^*-1) \int_{B_0\big( \frac{\mu_{\ell-1}}{\mu_\ell} \big)} U_0^{2^*-2} V_j \hat{W}_i dy + o(\Theta_{\ell,j})\\
& =  n \int_{B_0\big( \frac{\mu_{\ell-1}}{\mu_\ell} \big)} U_0^{2^*-1} \partial_j \hat{W}_i dy + o(\Theta_{\ell,j}).  \eal \]
 If $i \le \ell-2$ we have, for any $y \in B_0(\frac{\mu_\ell}{\mu_{\ell-1}})$,
 \[ \big| \partial_j \hat{W}_i(y) \big| \lesssim \left(\frac{\mu_\ell}{\mu_{i}} \right)^{\frac{n}{2}}, \]
 so that 
 \[  n \int_{B_0\big( \frac{\mu_{\ell-1}}{\mu_\ell} \big)} U_0^{2^*-1} \partial_j \hat{W}_i dy = o( \Theta_{\ell,j}). \]
 If $i = \ell-1$ we write that 
 \[  \big| \partial_j \hat{W}_{\ell-1}(y) - \partial_j \hat{W}_{\ell-1}(0)\big| \lesssim \left(\frac{\mu_\ell}{\mu_{\ell-1}} \right)^{\frac{n+2}{2}}|y|, \]
 so that 
 \[ n \int_{B_0\big( \frac{\mu_{\ell-1}}{\mu_\ell} \big)} U_0^{2^*-1} \partial_j \hat{W}_{\ell-1} dy = n D_3 \partial_j \hat{W}_{\ell-1}(0) + o( \Theta_{\ell,j}), \]
 where we have let 
 \ben \label{defD4} 
 D_3= \int_{\RR^n} U_0^{2^*-1} dy.
 \een It remains to notice that, by definition of $\xi_\ell$ in \eqref{defxi}, we have 
 \[ \partial_j \hat{W}_{\ell-1}(0) = \left( \frac{\mu_\ell}{\mu_{\ell-1}}\right)^{\frac{n}{2}}\big( \partial_j U_0(z_{\ell}) + o(1) \big), \]
 where $U_0$ is as in \eqref{defB0}. This concludes the proof of Lemma \ref{finalclaim} when $\ell \ge 2$. 
 
 \medskip
 
 It remains to prove the $\ell=1$ case. We have, with \eqref{rapportsbulles} and the convention that $W_{k+1} = 0$ and $W_{0} = 0$, that
 \[\bal  \Bigg|  -\int_{M} &\Big[ \big( \sum_{i=1}^k W_i\big)^{2^*-1} - \sum_{i=1}^k W_i^{2^*-1} \Big] Z_{1,j} dv_g \Bigg| \\
 &  \lesssim \sum_{i \ge 1} \int_{B_i \backslash B_{i+1}} W_i^{2^*-2} \big( W_{i-1} + W_{i+1} \big) |Z_{1,j}| dv_g \\
 & \lesssim  \sum_{i \ge 2} \Bigg( \frac{\mu_i^{\frac{n}{2}}}{\mu_{i-1}\mu_{1}^\pui} + \left( \frac{\mu_{i+1}}{\mu_1}\right)^{\pui}\Bigg) + \left( \frac{\mu_2}{\mu_1}\right)^{\pui} \\
 & = o (\Theta_{1,j}) \eal\]
 for any $0 \le j \le n$, by \eqref{coolfact}. This concludes the proof of Lemma \ref{finalclaim}.
\end{proof}

We are now in position to conclude the proof of Propositions \ref{propnu1} and \ref{propnuge2}. By Lemma \ref{finalclaim} and \eqref{stepX} we only need to estimate the first term in the right-hand side of \eqref{stepX}. We first treat the $j=0$ case. By definition of $W_{\ell}$ in \eqref{defWi} and $Z_{\ell,0}$ in \eqref{defZi} it is easily seen that there holds
$$ \bal
 \frac{2}{n-2} \mu_{\ell} \frac{\partial} {\partial \mu_{\ell}}W_{\ell}   & = (n-2) \omega_{n-1} G_{g_{\xi}}(\xi, x) d_{g_{\xi}}(\xi, x)^{n-2} Z_{\ell,0} + O(\mu_\ell^{\pui})
 \eal $$
By \eqref{defmass} and \eqref{nonlcf} this gives, since $(M,g)$ is locally conformally flat:
$$
Z_{\ell,0} = 
\Big( 1 + O\big(d_{g_{\xi_\ell}}(\xi_\ell, x)^{n-2}\big)\Big) 
\frac{2}{n-2} \mu_{\ell} \frac{\partial} {\partial \mu_{\ell}}W_{\ell}  + O(\mu_\ell^{\pui}) .
$$
The latter expression then gives 
\ben \label{DL11a}
\bal
& \int_M  \Bigg[  \Big(\triangle_g + c_n S_g + \ve h\Big)W_\ell - W_\ell^{2^*-1} \Bigg] Z_{\ell,0} dv_g \\
& = \frac{2}{n-2} \mu_\ell\frac{\partial}{\partial \mu_{\ell}} \Bigg[ \frac12 \int_M |\nabla W_{\ell}|^2 + (c_n S_g + \ve h) W_{\ell}^2 dv_g  -  \frac{1}{2^*} \int_M W_{\ell}^{2^*} dv_g \Bigg] \\
& + o (\Theta_{\ell,0}).
 \eal
 \een 
 For any $\ell\ge1$, The computations in Esposito-Pistoia-V\'etois \cite{EspositoPistoiaVetois} (Lemma $4.1$) show that
\ben \label{DL11b}
 \bal
&\frac12 \int_M |\nabla W_{\ell}|^2 + (c_n S_g + \ve h) W_{\ell}^2 dv_g -  \frac{1}{2^*} \int_M W_{\ell}^{2^*} dv_g \\
&= D_1 \ve h(\xi_{\ell})\mu_\ell^2 - 
D_2 \mu_\ell^{n-2} m(\xi_\ell)  + o(\ve \mu_\ell^2) 
\eal\een
for some positive constants $D_1$ and  $D_2$, where $m(\xi_\ell)$ is the mass as defined in \eqref{defmass}, and where the $o(\cdot)$ terms are $C^1$ with respect to $\mu_\ell$ and to $\xi_{\ell}$. Plugging \eqref{DL11b} into \eqref{DL11a}  yields
\ben \label{DL11c}
\bal
\int_M & \Bigg[  \Big(\triangle_g + c_n S_g + \ve h\Big)W_\ell - W_\ell^{2^*-1} \Bigg] Z_{\ell,0} dv_g \\
& = \frac{4D_1}{n-2} \ve h(\xi_{\ell})\mu_\ell^2 
  -  2 D_2 \mu_\ell^{n-2} m(\xi_\ell)+ o(\Theta_{\ell,0}) .
\eal
\een
Assume now that $1 \le j \le n$. Lemmas 6.2 and 6.3 in \cite{EspositoPistoiaVetois} similarly show that
\be
\bal 
&\int_{B_{\ell}} \Bigg[ \big(\triangle_g + c_n S_g + \ve h\big) W_{\ell} - W_{\ell}^{2^*-1} \Bigg] Z_{\ell,j} dv_g  \\
& =  n  \mu_\ell \frac{\partial}{\partial (\xi_\ell)_j} \Bigg[  \frac12 \int_M |\nabla W_{\ell}|^2 + (c_n S_g + \ve h) W_{\ell}^2 dv_g  -  \frac{1}{2^*} \int_M W_{\ell}^{2^*} dv_g\Bigg] + o(\Theta_{\ell,j}).
\eal \ee
With \eqref{DL11b} this gives:
\ben \label{DL11e}
\bal 
&\int_{B_{\ell}} \Bigg[ \big(\triangle_g + c_n S_g + \ve h\big) W_{\ell} - W_{\ell}^{2^*-1} \Bigg] Z_{\ell,j} dv_g  \\
& = nD_1 \ve \nabla_j h(\xi_{\ell}) \mu_\ell^3
- nD_2\mu_\ell^{n-1}\nabla_j m(\xi_\ell) + o(\Theta_{\ell,j}) \\
& = o(\Theta_{\ell,j}).
\eal \een
Coming back to \eqref{DL3} and using Lemma \ref{lemmaDL}, \eqref{stepX}, Lemma \ref{finalclaim}, \eqref{DL11c} and \eqref{DL11e} concludes the proof of Propositions \ref{propnu1} and \ref{propnuge2}.
\end{proof}

We now conclude the proof of Theorem \ref{maintheo}:

\begin{proof}[End of the proof of Theorem \ref{maintheo}.]
Recall that $(M,g)$ is a locally conformally flat manifold of dimension $n \ge7$. Define a function $F: S_k \to \RR^{k(1+n)}$ by:
$$ \bal
 F(t_1, \xi_1, \dots, t_k, z_k) = 
\left \{
\bal
&  \frac{4D_1}{n-2} h(\xi_1)  t_{1}^2 - 2 D_2 t_1^{n-2} m(\xi_1) \\
& \Big( n D_1 t_1^3 \nabla_j h(\xi_1) - n D_2 t_1^{n-1} \nabla_j m(\xi_1) \Big)_{1 \le j \le n} \\
& \left(
\bal
& \frac{4D_1}{n-2} h(\xi_1) t_{\ell}^2 - D_3 \left( \frac{t_{\ell}}{t_{\ell-1}} \right)^{\pui} U_0(z_\ell) \\
&  -  n D_3 \left( \frac{t_{\ell}}{t_{\ell-1}} \right)^{\frac{n}{2}} \partial_j U_0(z_\ell) \\
\eal \right)_{2 \le \ell \le k, 1 \le j \le n}
\eal \right \},
\eal
$$
where $S_k$ is as in \eqref{defA}. Assume that $(M,g)$ is not conformally diffeomorphic to the standard sphere and let $\xi_0$ be a critical point of the mass function $\xi \mapsto m(\xi) >0$. Choose now $h \in C^1(M)$ such that $h(\xi_0) = 1$ and $\nabla h(\xi_0) = 0$ and let 
$$ t_{1,0} = \left( \frac{2D_1}{(n-2) D_2 m(\xi_0)} \right)^{\frac{1}{n-4}}. $$
Assume also that $h$ is chosen so that $\xi_0$ has non-zero degree as a zero of
$$ \xi \mapsto n D_1 t_{1,0}^3 \nabla h(\xi) - n D_2 t_{1,0}^{n-1} \nabla m(\xi).$$ 
These conditions are obviously met if $h \equiv 1$ and $\xi_0$ is a \emph{non-degenerate} critical point of the mass function $\xi \mapsto m(\xi)$. Let finally, for $ 2 \le \ell \le k$, 
$$ t_{\ell,0} = \left( \frac{4D_1}{(n-2) D_3} t_{\ell-1,0}^{\pui}\right)^{\frac{2}{n-6}} . $$
It is then easily seen that  $\big(t_{1,0}, \xi_0, (t_{\ell,0}, 0)_{2 \le \ell \le k} \big)$ is a zero of $F$ with non-zero degree. In particular, for any $\ve$ small enough, by Propositions \ref{propnu1} and \ref{propnuge2}, there exists $(t_{1,\ve}, \xi_{1,\ve}, (t_{\ell, \ve}, z_{\ell,\ve})_{2 \le \ell \le k}) \in S_k$ that annihilates the $\nu_{i,j}$ defined in \eqref{eqvpC0}. Denote by $\mu_{\ell,\ve}$ and $\xi_{\ell,\ve}$, $1 \le \ell \le k$, the parameters associated to $(t_{1,\ve}, \xi_{1,\ve}, (t_{\ell, \ve}, z_{\ell,\ve})_{2 \le \ell \le k}) \in S_k$ as in \eqref{defmu} and \eqref{defxi}. By \eqref{eqvpC0}, the function 
$$ \bal 
u_\ve & = \sum_{i =1}^k W_{i,\ve} + \vp_{\ve}  \\
& = \sum_{i=1}^k W_{\mu_{\ell}, \xi_{\ell}} + \vp_{\ve}\Big( (t_{1,\ve}, \xi_{1,\ve}, (t_{\ell, \ve}, z_{\ell,\ve})_{2 \le \ell \le k}\Big) 
\eal $$
solves \eqref{eqY}. By the pointwise estimates on $\vp_\ve$ given by Proposition \ref{propnonlin} it is easily seen that $u_\ve$ blows-up with $k$ towering bubbles around $\xi_0$. This concludes the proof of Theorem \ref{maintheo}.
\end{proof}

\section{Sign-changing bubble towers for the Br\'ezis-Nirenberg problem} \label{BrezisNirenberg}

We prove in this section Theorem \ref{theoBN}. Let us consider the Brezis-Nirenberg problem
\begin{equation}
\label{bn}
\triangle_{\xi} u -\varepsilon u =|u|^{\frac{4}{n-2}}u \ \hbox{in}\ \Omega,\ u=0\ \hbox{on}\ \partial\Omega,
\end{equation}
where $\Omega$ is a bounded open domain in $\mathbb R^n$, $n \ge 7$, $\triangle_\xi = - \sum_{i=1}^n \partial_i^2$ and $\varepsilon$ is a positive small parameter. For $\mu >0$ and $\xi \in \Omega$ we introduce the standard bubble
\begin{equation}
\label{bu_bn}U_{\mu,\xi}(x):=\mu^{-\pui} U_0\(\frac{x-\xi}{\mu}\),
 \end{equation}
where $U_0$ is given by \eqref{defB0}, and its projection $P_\Omega U_{\mu,\xi}$ on the space $H^1_0(\Omega)$, i.e.  the solution 
to the Dirichlet boundary problem
$$\triangle_\xi P_\Omega U_{\mu,\xi}=\triangle_\xi  U_{\mu,\xi}\ \hbox{in}\ \Omega,\ P_\Omega U_{\mu,\xi}=0\ \hbox{on}\ \partial\Omega.$$
It is well known that (see e.g. Proposition 1 in Rey \cite{r1}) that 
\begin{equation}
\label{pro_bn}
 P_\Omega U_{\mu,\xi}(x)=   U_{\mu,\xi}(x)-c_n\mu^{\pui}H(x,\xi)+ O\(\mu^{\frac{n+2}{2}}\)
\end{equation}
with $c_n:=\(n(n-2)\)^{\pui}\omega_{n-1}$, and that this expansion holds true uniformly with respect to $x\in\overline\Omega$ and $\xi$ in compact sets of $\Omega.$
Here $H(\cdot,\xi)$ is the regular part of the Green function $G$ of $\triangle_\xi$ in $\Omega$ with Dirichlet boundary condition, defined by
\ben \label{defRobin}
 G(x,y) = \frac{1}{(n-2)\omega_{n-1}} |x-y|^{2-n} - H(x,y) \quad \forall x \neq y \in \Omega. 
 \een
For an integer $k\ge1,$ we look for a sign-changing solution to \eqref{bn} having the following form: 
\begin{equation}\label{ans_bn}
u(x)=\sum\limits_{\ell=1}^k(-1)^\ell P_\Omega U_{\mu_\ell,\xi_\ell}(x)+ \vp_\ve(x), 
\end{equation}
where the concentration parameters satisfy
 \begin{equation}\label{par_bn}\mu_\ell=t_\ell\epsilon^{\gamma_\ell}\ \hbox{for}\ t_\ell>0\ \hbox{and}\ \gamma_\ell={n-2\over 2(n-4)}\({n-2\over n-6}\)^{\ell-1}-\frac12\end{equation}
and  the concentration points are given by
 \begin{equation}\label{point_bn}\xi_1 \in\Omega\ \hbox{and}\  \xi_\ell=\xi_{\ell-1}+\mu_{\ell-1}z_\ell \hbox{ with}\ z_\ell\in\mathbb R^n,\ \hbox{for any}\ 2\le\ell\le k \end{equation}
 All the arguments developed in the proof of Theorem \ref{maintheo} adapt straightforwardly in this setting with
 \begin{itemize}
 \item the bubbles $W_\ell$ replaced by the  bubbles $P_\Omega U_{\mu_\ell,\xi_\ell}$,
 \item the elements $Z_{\ell,j}$ in \eqref{defZi} replaced by 
 $$Z_{\Omega, \ell,j} := P_{\Omega} \Big(\mu_\ell^{-{n-2\over2}}V_j\(x-\xi_\ell \over\mu_\ell \) \Big) $$
where $V_j$ is defined in \eqref{defVi}, and $K_k$ in \eqref{defKi} is replaced by 
 $$K_{\Omega,k} = \textrm{Span} \Big \{ Z_{\Omega, \ell,j}, 1 \le \ell \le k, 0 \le j \le n \Big \}.$$ 
 \item  and the configuration space $S_k$ defined in \eqref{defA}  replaced by the set 
 \[ S_{\Omega,k} = \Big \{(\xi,z_2,\dots,z_k,t_1,\dots,t_k)\ :\    d(\xi,\partial\Omega)\ge d,\ |z_\ell|\le 1,\  A^{-1}\le t_\ell\le A \Big  \}\]
 for some fixed constants $d >0$ and $A>1$.
 \end{itemize}
 
 More precisely, let for $(\xi,z_2,\dots,z_k,t_1,\dots,t_k) \in S_{\Omega}$ and for $x \in \Omega$:
   \be
  \bal
\bullet\  \Psi_{\Omega,\ve}(x) & =  \mu_{\ell-1}^{1 - \frac{n}{2}} \left( \frac{\mu_{\ell}}{\mu_\ell + |x - \xi_\ell|} \right)^2 + \ve \mu_{\ell-1}^{3 - \frac{n}{2}}  \\
&+ \frac{(\mu_{\ell+1} + |\xi_{\ell+1} - x|)^2}{\mu_{\ell}^2} P_\Omega U_{\mu_{\ell+1}, \xi_{\ell+1}}(x)  \mathds{1}_{\{ |x - \xi_{\ell+1}| \le \mu_{\ell}\}}  \\
&  \textrm{ if } x\in B_\ell \backslash B_{\ell+1}, \quad \textrm{ for } \ell \ge 2, \\
\bullet\  \Psi_{\Omega,\ve}(x) & =   \frac{\mu_{1}^{\frac{n+2}{2}}}{(\mu_1 + |\xi_1 - x|)^2}   + \frac{(\mu_{2} + |\xi_{2} - x|)^2}{\mu_{1}^2} P_{\Omega} U_{\mu_2, \xi_2}(x) \mathds{1}_{\{ |\xi_2 - x| \le \mu_{1}\}} \\
 &\textrm{ if } x\in M \backslash B_{2},
\eal
\ee
where we have let $B_i = B(\xi_i, \sqrt{\mu_i \mu_{i-1}})$ for $i \ge 2$ and $B_1 = \Omega \backslash B_2$.  A careful inspection of the proofs of Proposition \ref{proplin} and Proposition \ref{propnonlin} shows that the analysis carries on here (since the points $\xi_1, \dots, \xi_\ell$ live in compact sets in the interior) and yields the following result:

\begin{prop}  \label{propvpBN}
Let $k \ge 1$ be an integer and let $f(u) = |u|^{2^*-2}u$ for $u \in \RR$. There exist $\ve_1 >0$ and $C >0$ such that the following holds. For any $\ve \in (0, \ve_1)$ and for any $ \big (\xi,z_2,\dots,z_k,t_1,\dots,t_k \big) \in S_{\Omega,k} $, there exists
\[ \vp_\ve:= \vp_\ve \Big(\xi,z_2,\dots,z_k,t_1,\dots,t_k \Big)  \in K_{\Omega,k}^{\perp}\]
which is the unique solution to
\be 
 \Pi_{K_{\Omega, k}^{\perp}}\Bigg[ \sum\limits_{\ell=1}^k(-1)^\ell P_\Omega U_{\mu_\ell,\xi_\ell}+ \vp_{\ve} - (\triangle_\xi - \ve )^{-1}\Bigg( f\Big(  \sum\limits_{\ell=1}^k(-1)^\ell P_\Omega U_{\mu_\ell,\xi_\ell}+ \vp_\ve\Big) \Bigg) \Bigg] =0
 \ee
that both belongs to $K_{\Omega,k}^{\perp}$ and satisfies $|\vp_\ve| \le C  \Psi_{\Omega,\ve}$. In addition, 
\[ \big( \xi,z_2,\dots,z_k,t_1,\dots,t_k \big) \in S_{\Omega,k} \mapsto \vp_\ve \in C^0(\overline{\Omega}) \]
is continuous for any fixed value of $\ve$.
\end{prop}
As in Section \ref{expansionkernel}, this Proposition asserts the existence of $(\nu_{\ell,j})$, $1 \le \ell \le k$ and $0 \le j \le n$ such that 
\ben \label{eqbn}
\bal
 \big( \triangle_\xi- \ve )\Big(  \sum\limits_{\ell=1}^k(-1)^\ell P_\Omega U_{\mu_\ell,\xi_\ell}+ \vp_{\ve} \Big)&  - f \Bigg(   \sum\limits_{\ell=1}^k(-1)^\ell P_\Omega U_{\mu_\ell,\xi_\ell}+ \vp_{\ve} \Bigg) \\
 & = \sum \limits_{\underset{j =0 .. n }{\ell=1..k}}   \nu_{\ell,j} ( \triangle_\xi - \ve )Z_{\Omega, \ell,j}
 \eal 
 \een
 in $\Omega$. Again, the analysis carried on in Section \ref{expansionkernel} remains true here. By \eqref{defmass} and \eqref{defRobin} the analogue of the mass $m(\xi_1)$ is the opposite of Robin's function $-H(\xi,\xi)$. Taking into account that two consecutive bubbles have different sign so that their interaction is {\em negative} instead of {\em positive}, Proposition \ref{propnu1} and Proposition  \ref{propnu1} can be rephrased as follows:
 \begin{prop}  \label{propnu_bn}${}$\\
\begin{itemize}
\item{(i)} If $\ell=1$
then:
$$\bal \Vert \nabla V_0 \Vert^2_{L^2(\RR^n)} & \nu_{1,0}  \big( \xi,z_2,\dots,z_k,t_1,\dots,t_k \big) \\
& = \ve^{1 + 2 \gamma_1} \Big(-\frac{4D_1}{n-2} t_{1}^2 +2D_2 t_1^{n-2} H(\xi,\xi) + \Lambda_{0,1}\Big) \eal $$
and,  for $1 \le j \le n$,
$$  \Vert \nabla V_j \Vert^2_{L^2(\RR^n)}  \nu_{1,j}   \big( \xi,z_2,\dots,z_k,t_1,\dots,t_k \big) = \ve^{1 + 3 \gamma_1} \Big( n D_2 t_1^{n-1}\partial_j H(\xi,\xi) + \Lambda_{j,1} \Big) $$
as $\ve \to 0$.  
\item{(ii)} If  $2 \le \ell \le k$
$$\bal
\Vert \nabla V_0 \Vert^2_{L^2(\RR^n)} &\nu_{\ell,0} \big( \xi,z_2,\dots,z_k,t_1,\dots,t_k \big)  \\
& = \ve^{1 + 2 \gamma_{\ell}} \Big( -\frac{4D_1}{n-2} t_{\ell}^2 +D_3 \left( \frac{t_{\ell}}{t_{\ell-1}} \right)^{\pui}U_0(z_\ell) + \Lambda_{0,\ell} \Big)
\eal 
$$
and, for $1 \le j \le n$,
$$ \bal \Vert \nabla V_j \Vert^2_{L^2(\RR^n)} &  \nu_{\ell,j}  \big( \xi,z_2,\dots,z_k,t_1,\dots,t_k \big) \\
& = \ve^{ \frac{n}{2}(\gamma_{\ell} - \gamma_{\ell-1})} \Bigg( nD_3 \left( \frac{t_{\ell}}{t_{\ell-1}} \right)^{\frac{n}{2}} \partial_j U_0(z_\ell) + \Lambda_{j,\ell}\Bigg)  \eal $$
as $\ve \to 0$.
\end{itemize}
The functions $U_0$ is defined in \eqref{defB0}, the positive constants $D_i$ are the same as in Propositions \ref{propnu1} and \ref{propnuge2}  and the $\Lambda_{j,\ell}$'s  are continuous functions on $S_{\Omega,k}$ which  converge uniformly to $0$ as $\ve \to 0$.  
\end{prop}

\begin{proof}[Proof of Theorem \ref{theoBN}]
With Proposition \ref{propnu_bn} the end of the proof of Theorem \ref{theoBN} follows the same lines than the one of Theorem \ref{maintheo}. Let $\xi_0$ be a non-degenerate interior critical point of the Robin's function $\xi \mapsto H(\xi, \xi)$. Let 
$$ t_{1,0} = \left(  \frac{2D_1}{(n-2)D_2H(\xi_0, \xi_0)} \right)^{\frac{1}{n-4}} $$
and, for $\ell \ge2$,
$$ t_{\ell,0} = \left(\frac{4D_1}{(n-2)D_3} t_{\ell-1,0}^{\pui} \right)^{\frac{2}{n-6}} . $$
It is then easily checked that 
$$ \Big( t_{1,0}, \xi_0, (t_{\ell,0}, 0)_{2 \le \ell \le k} \Big)$$
is a zero of the system 
$$\left\{\begin{aligned}  &-\frac{4D_1}{n-2} t_{1}^2 +2D_2 t_1^{n-2} H(\xi,\xi)=0\\
 & n D_2t_1^{n-1}\partial_j H(\xi,\xi)=0\\
&\left( \bal & - \frac{4D_1}{n-2} t_{\ell}^2 +D_3 \left( \frac{t_{\ell}}{t_{\ell-1}} \right)^{\pui}U_0(z_\ell) \\
&  nD_3 \left( \frac{t_{\ell}}{t_{\ell-1}} \right)^{\frac{n}{2}} \partial_j U_0(z_\ell)\\
\eal \right)_{2 \le \ell \le k}= 0 
\end{aligned}\right.$$
with non-zero degree. It is then stable under $C^0-$perturbations. For any $\ve >0$ this provides us with $ \big( \xi_{\ve},z_{2,\ve},\dots,z_{k,\ve},t_{1,\ve},\dots,t_{k,\ve} \big)$ such that
\[ \nu_{\ell,j}  \big( \xi_{\ve},z_{2,\ve},\dots,z_{k,\ve},t_{1,\ve},\dots,t_{k,\ve} \big) = 0 \]
for all $1 \le \ell \le k$ and $0 \le j \le n$. By \eqref{eqbn}, the corresponding 
\[\vp_\ve = \vp_\ve  \big( \xi_{\ve},z_{2,\ve},\dots,z_{k,\ve},t_{1,\ve},\dots,t_{k,\ve} \big)\] provides us with a solution of \eqref{bn} and concludes the proof of Theorem \ref{theoBN}.
\end{proof}

\medskip

The Robin's function $\xi \mapsto H(\xi, \xi)$ always admits an interior critical point since $H$ is positive and $H(\xi, \xi) \to + \infty$ as $\xi \to \partial \Omega$. And for generic domains $\Omega$, all the critical points of the Robin's function are non-degenerate as proved in Micheletti-Pistoia \cite{mipi2}. Finally, here again, if one considers more general equations of the form
\begin{equation*}
\triangle_\xi u - \ve h u =|u|^{4\over n-2}u \ \hbox{in}\ \Omega,\ u=0\ \hbox{on}\ \partial\Omega
\end{equation*}
for $h \in C^0(\overline{\Omega})$, an analogue of Theorem \ref{theoBN} can be proven where towering phenomena occur at \emph{any} critical point of the Robin's function $\xi \mapsto H(\xi, \xi)$ (provided $h$ is suitably chosen). Proposition \ref{propvpBN} remains indeed true and the details are identical to those developed at the very end of Section \ref{expansionkernel}.

 \appendix

\section{Technical Results}

We keep the notations of Section \ref{notations}. The integer $k\ge 1$ refers to the number of bubbles.

\subsection{Proof of \eqref{psZi} and \eqref{psZik}.} Let $1 \le i < j \le k$ and $0 \le p,q \le n$. We first prove \eqref{psZi}. We write for this, using \eqref{psH1}, \eqref{lapZ0} and \eqref{lapZi} and since $|Z_{i,p}| \lesssim W_i$, that
\[ \bal
\langle Z_{i,p}, Z_{i,q}\rangle & = \int_M \big( \triangle_g + c_n S_g + \ve h) Z_{i,p} Z_{i,q} dv_g \\
& = (2^*-1) \int_M W_i^{2^*-2} Z_{i,p} Z_{i,q} dv_g + O(\int_M W_i^2 dv_g) \\
& = (2^*-1) \int_M W_i^{2^*-2} Z_{i,p} Z_{i,q} dv_g + O(\mu_i^2).
\eal \]
Letting $y = \exp_{\xi_i}^{g_{\xi_i}}(\mu_i x)$ and using the conformal invariance of the conformal laplacian and \eqref{defW} we get:
\[\bal   \int_M W_i^{2^*-2} Z_{i,p} Z_{i,q} dv_g & = \int_{B_{g_{\xi_i}}(\xi_i, r_0)}  W_i^{2^*-2} Z_{i,p} Z_{i,q} dv_g + O(\mu_i^n)  \\
& = \int_{B(0, \frac{r_0}{\mu_i})} U_0^{2^*-2} V_p V_q dx + O(\mu_i^4) \\
& = \int_{\RR^n}  U_0^{2^*-2} V_p V_q dx + O(\mu_i^4) \\
& = \Vert \nabla V_p \Vert_{L^2(\RR^n)}^2 \delta_{pq}+ O(\mu_i^4),
\eal \]
where $r_0$ is as in \eqref{defr0}, which proves \eqref{psZi}.

\medskip

We now prove  \eqref{psZik}. With \eqref{lapZ0} and \eqref{lapZi} we have:
$$ \begin{aligned}
\langle & Z_{i,p}, Z_{j,q} \rangle \\
& = \int_M \Big[ \triangle_gZ_{j,q} + (c_n S_g + \ve h) Z_{j,q} \Big] Z_{i,p} dv_g \\
& = (2^*-1) \int_M W_j^{2^*-2} Z_{j,q}Z_{i,p} dv_g+\int_M \Big( \ve Z_{j,q}Z_{i,p}  + O \big(\mu_j W_j |Z_{i,q}|\big) \Big) dv_g     \\
& = O \left( \left( \frac{\mu_j}{\mu_i} \right)^{\pui} \right) ,
\end{aligned} $$
where we just used the rough estimate $|Z_{i,p}| \lesssim \mu_i^{1 - \frac{n}{2}}$.

\subsection{Additional material.} We prove two inequalities that were used several times throughout the proof. 

\begin{lemme}
Let $i \in \{1, \dots, k\}$, and $x \in M$. We have:
\ben \label{tech1}
\bal
& \textrm{ If } 0 \le p < n-4: \\
& \int_{B_i\backslash B_{i+1}} d_g(x, \cdot)^{2-n} W_i^{2^*-2} \left(\frac{\mu_i}{\theta_i} \right)^p dv_g \\
&\quad \quad \quad  \lesssim \left \{ \bal & \left( \frac{\mu_i}{\theta_i(x)} \right)^{p+2} & \textrm{ if } x \in B_i \\ & \mu_i^{\frac{p+2}{2}} \mu_{i-1}^{\frac{n-4-p}{2}} W_i (x) &  \textrm{ if } x \in M \backslash B_i \eal \right. \\
& \textrm{ If } p > n-4: \\
& \int_{B_i\backslash B_{i+1}} d_g(x, \cdot)^{2-n} W_i^{2^*-2} \left(\frac{\mu_i}{\theta_i} \right)^p dv_g \lesssim \left( \frac{\mu_i}{\theta_i(x)} \right)^{n-2} = \mu_i^\pui W_i(x) , \\
\eal \een
where $\theta_i $ is as in \eqref{defti}.
\end{lemme}

\begin{proof}
Assume first that $d_{g_{\xi_i}}(\xi_i, x) \le 2 \sqrt{\mu_i \mu_{i-1}}$. Then $W_i^{2^*-2} \lesssim \mu_i^2 \theta_i(x)^{-4}$, where $\theta_i$ is as in \eqref{defti}. If $0 \le p < n-4$, the result follows from Giraud's lemma. While if $p > n-4$ letting $y = \exp_{\xi_i}^{g_{\xi_i}}(\mu_i z)$ in the integral  and $\check{x} = \frac{1}{\mu_i} {\exp_{\xi_i}^{g_{\xi_i}}}^{-1}(x)$ yields :
\[ \bal \int_{B_i\backslash B_{i+1}} d_g(x, \cdot)^{2-n} W_i^{2^*-2} \left(\frac{\mu_i}{\theta_i} \right)^p dv_g &\lesssim \int_{\RR^n} |\check{x}-z|^{2-n} (1+|z|)^{-n-2} dz \\
& \lesssim \frac{1}{(1+|\check{x}|)^{n-2}} \lesssim   \left( \frac{\mu_i}{\theta_i(x)} \right)^{n-2}. \eal  \]
  Assume then that $d_{g_{\xi_i}}(\xi_i, x) \ge 2 \sqrt{\mu_{i-1}\mu_i}$ (note that this only makes sense for $ i\ge2$). Then for any $y \in B_i$ one has $d_g(x,y) \gtrsim d_g(x, \xi_i) \gtrsim \theta_i(x)$. If $p < n-4$ we then have
   \[ \bal \int_{B_i\backslash B_{i+1}} d_g(x, \cdot)^{2-n} W_i^{2^*-2} \left(\frac{\mu_i}{\theta_i} \right)^pdv_g & \lesssim \theta_i(x)^{2-n} \int_{B_i} W_i^{2^*-2} \left(\frac{\mu_i}{\theta_i} \right)^p\\
& = \theta_i(x)^{2-n} \mu_i^{p+2} (\mu_i \mu_{i-1})^{\frac{n-4-p}{2}} \\
& =  \mu_i^{\frac{p+2}{2}} \mu_{i-1}^{\frac{n-4-p}{2}} W_i (x), \eal \]
while if $p > n-4$ we have
   \[ \bal \int_{B_i\backslash B_{i+1}} d_g(x, \cdot)^{2-n} W_i^{2^*-2} \left(\frac{\mu_i}{\theta_i} \right)^pdv_g & \lesssim \theta_i(x)^{2-n} \int_{B_i} W_i^{2^*-2} \left(\frac{\mu_i}{\theta_i} \right)^p\\
& = \theta_i(x)^{2-n} \mu_i^{n-2} \int_{\RR^n} (1+|y|)^{-p-4} dy \\
& \lesssim \left( \frac{\mu_i}{\theta_i(x)}\right)^{n-2}. \eal \]
Note finally that, in the intermediate case $\sqrt{\mu_{i-1}\mu_i} \le d_{g_{\xi_i}}(\xi_i, x) \le 2 \sqrt{\mu_{i-1}\mu_i}$ we have 
\[  \left( \frac{\mu_i}{\theta_i(x)} \right)^{p+2} \sim \left( \frac{\mu_i}{\mu_{i-1}}\right)^{\frac{p+2}{2}} \sim  \mu_i^{\frac{p+2}{2}} \mu_{i-1}^{\frac{n-4-p}{2}} W_i (x) \] 
and this holds true for $0 \le p \le n-4$, so that the result follows.
\end{proof}

\begin{lemme}
For any $i, j \in \{1, \dots, k\}$, $i < j$, and any $x \in M$, we have
 \ben \label{tech2}
\bal
& \int_{B_{i} \backslash B_{i+1}} W_i^{2^*-2} W_j d_g(x,\cdot)^{2-n} dv_g \\
&\lesssim   
\left \{
\bal
& \left( \frac{\mu_j}{\mu_{i+1}} \right)^{\pui} \frac{\mu_{i+1}}{\mu_i} \mu_i^{1 - \frac{n}{2}}  
&  \textrm{ if } x \in B_{i+1} \\
&  \left( \frac{\mu_j}{\mu_i} \right)^{\pui} W_i(x)  &  \textrm{ if } x \in  B_i \backslash B_{i+1} \textrm{ and } \theta_j(x) \ge \mu_i  \\
&  \frac{\theta_j(x)^2}{\mu_{i}^2} W_j(x) &  \textrm{ if } x \in B_i \backslash B_{i+1} \textrm{ and } \theta_j(x) \le \mu_i \\
&  \left( \frac{\mu_j}{\mu_i} \right)^{\pui} W_i(x)  &  \textrm{ if } x \in M \backslash B_i .
\eal \right. \\
\eal
\een
\end{lemme}
\begin{proof}

Assume first that $d_{g_{\xi_i}}(\xi_i, x) \ge  \sqrt{\mu_i \mu_{i-1}}$, that is $x \in M \backslash B_i$ (this case is empty if $i=1$). Changing variables in the integral by $y = \exp_{\xi_i}^{g_{\xi_i}}(\mu_i x)$ and using Giraud's lemma yields in this case
 $$\int_{B_{i} \backslash B_{i+1}} W_i^{2^*-2} W_j d_g(x,\cdot)^{2-n} dv_g \lesssim \theta_i(x)^{2-n} \mu_j^{\pui} =  \left( \frac{\mu_j}{\mu_i} \right)^{\pui} W_i(x) .  $$
Assume now that $x \in B_{i+1}$, so in particular $\theta_j(x) \lesssim \sqrt{\mu_{i+1} \mu_i}$ by \eqref{defti}. We let the change of variables 
 $$y = \exp_{\xi_j}^{g_{\xi_j}}(\theta_j(x) z)$$
  in the integral. By \eqref{defxi} one has 
$$\bal
  \int_{B_{i} \backslash B_{i+1}} & W_i^{2^*-2} W_j d_g(x,\cdot)^{2-n} dv_g  \\
  &\lesssim \mu_i^{-2} \mu_j^{\pui} \theta_j(x)^{4-n} \\
  & \times \int_{B_0 \left(2\frac{\sqrt{\mu_{i-1}\mu_i}}{\theta_j(x)} \right) \backslash B_0 \left(\frac12 \frac{\sqrt{\mu_{i+1}\mu_i}}{\theta_j(x)} \right)} |\check{x}_i - z|^{2-n} \Big( \frac{\mu_j}{\theta_j(x)} + |z|\Big) ^{2-n} dz,  
  \eal$$
where we have let $\check{x}_i= \frac{1}{\theta_j(x)} (\exp_{\xi_j}^{g_{\xi_j}})^{-1}(x)$. As is easily checked, the previous integral is always uniformly integrable whatever the value of $\theta_j(x)$ is, and there holds $|\check{x}_i| \le 1$ by definition of $\theta_j$. So that in the end
 $$ \bal \int_{B_{i} \backslash B_{i+1}} W_i^{2^*-2} W_j d_g(x,\cdot)^{2-n} dv_g & \lesssim \mu_i^{-2} \mu_j^{\pui} \theta_j(x)^{4-n} \int_{\mathbb{R}^n \backslash B_0 \left(\frac{\sqrt{\mu_{i+1}\mu_i}}{\theta_j(x)} \right)}  |z|^{4-2n} dz \\
 & \lesssim \mu_i^{-2} \mu_j^{\pui} \theta_j(x)^{4-n}\left( \frac{\theta_j(x)}{\sqrt{\mu_{i+1}\mu_i}}\right)^{n-4} \\
 & \lesssim \left( \frac{\mu_j}{\mu_{i+1}} \right)^{\pui} \frac{\mu_{i+1}}{\mu_i} \mu_i^{1 - \frac{n}{2}}. 
 \eal $$
Assume now that $x \in B_i \backslash B_{i+1}$. The change of variables
 $$y = \exp_{\xi_j}^{g_{\xi_j}}(\mu_i z)$$
 gives 
 \ben \label{controleint}
 \bal 
     \int_{B_{i} \backslash B_{i+1}} & W_i^{2^*-2} W_j d_g(x,\cdot)^{2-n} dv_g \\
  &  \lesssim \mu_j^{\pui} \mu_i^{2-n} \int_{\RR^n \backslash B_0 \big(\frac12 \sqrt{\frac{\mu_{i+1}}{\mu_i}}\big)} (1 + |\check{\xi}-y|)^{-4} |y|^{2-n} |\check{x}-y|^{2-n} dy,
    \eal \een
where we have let $\check{\xi} = \frac{1}{\mu_i} (\exp_{\xi_j}^{g_{\xi_j}})^{-1}(\xi_i)$ and $\check{x} = \frac{1}{\mu_i} (\exp_{\xi_j}^{g_{\xi_j}})^{-1}(x)$. Since $x \not \in B_{i+1}$ we have $\theta_j(x) \gtrsim \sqrt{\mu_{i+1} \mu_i}$ and thus $|\check{x}| \gtrsim \frac{\theta_j(x)}{\mu_i}$. Assume first that $|\check{x}| \gtrsim 1$. Then we also have $\theta_j(x) \gtrsim \theta_i(x)$, so that classical integral comparison results yield
$$  \int_{B_{i} \backslash B_{i+1}}  W_i^{2^*-2} W_j d_g(x,\cdot)^{2-n} dv_g \lesssim \mu_j^{\pui} \mu_i^{2-n} (1 + |\check{x}|)^{2-n} \lesssim  \left( \frac{\mu_j}{\mu_i} \right)^{\pui} W_i(x) . $$
Assume finally that $|\check{x}| \lesssim 1$. A standard Giraud-type argument (see e.g. lemma 7.5 in Hebey \cite{HebeyZLAM}) yields with \eqref{controleint}
$$ \bal
  \int_{B_{i} \backslash B_{i+1}}  W_i^{2^*-2} W_j d_g(x,\cdot)^{2-n} dv_g & \lesssim  \mu_j^{\pui} \mu_i^{2-n} |\check{x}|^{4-n} \\
  & \lesssim  \frac{\theta_j(x)^2}{\mu_{i}^2} W_j(x),
  \eal $$
    which concludes the proof of \eqref{tech2}.
\end{proof}

\bibliographystyle{amsplain}
\bibliography{bibliobis}

\end{document}